\documentclass[12pt,reqno]{amsart}
\usepackage{amsmath,amsthm,amssymb,mathrsfs,stmaryrd,color,mathtools}

\usepackage[all]{xy}
\usepackage{url}
\usepackage{geometry}
\usepackage{tikz-cd}
\usepackage[utf8]{inputenc}
\usepackage[T1]{fontenc}

\usepackage{relsize}
\usepackage[bbgreekl]{mathbbol}
\usepackage{amsfonts}

\DeclareSymbolFontAlphabet{\mathbb}{AMSb}
\DeclareSymbolFontAlphabet{\mathbbl}{bbold}
\newcommand{\Prism}{{\mathlarger{\mathbbl{\Delta}}}}
\usepackage{enumitem}
\usepackage[colorlinks=true,hyperindex, linkcolor=magenta, pagebackref=false, citecolor=cyan,pdfpagelabels]{hyperref}
\usepackage[capitalize]{cleveref}

\newtheorem{thm}{Theorem}[section]
\newtheorem{thm2}{Theorem}

\newtheorem{prop}[thm]{Proposition}

\newtheorem{cor}[thm]{Corollary}
\newtheorem{lem}[thm]{Lemma}

\theoremstyle{definition}
\newtheorem{defi}[thm]{Definition}
\newtheorem{rmk}[thm]{Remark}
\newtheorem{exa}[thm]{Example}
\newtheorem{const}[thm]{Construction}

\numberwithin{equation}{section}

\newcommand{\bb}[1]{\mathbb{#1}}
\newcommand{\cl}[1]{{\mathcal{#1}}}
\newcommand{\scr}[1]{{\mathscr{#1}}}
\newcommand{\msf}[1]{{\mathsf{#1}}}
\newcommand{\mfr}[1]{{\mathfrak{#1}}}
\newcommand{\mrm}[1]{{\mathrm{#1}}}
\newcommand{\mbf}[1]{\mathbf{#1}}

\newcommand{\ov}[1]{{\overline{#1}}}

\newcommand{\wtd}[1]{{\widetilde{#1}}}

\newcommand{\Frob}{{\operatorname{Frob}}}
\newcommand{\Frac}{{\operatorname{Frac}}}

\newcommand{\Img}{{\operatorname{Im}}}
\newcommand{\clos}{{\operatorname{cl}}}

\newcommand{\ints}{{\operatorname{int}}}
\newcommand{\bdd}{{\operatorname{bdd}}}
\newcommand{\op}{{\operatorname{op}}}
\newcommand{\BK}{{\operatorname{BK}}}

\newcommand{\sch}{{\operatorname{sch}}}
\newcommand{\spc}{{\operatorname{sp}}}
\newcommand{\SP}{{\operatorname{SP}}}
\newcommand{\wgc}{{\operatorname{wgc}}}
\newcommand{\Gr}{{\operatorname{Gr}}}
\newcommand{\Spd}{{\operatorname{Spd}}}
\newcommand{\Spec}{{\operatorname{Spec}}}
\newcommand{\Spa}{{\operatorname{Spa}}}
\newcommand{\Spf}{{\operatorname{Spf}}}
\newcommand{\Perf}{{\operatorname{Perf}}}
\newcommand{\SchPerf}{{\operatorname{SchPerf}}}

\newcommand{\Sh}{{\operatorname{Sh}}}
\newcommand{\supp}{{\operatorname{supp}}}
\newcommand{\univ}{{\operatorname{univ}}}

\newcommand{\loc}{{\operatorname{loc}}}

\newcommand{\perf}{{\operatorname{perf}}}

\newcommand{\Ker}{{\operatorname{Ker}}}

\newcommand{\torsion}{{\operatorname{torsion}}}
\newcommand{\id}{{\operatorname{id}}}
\newcommand{\an}{{\operatorname{an}}}

\newcommand{\colim}{\mathop{\operatorname{colim}}}

\newcommand{\ad}{{\operatorname{ad}}}
\newcommand{\red}{{\operatorname{red}}}

\newcommand{\Hom}{{\operatorname{Hom}}}

\newcommand{\Gal}{{\operatorname{Gal}}}

\newcommand{\Sht}{{\operatorname{Sht}}}
\newcommand{\Grp}{{\operatorname{Grp}}}
\newcommand{\Set}{{\operatorname{Set}}}

\newcommand{\Flag}{{\mathscr{F}\!\ell}}
\usepackage[citestyle=alphabetic,bibstyle=alphabetic,backend=bibtex, maxalphanames=10, maxnames=10, url=false, doi=false]{biblatex}
\title{Relative representability and parahoric level structures}
\author{Yuta Takaya}
\address{Graduate School of Mathematical Sciences, The University of Tokyo, 3-8-1 Komaba,
Meguro-ku, Tokyo 153-8914, Japan}
\email{takaya@ms.u-tokyo.ac.jp}
\begin{document}
\begin{abstract}
    We establish a representability criterion of $v$-sheaf theoretic modifications of formal schemes and apply this criterion to moduli spaces of parahoric level structures on local shtukas. In the proof, we introduce nice classes of equivariant profinite perfectoid covers and study geometric quotients of perfectoid formal schemes by profinite groups. As a corollary, we show the local representability of integral models of local Shimura varieties under hyperspecial levels, and study the forgetful morphisms between integral models of Shimura varieties associated with inclusions of parahoric subgroups under hyperspecial levels. 
\end{abstract}

\maketitle
\setcounter{tocdepth}{2}
\tableofcontents

\section{Introduction}

In \cite{Art69} and \cite{Art74}, Artin established algebraicity criteria of functors and stacks over schemes and these criteria play a central role in moduli problems over schemes. In recent years, $p$-adic Hodge theory has developed significantly, largely due to the theory of perfectoid spaces and prismatic cohomology. It turns out that many arithmetically important spaces admit moduli interpretation in perfect schemes or perfectoid spaces, that is, as $v$-sheaves (e.g. \cite{Zhu17}, \cite{SW20}). Building on this development, new $v$-sheaf theoretic moduli problems for local models and integral models of local Shimura varieties were formulated in \cite{SW20} and the one for local models was completely solved in \cite{AGLR22}. For these moduli problems, it is important to know whether $v$-sheaf theoretic moduli problems are representable in formal schemes. In this paper, we will tackle this problem and establish a representability criterion in the style of \cite{Art70}.

In \cite{Art70}, Artin proved that every formal modification of the completion of a locally Noetherian algebraic space is uniquely algebraizable to its algebraic modification. Our first observation is that the conditions of formal modifications can be roughly translated to the $v$-sheaf theory using the language of \cite{Gle24}, where \textit{(pre)kimberlites} are introduced as a $v$-sheaf theoretic generalization of formal schemes. 

Let $R$ be a complete adic ring in which $p$ is topologically nilpotent. As a $v$-sheaf theoretic modification of $\Spf(R)$, we consider a thick prekimberlite $Y$ formally adic over $\Spd(R)$ with $Y^\an \cong \Spd(R)^\an$. Here, a prekimberlite formally adic over $\Spd(R)$ is a $v$-sheaf theoretic counterpart of formal schemes adic over $R$, and the condition $Y^\an \cong \Spd(R)^\an$ roughly amounts to $Y$ being a formal modification of $\Spf(R)$ (see \Cref{prop:formalmod}). Here, we say that $Y$ is thick if the specialization map $\spc_{Y^\an}$ is surjective. Since we do not have control over deperfections of $Y^\red-\Img(\spc_{Y^\an})$, this assumption on $Y$ is essential. 

We prove the representability of $Y$ by additionally assuming the existence of a nice equivariant profinite perfectoid cover $(R_\bullet,\Gamma_\bullet)$ of $R$ and the representability of $Y$ over the perfectoid colimit $R_\infty$. In \Cref{ssec:verygoodcovers} and \Cref{ssec:goodcov}, we introduce the axioms of $(R_\bullet, \Gamma_\bullet)$ and simply call $(R_\bullet, \Gamma_\bullet)$ a (very) good cover of $R$. The construction of good covers is inspired by \cite[Remark 3.11]{BS22}, but the notable difference is the existence of a profinite group action of $\Gamma_\infty$ on $R_\infty$, similarly to the Sen theory. Identifying the complicated axioms of good covers is one essential part of the proof of our representability criterion. The main source of formal schemes admitting good covers is the completion of smooth schemes (see \Cref{rmk:exverygood}). In \Cref{ssc:padicexa}, we discuss how wide the class of $p$-adic formal schemes admitting good covers is. The difficulty in obtaining examples other than $p$-adic smooth formal schemes lies in the condition (unifsec) in \Cref{ssec:goodcov}. This question is strongly related to how wide the class of sousperfectoid affinoids is (see \cite{HK22}) .  

In \Cref{sec:perfdformal}, we extend the definition of $p$-adic perfectoid formal schemes introduced in \cite{RC21} and study the basic properties of perfectoid formal schemes. The affine building block of perfectoid formal schemes is a perfectoid ring with a complete linear topology of a finitely generated ideal containing $p$. This notion is equivalent to perfectoid rings in the sense of Gabber-Romero (see \cite[Definition 16.3.1]{GR18}), but we call it a complete adic perfectoid ring to clarify the terminology. By construction, $R_\infty$ is a complete adic perfectoid ring, and the representability of $Y\times_{\Spd(R)} \Spd(R_\infty)$ is way easier to verify than that of $Y$. 

Now, we state a form of our representability criterion. 

\begin{thm2} \textup{(\Cref{cor:vshfmodif})}
    Let $R$ be a reduced excellent complete adic ring admitting a good cover $(R_\bullet,\Gamma_\bullet)$. Let $Y$ be a thick prekimberlite formally adic over $\Spd(R)$ with $Y^\an \cong \Spd(R)^\an$. Suppose that $Y^\red$ is perfectly of finite type over $R_\red$ and $Y\times_{\Spd(R)} \Spd(R_\infty)$ is representable by a perfectoid formal scheme adic over $R_\infty$. Then, $Y$ is representable by a unique proper formal $R$-scheme $\mfr{Y}$ admitting a maximal good cover with $\mfr{Y}^\an \cong \Spa(R)^\an$. 
\end{thm2}

The notion of maximal good covers is introduced in \Cref{ssec:maxlgoodcov}. The existence of maximal good covers ensures the uniqueness of $\mfr{Y}$ (see \Cref{lem:uniqmapmaxlgood}) and $\mfr{Y}$ is automatically reduced and distinguished (in the sense of \cite{FK18}). The proof is essentially reduced to the affine case proved in \Cref{thm:affrep}. The construction of (affine open formal subschemes of) $\mfr{Y}$ is easy to describe as in \Cref{cor:expaffrep}, and the essential difficulties lie in verifying the desired properties of $\mfr{Y}$. In \Cref{thm:topfin}, we prove a finiteness criterion needed to prove the finiteness of $\mfr{Y}$. In \Cref{ssc:geomquot}, we study geometric quotients of $v$-sheaves by (pro)finite groups, and we use the theory to show that $\mfr{Y}$ represents $Y$. 

Another form of our representability criterion is \Cref{prop:finetrep}, which deals with the case where $R$ is $p$-adic and $Y^\an$ is finite \'{e}tale over $\Spd(R)^\an$. As an arithmetic application, we expect that this criterion enables us to study moduli spaces of level structures on local shtukas. In this paper, we only deal with parahoric level structures on local shtukas. In our forthcoming work, we will study depth-zero integral models of local Shimura varieties by studying depth-zero level structures using \Cref{prop:finetrep}. 

Let $F$ be a non-archimedean local field over $\bb{Q}_p$ and let $O_F$ be the ring of integers of $F$. Let $G$ be a connected reductive group over $F$ and let $\cl{G}'\to \cl{G}$ be a morphism of parahoric group schemes of $G$ over $O_F$. As introduced in \cite[Section 23.1]{SW20}, consider the $v$-stack of local $\cl{G}$-shtukas
\[
    \Sht_{\cl{G}} \to \Spd(O_F). 
\]
For a $v$-sheaf $X$ over $\Spd(O_F)$, we say that a map $X\to \Sht_{\cl{G}}$ over $\Spd(O_F)$ is a local shtuka over $X_{/O_F}$. We define $\Sht_{\cl{G}'}$ similarly. As an application of \Cref{prop:finetrep}, we obtain the following representability of moduli spaces of parahoric level structures. 

\begin{thm2} \textup{(\Cref{thm:replevel})}
    Let $R$ be a reduced excellent $p$-adically complete $O_F$-algebra admitting a good cover and let $\cl{P}$ be a local $\cl{G}$-shtuka over $\Spd(R)_{/O_F}$. The $v$-closure of the generic fiber in $\Sht_{\cl{G}'}\times_{\Sht_{\cl{G}},\cl{P}} \Spd(R)$ is representable by a unique proper $p$-adic formal $R$-scheme $\mfr{Y}$ admitting a maximal good cover with $\mfr{Y}_\eta$ finite \'{e}tale over $\Spf(R)_\eta$. 
\end{thm2}

Thanks to the characterization by Scholze-Weinstein and Pappas-Rapoport, it has direct applications to integral models of local and global Shimura varieties. For local Shimura varieties, we can show the local representability of integral models of local Shimura varieties under hyperspecial levels (see \Cref{thm:localrep}) by using the result at hyperspecial levels (see \cite{Bar22}, \cite{Ito25a}). As a byproduct, we show the representability of $v$-sheaf theoretic local model diagrams, modified from \cite[Section 4.9.3]{PR24}, under hyperspecial levels, but this approach does not enable us to prove the basic properties of local model diagrams (see \Cref{prop:LMDrep}). For global Shimura varieties, we can construct integral models of Shimura varieties under hyperspecial levels as algebraic spaces from those at hyperspecial levels (see \Cref{thm:canint}). We verify that they satisfy the axiom of canonicality in \cite{PR24} and are equal to canonical integral models if they exist. 

\addtocontents{toc}{\protect\setcounter{tocdepth}{0}}

\subsection{The structure of the paper}
In \Cref{sec:adicsp}, we explain the basic properties of (reduced excellent) adic spaces and prove a finiteness criterion needed in the proof of \Cref{thm:affrep}. In \Cref{sec:review-v}, we review the $v$-sheaf theory and study geometric quotients of $v$-sheaves. In \Cref{sec:perfdformal}, we introduce perfectoid formal schemes and show the basic properties. In \Cref{sec:vdil}, we introduce (very) good covers and establish our main representability criterion. In \Cref{sec:locsht}, we study moduli spaces of parahoric level structures on local shtukas as an application of our representability criterion. 

\subsection*{Acknowledgements}
I would like to thank my advisor Yoichi Mieda for his constant support and encouragement. I would also like to thank Kazuhiro Ito and Alex Youcis for helpful discussions, and to the anonymous referee for a careful reading and detailed comments. This work was supported by the WINGS-FMSP program at the Graduate School of Mathematical Sciences, the University of Tokyo and JSPS KAKENHI Grant number JP24KJ0865. 

\subsection*{Notation}
Fix a prime number $p$. All rings are assumed to be commutative. The ring of invariants of a ring $R$ with respect to a $\Gamma$-action is denoted by $R^\Gamma$. The reduction of a ring $R$ is denoted by $R_\red$. The perfection of a scheme $X$ is denoted by $X^\mrm{perf}$. For a scheme $X$ over a ring $A$, its base change to an $A$-algebra $B$ is denoted by $X_B$ or $X\otimes_A B$. 

Following \cite[Section 2.2]{SW13}, the pre-adic space associated with a formal scheme $\mfr{X}$ is denoted by $\mfr{X}^\ad$. The analytic locus of $\mfr{X}^\ad$ is denoted by $\mfr{X}^\an$. The tilting of a perfectoid ring $R$ is denoted by $R^\flat$, and the tilting of a perfectoid Huber pair $(R,R^+)$ is denoted by $(R^\flat,R^{\flat +})$. 

\addtocontents{toc}{\protect\setcounter{tocdepth}{2}}

\section{Preliminaries on adic spaces} \label{sec:adicsp}

\subsection{Generalities}
In this section, we introduce our notation and explain the basic properties of adic spaces. 

An adic ring $R$ is a topological ring endowed with a linear topology of a finitely generated ideal $I\subset R$, where $\{I^n\}_{n \geq 0}$ is a neighborhood basis of $0 \in R$. When $R$ is complete, we say that $R$ is a complete adic ring. If we like to emphasize the topology on $R$, we say that $R$ is a (complete) $I$-adic ring. The set of topologically nilpotent elements in $R$ is denoted by $R^{\circ \circ}$ and $R/R^{\circ \circ}$ is denoted by $R_\red$. 

For an Huber pair $(A,A^+)$, which is usually called a Huber pair, we have a topological space $\Spa(A,A^+)$ and a structural presheaf $\cl{O}_{\Spa(A,A^+)}$ with a subpresheaf $\cl{O}^+_{\Spa(A,A^+)}\subset \cl{O}_{\Spa(A,A^+)}$ introduced in \cite{Hub93} and \cite{Hub94}. When $\cl{O}_{\Spa(A,A^+)}$ is a sheaf, we say that $\Spa(A,A^+)$ is an adic space. For every open subset $U\subset \Spa(A,A^+)$, $\cl{O}_{\Spa(A,A^+)}(U)$ (resp.\ $\cl{O}^+_{\Spa(A,A^+)}$) is simply denoted by $\cl{O}(U)$ (resp.\ $\cl{O}^+(U)$). The set of topologically nilpotent elements in $\cl{O}(U)$ is denoted by $\cl{O}^{\circ \circ}(U)$.  

For every set of elements $s,T_1,\ldots,T_r \in A$ such that $T_1,\ldots, T_r$ generate an open ideal of $A$, we have a rational domain
\[  
    U(\tfrac{T_1,\ldots,T_r}{s}) = \{ \lvert T_i \rvert \leq \lvert s \rvert \neq 0 \: \vert \: 1 \leq i \leq r \} \subset \Spa(A, A^+). 
\]
The set of rational domains forms an open basis of $\Spa(A,A^+)$. There is a Huber pair $(A\langle \tfrac{T_1,\ldots,T_r}{s} \rangle,A\langle \tfrac{T_1,\ldots,T_r}{s} \rangle^+)$ that is initial among Huber pairs $(B, B^+)$ over $(A, A^+)$ such that $\Spa(B, B^+) \to \Spa(A, A^+)$ factors through $U(\tfrac{T_1,\ldots,T_r}{s})$. Note that $\Spa(A\langle\tfrac{T_1,\ldots,T_r}{s} \rangle, A\langle \tfrac{T_1,\ldots,T_r}{s} \rangle^+)$ is homeomorphic to $U(\tfrac{T_1,\ldots,T_r}{s})$. When $s$ is topologically nilpotent, $(A\langle \tfrac{T_1,\ldots,T_r}{s} \rangle,A\langle \tfrac{T_1,\ldots,T_r}{s} \rangle^+)$ is Tate. In this paper, a rational domain $U(\tfrac{T_1,\ldots,T_r}{s})$ is called Tate if $s$ is topologically nilpotent. 

When $(A,A^+)=(R,R)$ for a complete adic ring $R$, $\Spa(R,R)$ is simply denoted by $\Spa(R)$. The set of analytic points in $\Spa(R)$ is a quasicompact open subset of $\Spa(R)$ and denoted by $\Spa(R)^\an$. The analytic locus $\Spa(R)^\an$ is empty if and only if $R$ is discrete (see e.g. \cite[Proposition 7.49 (2)]{Wed12}). 

\begin{lem} \label{lem:ancov}
    Let $R$ be a complete adic ring and let $f_1,\ldots,f_r\in R$ be topologically nilpotent elements. If for every $\vert \cdot \vert \in \Spa(R)^\an$, we have $\lvert f_i \rvert \neq 0$ for some $1\leq i \leq r$, then $f_1,\ldots,f_r$ generate an ideal of definition of $R$. 
\end{lem}
\begin{proof}
    Let $I=(f_1,\ldots,f_r) \subset R$ and let $\overline{I}$ be the closure of $I$. Then, $R/\overline{I}$ is a complete adic ring. By assumption, $\Spa(R/\overline{I})^\an$ is empty, so $R/\overline{I}$ is discrete and $\overline{I}$ is open. Let $J \subset R$ be a finitely generated ideal of definition such that $J \subset \ov{I}$. Then, $I+J \subset \ov{I}$, and as $I+J$ is an open submodule of $R$, we also have $\ov{I} \subset I+J$. In particular, $\ov{I} = I +J$. By the same argument, we have $\ov{I} = I + J^2$. Since $1 + J \subset R^\times$, by applying Nakayama's lemma to $(I + J) / I \subset R / I$, we get $I + J = I$. In particular, $I$ is open. 
\end{proof}

\begin{lem}
    Let $R$ be a complete adic ring. Every Tate rational domain $R(\tfrac{T_{1}, \ldots,T_{r}}{s}) \subset \Spa(R)^\an$ can be taken so that $s,T_{1},\ldots,T_{r} \in R$ generate an ideal of definition. 
\end{lem}
\begin{proof}
    Since $(T_1,\ldots,T_r) \subset R$ is open, we can take a finitely generated ideal of definition $I \subset (T_1,\ldots,T_r)^2$ of $R$. It is easy to see that $R(\tfrac{T_{1}, \ldots,T_{r}}{s}) = R(\tfrac{f_1,\ldots,f_s,sT_{1}, \ldots,sT_{r}}{s^2})$ for any finite set of generators $f_1,\ldots,f_s \in I$. Since $s$ is topologically nilpotent, the ideals
    \[
        I \subset (f_1,\ldots, f_s, sT_1, \ldots, sT_r) \subset (I, s)
    \]
    are all ideals of definition. Thus, the description as $ R(\tfrac{f_1,\ldots,f_s,sT_{1}, \ldots,sT_{r}}{s^2})$ satisfies the claim. 
\end{proof}

\begin{lem} \label{lem:Tatecov}
    Let $R$ be a complete adic ring and let $U\subset \Spa(R)^\an$ be a quasicompact open subset. There is a set of generators $f_1,\ldots,f_n \in R$ of an ideal of definition of $R$ such that $U=\bigcup_{1\leq i \leq r} R(\tfrac{f_1,\ldots,f_n}{f_i})$ for some $1\leq r \leq n$. 
\end{lem}
\begin{proof}
    This can be proved in the same way as \cite[\S 8.4, Lemma 5]{Bos14}. Since $U$ is quasicompact, we can take a finite covering of $U$ by Tate rational domains $R(\tfrac{T_{i,1}, \ldots,T_{i,r_i}}{s_i})$ indexed by $i\in I$. We may assume that $s_i,T_{i,1},\ldots,T_{i,r_i} \in R$ generate an ideal of definition for all $i\in I$. Then, for every $i\in I$, $\Spa(R)^\an$ is covered by $R(\tfrac{s_i,T_{i,1}, \ldots,T_{i,r_i}}{a_i})$ with $a_i \in \{s_i,T_{i,1},\ldots,T_{i,r_i}\}$. 

    Consider the intersection of these coverings of $\Spa(R)^\an$. Let $F \subset R$ be the collection of elements of $R$ of the form $\prod_{i\in I} a_i$ with $a_i \in \{s_i,T_{i,1},\ldots,T_{i,r_i}\}$. Let $F_0\subset F$ be the subset consisting of elements $\prod_{i\in I} a_i$ with $a_i=s_i$ for some $i\in I$. Then, $\Spa(R)^\an$ is covered by $R(\tfrac{F}{f})$ with $f\in F$ and $U$ is covered by $R(\tfrac{F}{f})$ with $f\in F_0$. Thus, we obtain the claim.  
\end{proof}

\subsection{Adic spectra of reduced excellent complete adic rings}
In this section, we explain the basic properties of adic spectra of reduced excellent complete adic rings. 

Let $R$ be a Noetherian complete adic ring. By \cite{Hub94}, $\Spa(R)$ is an adic space. For every Tate rational domain $U(\tfrac{T_1,\ldots,T_r}{s}) \subset \Spa(R)$, the rational localization $R\langle \tfrac{T_1,\ldots,T_r}{s} \rangle$ admits a Noetherian ring of definition, such as the $s$-adic completion of $R[\tfrac{T_1,\ldots,T_r}{s}]$. In particular, $R\langle \tfrac{T_1,\ldots,T_r}{s} \rangle$ is a strongly Noetherian Tate ring. If we assume in addition that $R$ is reduced and excellent, we can prove more properties of $\Spa(R)^\an$. 

\begin{lem} \label{lem:excellentst}
    Let $R$ be a reduced excellent complete adic ring. For every Tate rational domain $U(\tfrac{T_1,\ldots,T_r}{s}) \subset \Spa(R)^\an$, $R\langle \tfrac{T_1,\ldots,T_r}{s} \rangle$ is reduced and uniform. In particular, $\Spa(R)^\an$ is stably uniform. 
\end{lem}
\begin{proof}
    Since $R$ is excellent, $R[\tfrac{T_1,\ldots,T_r}{s}]$ is excellent and its integral closure $R(\tfrac{T_1,\ldots,T_r}{s})^+$ in the reduced ring $R[\tfrac{1}{s}]$ is finite over $R[\tfrac{T_1,\ldots,T_r}{s}]$ (see \cite[Theorem 78]{Mat80}, \cite[Tag 03GH]{stacks-project}). In particular, $R\langle \tfrac{T_1,\ldots,T_r}{s}\rangle^+$ is equal to the $s$-adic completion of $R(\tfrac{T_1,\ldots,T_r}{s})^+$, so it is reduced and bounded (see \cite[Theorem 79]{Mat80}). 
\end{proof}

\begin{prop} \label{prop:subspexcellent}
    Let $R$ be a reduced excellent complete adic ring with an ideal of definition $I\subset R$. The topological ring $\Gamma(\Spa(R)^\an, \cl{O}^+)$ is $I$-adic and finite over $R$. 
\end{prop}
\begin{proof}
    Let $f_1,\ldots, f_n \in R$ be generators of $I$. Let $Y\to \Spec(R)$ be the blowup along $V(I)$. It is an isomorphism outside $V(I)$ and the natural section $\Spec(R)-V(I) \to Y$ is affine. Let $Y^+\to Y$ be the normalization of $Y$ in $\Spec(R)-V(I)$. Since $R$ is reduced and excellent, $Y^+\to Y$ is finite. For each $1\leq i \leq n$, let $V_i\subset Y$ be the open subset where $f_i$ divides $f_j$ for every $1\leq j \leq n$ and let $V_i^+$ be the inverse image of $V_i$ in $Y^+$. Now, $V_i$ and $V_i^+$ is affine and the proof of \Cref{lem:excellentst} shows that the $I$-adic completion of $V_i^+$ is isomorphic to $\Spf(R\langle \tfrac{f_1,\ldots,f_n}{f_i} \rangle^+)$. In particular, $\Gamma(\Spa(R)^\an, \cl{O}^+) \cong \lim_{k \geq 0} \Gamma(Y^+\otimes_R (R/I^k), \cl{O})$ as topological rings. Since $Y^+$ is proper over $R$, it follows from the formal function theorem (see e.g. \cite[Corollaire 4.1.7]{EGA3-1}) that $\Gamma(\Spa(R)^\an, \cl{O}^+)$ is $I$-adic and finite over $R$. 
\end{proof}

\begin{lem} \label{lem:Oooinimg}
    Let $R$ be a reduced excellent complete adic ring with an ideal of definition $I\subset R$. Let $f\in R$ be an arbitrary element and let $R'=R[\tfrac{1}{f}]^\wedge$. If $R\to \Gamma(\Spa(R)^\an, \cl{O}^+)$ is a homeomorphism onto the image and the image contains $\Gamma(\Spa(R)^\an, \cl{O}^{\circ \circ})$, then the same properties hold for $R'$. 
\end{lem}
\begin{proof}
    We keep the notation in the proof of \Cref{prop:subspexcellent}. We have $\Gamma(\Spa(R)^\an, \cl{O}^+) \cong \Gamma(Y^+, \cl{O})$. Let $Y^{+,\red}$ be the reduction of $Y^+\otimes_R (R/I)$. Then, $\Gamma(\Spa(R)^\an, \cl{O}^{\circ \circ})$ is isomorphic to $\Ker(\Gamma(Y^+, \cl{O}) \to \Gamma(Y^{+,\red}, \cl{O}))$. Since Zariski localization preserves normalization, it follows from the same argument that $\Gamma(\Spa(R')^\an, \cl{O}^+) \cong \lim_{k \geq 0} \Gamma(Y^+\otimes_R (R'/I^k), \cl{O})$ as topological rings. By the formal function theorem, we have $\Gamma(\Spa(R')^\an, \cl{O}^+) \cong \Gamma(Y^+, \cl{O})[\tfrac{1}{f}]^\wedge$ where $(-)^\wedge$ denotes the $I$-adic completion. Since $R[\tfrac{1}{f}]$ is Noetherian and $\Gamma(Y^+,\cl{O})$ is finite over $R$, $\Gamma(\Spa(R')^\an, \cl{O}^+) \cong \Gamma(\Spa(R)^\an, \cl{O}^+) \otimes_R R'$. Since $R\to R'$ is flat, we see that $R' \to \Gamma(\Spa(R')^\an, \cl{O}^+)$ is injective. 
    
    It is enough to show that the image contains $\Gamma(\Spa(R')^\an, \cl{O}^{\circ \circ})$. Since $R'=R[\tfrac{1}{f}]^\wedge$, $Y^{+,\red}\otimes_R R'$ is reduced. Thus, we see that $\Gamma(\Spa(R')^\an, \cl{O}^{\circ \circ}) \cong \Ker(\Gamma(Y^+, \cl{O})\otimes_R R' \to \Gamma(Y^{+,\red}\otimes_R R', \cl{O}))\cong \Gamma(\Spa(R)^\an, \cl{O}^{\circ \circ}) \otimes_R R'$ as $R'$ is flat over $R$. Thus, the claim follows since the image of $R \to \Gamma(\Spa(R)^\an, \cl{O}^+)$ contains $\Gamma(\Spa(R)^\an, \cl{O}^{\circ \circ})$. 
\end{proof}

Next, we explain the normality of reduced excellent analytic adic spaces. 

\begin{defi}
    We say that a Tate Huber pair $(A,A^+)$ is reduced (resp.\ normal) and excellent if $A^+$ is reduced (resp.\ normal), excellent and bounded. 
\end{defi}

Note that $A^+$ is also normal if $(A,A^+)$ is normal and excellent. 

\begin{lem}
    Let $(A,A^+)$ be a reduced excellent Tate Huber pair. Any rational localization $(B,B^+)$ of $(A,A^+)$ is reduced and excellent. 
\end{lem}
\begin{proof}
    We may write $(B,B^+) = (A\langle \tfrac{T_1,\ldots,T_r}{s} \rangle, A\langle \tfrac{T_1,\ldots,T_r}{s} \rangle^+)$ for $s,T_1,\ldots,T_r \in A^{\circ \circ}$. As in \Cref{lem:excellentst}, let $A(\tfrac{T_1,\ldots,T_r}{s})^+$ be the integral closure of $A^+[\tfrac{T_1,\ldots,T_r}{s}]$ in $A[\tfrac{1}{s}]$. Since $A^+$ is excellent and $A[\tfrac{1}{s}]$ is reduced and of finite type over $A^+$, $A(\tfrac{T_1,\ldots,T_r}{s})^+$ is finite over $A^+[\tfrac{T_1,\ldots,T_r}{s}]$. As in \Cref{lem:excellentst}, we see that $B^+$ is reduced, excellent and bounded. 
\end{proof}

\begin{lem} \label{lem:normalanloc}
    Let $(A,A^+)$ be a reduced excellent Tate Huber pair. Let $\{(B_i,B_i^+)\}_{i\in I}$ be a finite set of rational localizations of $(A,A^+)$ such that $\Spa(A,A^+) = \bigcup_{i\in I} \Spa(B_i,B_i^+)$. Then, $(A,A^+)$ is normal if and only if $(B_i,B_i^+)$ is normal for every $i\in I$. 
\end{lem}
\begin{proof}
    By \cite[Lemma 1.7.6]{Hub96}, $A \to \prod_{i\in I} B_i$ is faithfully flat, so $A$ is normal if $B_i$ is normal for every $i\in I$ by \cite[Corollary 21.3]{Mat80}. On the other hand, suppose that $A$ is normal and let $(B_i,B_i^+) = (A\langle \tfrac{T_1,\ldots,T_r}{s} \rangle, A\langle \tfrac{T_1,\ldots,T_r}{s} \rangle^+)$ for $s,T_1,\ldots,T_r \in A^{\circ \circ}$. Since $A[\tfrac{1}{s}]$ is normal, $A(\tfrac{T_1,\ldots,T_r}{s})^+$ is normal and excellent. Thus, $A\langle \tfrac{T_1,\ldots,T_r}{s} \rangle^+$ is normal by \cite[Theorem 79]{Mat80}, so $B_i$ is normal. 
\end{proof}

\begin{cor} \label{cor:normaletloc}
    Let $(A,A^+)$ be a reduced excellent Tate Huber pair. Let $(B,B^+)$ be a Tate Huber pair \'{e}tale over $(A,A^+)$. Then, $B$ is reduced, and $B$ is normal if $A$ is normal. 
\end{cor}
\begin{proof}
    By \Cref{lem:normalanloc} and \cite[Lemma 2.2.8]{Hub96}, we may localize $(B,B^+)$ so that $(A,A^+)\to (B,B^+)$ is decomposed into a finite \'{e}tale homomorphism and a rational localization. By \cite[Expose I, Theorem 9.5 (i)]{SGA1}, \'{e}tale homomorphisms preserve reducedness and normality, so the claim follows. 
\end{proof}

We say that an analytic adic space is reduced (resp. normal) and excellent if it is analytic locally isomorphic to the adic spectrum of a reduced (resp. normal) and excellent Tate Huber pair. By the proof of \Cref{lem:excellentst}, $\Spa(R)^\an$ is reduced and excellent for every reduced excellent complete adic ring $R$. 

\subsection{Finiteness criterion}
The aim of this section is to prove the following finiteness criterion. It is a key ingredient in our affine representability criterion (see \Cref{thm:affrep}). 

For a complete adic ring $R$ with an ideal of definition $I \subset R$, we say that a complete $I$-adic $R$-algebra $S$ is topologically of finite type if $S/IS$ is of finite type over $R/I$. 

\begin{prop} \label{thm:topfin}
    Let $R$ be a reduced excellent complete adic ring with an ideal of definition $p \in I\subset R$. Let $U\subset \Spa(R)^\an$ be a quasicompact open subset and let $S\subset \cl{O}^+(U)$ be an $I$-adically complete $R$-subalgebra. Suppose that the following conditions hold. 
    \begin{enumerate}
        \item The subspace topology on $S$ from $\cl{O}^+(U)$ is the $I$-adic topology. 
        \item $\Spa(S)^\an \to \Spa(R)^\an$ is a homeomorphism onto $U$. 
        \item $(S_\red)^\perf$ is perfectly of finite type over $(R_\red)^\perf$. 
    \end{enumerate}
    Then, $S$ is reduced, excellent and topologically of finite type over $R$, and $\Spa(S)^\an \to U$ is an isomorphism of adic spaces. 
\end{prop}

\begin{rmk}
    By \Cref{prop:subspexcellent}, the condition (1) is necessary for the claim. The merit of this criterion is that the conditions (2) and (3) can be checked from the $v$-sheafified map $\Spd(S) \to \Spd(R)$ by \cite[Lemma 15.7]{Sch17} and \cite[Proposition 18.3.1]{SW20}. 
\end{rmk} 

\begin{proof}

By \Cref{lem:Tatecov}, we can take a generator $f_1,\ldots,f_n\in R$ of an ideal of definition such that $U$ is the union of $R(\tfrac{f_1,\ldots,f_n}{f_i})$ with $1\leq i \leq r$ for some $1\leq r \leq n$. We may assume that $I=(f_1,\ldots,f_n)$. We keep the notation from the proof of \Cref{prop:subspexcellent}. Let $Y_0 \subset Y$ (resp.\ $Y_0^+ \subset Y^+$) be the union of $V_i$ (resp.\ $V_i^+$) with $1\leq i \leq r$. Then, the $I$-adic completion of $Y_0^+$ is a formal model of $U$. 

By the condition (2), $\lvert f_i \rvert \neq 0$ for some $1 \leq i \leq r$ for every $\vert \cdot \vert \in \Spa(S)^\an$. Thus, $f_1,\ldots,f_r$ generate an ideal of definition of $S$ by \Cref{lem:ancov}. Let $Z\to \Spec(S)$ be the blowup along $V(f_1,\ldots,f_r)$ and let $Z^+\to Z$ be the normalization in $\Spec(S) - V(f_1,\ldots,f_r) \to Z$. For each $1\leq i \leq r$, let $W_i\subset Z$ be the open subset where $f_i$ divides $f_j$ for every $1\leq j \leq r$ and let $W_i^+$ be the inverse image of $W_i$ in $Z^+$. 

\begin{lem} \label{lem:constrVW}
    There is a natural morphism $W_i^+\to V_i^+$ for every $1\leq i \leq r$. 
\end{lem}
\begin{proof}
    By construction, $W_i$ and $W_i^+$ is the spectrum of $S[\tfrac{f_1,\ldots,f_r}{f_i}]$ and $S(\tfrac{f_1,\ldots,f_r}{f_i})^+$, respectively. It is enough to show that $f_j/f_i \in S(\tfrac{f_1,\ldots,f_r}{f_i})^+$ for every $1 \leq j \leq n$. Since the rational domain $U(\tfrac{f_1,\ldots,f_r}{f_i}) \subset \Spa(S)^\an$ maps to $U(\tfrac{f_1,\ldots,f_n}{f_i}) \subset \Spa(R)^\an$ by the condition (2), we have a continuous homomorphism $R\langle \tfrac{f_1,\ldots,f_n}{f_i} \rangle\to S\langle  \tfrac{f_1,\ldots,f_r}{f_i} \rangle$. Thus, $f_j/f_i \in  S\langle  \tfrac{f_1,\ldots,f_r}{f_i} \rangle^+$. Since $S(\tfrac{f_1,\ldots,f_r}{f_i})^{+, \wedge}$ is automatically normal inside $S(\tfrac{f_1,\ldots,f_r}{f_i})^{+, \wedge}[\tfrac{1}{f_i}]$, we get a natural map
    \[
        S[\tfrac{f_1,\ldots,f_r}{f_i}]^\wedge \hookrightarrow S\langle  \tfrac{f_1,\ldots,f_r}{f_i} \rangle^+ \to S(\tfrac{f_1,\ldots,f_r}{f_i})^{+, \wedge}
    \]
    by passing to the $f_i$-adic completion $(-)^\wedge$ and the normalization. Thus, $f_j/f_i \in S(\tfrac{f_1,\ldots,f_r}{f_i})^{+, \wedge}$. 
    
    Let $\pi \colon S(\tfrac{f_1,\ldots,f_r}{f_i})^+ \to S(\tfrac{f_1,\ldots,f_r}{f_i})^{+, \wedge}$. Since $f_i$ is invertible in $S(\tfrac{f_1,\ldots,f_r}{f_i})^+$, we have 
    \[
        f_j \in \pi^{-1}(f_i \cdot S(\tfrac{f_1,\ldots,f_r}{f_i})^{+, \wedge}) = f_i \cdot S(\tfrac{f_1,\ldots,f_r}{f_i})^{+}. 
    \]
    Thus, we get $f_j/f_i \in S(\tfrac{f_1,\ldots,f_r}{f_i})^+$. 
\end{proof}

By gluing these morphisms, we get an affine morphism $Z^+\to Y^+_0$. Let $Z^+_0$ be the scheme-theoretic image of $Z^+ \to Z\times Y^+_0$. Since $Z^+$ is integral over $Z$ and $Y^+$ is of finite type over $R$, $Z_0^+$ is finite over $Z$. 

Let $\mfr{Y}_0^+$ (resp.\ $\mfr{Z}_0^+$) be the $(f_1,\ldots,f_r)$-adic completion of $Y_0^+$ (resp.\ $Z_0^+$). Since $S\subset \cl{O}^+(U)$, we have a morphism $\mfr{Y}_0^+ \to \Spf(S)$. We have the following commutative diagram. 
\begin{center}
    \begin{tikzcd}
        \mfr{Z}_0^+ \ar[r] \ar[rd] & \mfr{Y}_0^+ \ar[d] \\
        & \Spf(S). 
    \end{tikzcd}
\end{center}

\begin{lem}
    The underlying map $\lvert \mfr{Z}_0^+ \rvert \to \lvert \mfr{Y}_0^+ \rvert$ is surjective. 
\end{lem}
\begin{proof}
    By construction, $Y_0^+$ is $I$-torsion free, so the specialization map of $\mfr{Y}_0^+$ is surjective (see \cite[Proposition 3.1.5]{FK18}). Since the specialization map is functorial, it is enough to show that $(\mfr{Z}_0^+)^\an \to (\mfr{Y}_0^+)^\an$ is surjective (see \Cref{ssc:kimber} for the specialization maps of formal schemes over $\bb{Z}_p$). Since $Y_0^+$ (resp.\ $Z_0^+$) is finite over $Y_0$ (resp.\ $Z$), it follows from the construction that $(\mfr{Y}_0^+)^\an \cong \mfr{Y}_0^\an \cong U$ and $(\mfr{Z}_0^+)^\an \cong \mfr{Z}^\an \cong \Spa(S)^\an$. Thus, the claim follows from the condition (2). 
\end{proof}

Since $Z_0^+$ is proper over $\Spec(S)$, we see that $\mfr{Y}_0^+\to \Spf(S)$ is universally closed. 

Now, since $(S_\red)^\perf$ is perfectly of finite type over $(R_\red)^\perf$, we can take a homomorphism $P=R[T_1,\ldots,T_m] \to S$ such that $\Spec(S/IS) \to \Spec(P/IP)$ is universally closed. Then, for every $k\geq 0$, the composition $Y_0^+ \otimes_R (R/I^k) \to \Spec(S/I^kS) \to \Spec(P/I^kP)$ is proper. Since $R$ is Noetherian, it follows that $\Gamma(Y_0^+ \otimes_R (R/I^k), \cl{O})$ is finite over $P$. 

Let $S_k \subset \Gamma(Y_0^+ \otimes_R (R/I^k), \cl{O})$ be the image of $S$. Since $P$ is Noetherian, $S_k$ is also finite over $P$. Let $I_k = \Ker(S \twoheadrightarrow S_k)$. The condition (1) means that the subspace topology on $S$, which is defined by the system $\{I_k\}$, is equal to the $I$-adic topology. Thus, for sufficiently large $k$, we have $I_k \subset IS$, so $S/IS$ is finite over $P$. In particular, $S$ is topologically of finite type over $R$. 

By \Cref{lem:excellentst}, $\cl{O}^+(U)$ is reduced, so $S$ is reduced and excellent. In particular, $\Spa(S)^\an$ is an adic space. Since $(\mfr{Y}_0^+)^\an\cong U$, the analytification of $\mfr{Y}_0^+\to \Spf(S)$ gives a right inverse to $\Spa(S)^\an \to U$. Since $(\mfr{Z}_0^+)^\an\cong \Spa(S)^\an$, it is also a left inverse. Thus, we get $\Spa(S)^\an \cong U$. 

\end{proof}

\subsection{Formal modifications}

In this section, we explain an interpretation of formal modifications (see \cite[Tag 0GDK]{stacks-project}) of reduced excellent formal schemes in terms of adic spaces. The axiom of formal modifications consists of properness, rig-\'{e}taleness (see \cite[Tag 0ALP]{stacks-project}) and rig-surjectivity (see \cite[Tag 0AQQ]{stacks-project}). First, we treat properness and rig-surjectivity. Recall that a morphism of formal schemes is proper if it is topologically of finite type, separated and universally closed. We say that a formal scheme is thick if its specialization map is surjective (see \Cref{ssc:kimber}).  

\begin{lem} \label{lem:dilatationproper}
    Let $R$ be a reduced excellent complete adic ring with an ideal of definition $I\subset R$. Let $\mfr{X}$ be a thick separated formal scheme topologically of finite type over $R$. If $\mfr{X}^\an \to \Spa(R)^\an$ is an isomorphism, then $\mfr{X}$ is proper and rig-surjective over $\Spf(R)$ and $\mfr{X}\to \mfr{X} \times_{\Spf(R)} \mfr{X}$ is rig-surjective. 
\end{lem}

\begin{proof}
    Let $\{\mfr{U}_i\}_{i\in I}$ be a finite affine covering of $\mfr{X}$. As in \Cref{lem:Tatecov}, we can take a set of generators $f_1,\ldots,f_n\in R$ of an ideal of definition such that for every $1\leq j \leq n$, $R(\tfrac{f_1,\ldots,f_n}{f_j})$ is contained in $\mfr{U}_i^\an$ for some $i\in I$. 

    We repeat the construction in the proof of \Cref{prop:subspexcellent} for this choice of $f_1,\ldots,f_n$. For each $1\leq j \leq n$, let $i_j\in I$ be an element such that $R(\tfrac{f_1,\ldots,f_n}{f_j})\subset \mfr{U}_{i_j}^\an$. Let $\mfr{Y}^+$ (resp.\ $\mfr{V}_j^+$) be the $I$-adic completion of $Y^+$ (resp.\ $V_j^+$). By the property of $\mfr{V}_j^+$, this inclusion extends to $\mfr{V}_j^+ \to \mfr{U}_{i_j}$ and the composition $\mfr{V}_j^+ \to \mfr{U}_{i_j} \hookrightarrow \mfr{X}$ is independent of the choice of $i_j$. Thus, by gluing the morphisms for all $1 \leq j \leq n$, we get a morphism $\mfr{Y}^+\to \mfr{X}$. Since $\mfr{X}$ is thick and $(\mfr{Y}^+)^\an \cong \Spa(R)^\an$, we see by looking at the specialization maps that $\lvert \mfr{Y}^+ \rvert\to \lvert \mfr{X} \rvert$ is surjective. Since $Y^+$ is proper over $R$, it follows that $\mfr{X}$ is proper. 

    Next, we prove rig-surjectivity. It is enough to show that for every adic morphism $x_0\colon \Spf(V) \to \Spf(R)$ with $V$ a complete discrete valuation ring, there is a unique lift $\Spf(V)\to \mfr{X}$. Let $K=\Frac(V)$. Since $x_0$ is adic, we have $x_0^\an \colon \Spa(K,V) \to \Spa(R)^\an \cong \mfr{X}^\an$. Thus, $x_0^\an$ factors through $\mfr{U}_i^\an$ for some $i\in I$. Since $\mfr{U}_i$ is affine, $x_0^\an$ uniquely lifts to $\Spf(V)\to \mfr{U}_i$ and we get a desired lift of $x_0$. It is easy to see that such a lift is independent of the choice of $\mfr{U}_i$, so we get the uniqueness of a lift of $x_0$. 
\end{proof}

Next, we treat rig-\'{e}taleness. Here, we can only deal with the monogeneous case due to essential technical difficulty. 

\begin{lem} \label{lem:rigetan}
    Let $R$ be a Noetherian complete adic ring with the $\varpi$-adic topology for $\varpi \in R$. Let $S$ be a $\varpi$-torsion free complete adic $R$-algebra topologically of finite type. Then, $\Spa(S)^\an \to \Spa(R)^\an$ is \'{e}tale if and only if $R\to S$ is rig-\'{e}tale. 
\end{lem}
\begin{proof}
    First, suppose that $R\to S$ is rig-\'{e}tale. By \cite[Tag 0AKG]{stacks-project}, there is a finite type $R$-algebra $S_0$ such that $S_0[\tfrac{1}{\varpi}]$ is \'{e}tale over $R[\tfrac{1}{\varpi}]$ and $(S_0)^\wedge \cong S$. It follows from \cite[Corollary 1.7.3 (iii)]{Hub96} that $\Spa(S)^\an \to \Spa(R)^\an$ is \'{e}tale. 
    
    Next, suppose that $\Spa(S)^\an \to \Spa(R)^\an$ is \'{e}tale. By \cite[Corollary 1.7.3 (iii)]{Hub96}, there is a finite type $R$-algebra $S_0$ such that $S_0[\tfrac{1}{\varpi}]$ is \'{e}tale over $R[\tfrac{1}{\varpi}]$ and $\Gamma(\Spa(S)^\an, \cl{O}^+)$ is isomorphic to the integral closure of $(S_0)^\wedge$ in $(S_0)^\wedge[\tfrac{1}{\varpi}]$. Since $S$ is topologically of finite type over $R$ and is a subring of $\Gamma(\Spa(S)^\an, \cl{O}^+)$ as $S$ is $\varpi$-torsion free, we may enlarge $S_0$ in $S_0[\tfrac{1}{\varpi}]$ so that $(S_0)^\wedge$ contains $S$. We will show that there is a finite type $R$-subalgebra $S_1 \subset S_0$ such that $(S_1)^\wedge \cong S$. 

    Since $\Gamma(\Spa(S)^\an, \cl{O}^+)$ is also the integral closure of $S$ in $S[\tfrac{1}{\varpi}]$, $(S_0)^\wedge$ is finite over $S$. In particular, for sufficiently large $n\geq 0$, we have $\varpi^n (S_0)^\wedge \subset S$. Let $A\subset S_0/\varpi^n S_0$ be the image of $S \to (S_0)^\wedge \to S_0/\varpi^n S_0$ and let $S_1 \subset S_0$ be the inverse image of $A$. Then, $S_0/\varpi^n S_0$ is finite over $A$, so there is a finite $S_1$-submodule $M \subset S_0$ such that $S_0 = M + \varpi^n S_0$. Since $\varpi^n S_0 \subset S_1$, $S_0 = M + S_1$, so $S_0$ is finite over $S_1$. Since $S_0$ is of finite type over $R$ and $R$ is Noetherian, $S_1$ is of finite type over $R$. By construction, we see that $(S_1)^\wedge \subset (S_0)^\wedge$ is equal to $S$. Thus, $R\to S$ is rig-\'{e}tale by \cite[Tag 0AKG]{stacks-project}. 
\end{proof}

The following lemmas are needed to compare conditions on rig-points with those on analytic points. 

\begin{lem} \label{lem:existencerigpoint}
    Let $\mfr{X}$ be a Noetherian formal scheme. For every nonempty open subset $U\subset \mfr{X}^\an$, there is a complete discrete valuation ring $V$ with an adic morphism $\Spf(V) \to \mfr{X}$ such that $\Spa(K,V) \to \mfr{X}^\an$ factors through $U$ where $K=\Frac(V)$. 
\end{lem}
\begin{proof}
    We may assume that $\mfr{X}$ is affine. Let $\mfr{X}=\Spf(R)$. Since $U$ is nonempty, there is a set of generators $f_1,\ldots,f_n \in R$ of an ideal of definition such that $R(\tfrac{f_1,\ldots,f_n}{f_1})$ is nonempty and contained in $U$ (cf. \Cref{lem:Tatecov}). By replacing $R$ by $R[\tfrac{f_1,\ldots,f_n}{f_1}]^\wedge$, we may assume that $U=\Spa(R)^\an$. It is enough to show that there is an adic morphism $\Spf(V) \to \Spf(R)$ with $V$ a complete discrete valuation ring. 
    
    Since $R$ is $I$-adically complete, there is a point $x\in V(I) \subset \Spec(R)$ that lies in the closure of $\Spec(R)-V(I) \hookrightarrow \Spec(R)$. By applying \cite[Tag 0CM2]{stacks-project} to $\Spec(R)-V(I) \hookrightarrow \Spec(R)$, we see that there is a discrete valuation $R$-algebra $V$ such that $\Spec(V) \to \Spec(R)$ sends the closed point to $x$ and the generic point to a point outside $V(I)$. Then, $\Spf(V^\wedge) \to \Spf(R)$ is an adic morphism where $V^\wedge$ is the completion of $V$. 
\end{proof}

\begin{lem} \label{lem:complementrigpoint}
    Let $\mfr{X}$ be a Noetherian formal scheme. For every quasicompact open immersion $U\hookrightarrow \mfr{X}^\an$ with $U\neq \mfr{X}^\an$, there is a complete discrete valuation ring $V$ with an adic morphism $\Spf(V) \to \mfr{X}$ such that $\Spa(K,V) \to \mfr{X}^\an$ does not factor through $U$ where $K=\Frac(V)$. 
\end{lem}
\begin{proof}
    We may assume that $\mfr{X}$ is affine. Let $\mfr{X}=\Spf(R)$. By \Cref{lem:Tatecov}, there is a set of generators $f_1,\ldots,f_n \in R$ of an ideal of definition such that $U$ is the union of $R(\tfrac{f_1,\ldots,f_n}{f_i})$ with $1\leq i \leq r$ for some $1\leq r \leq n$. Let $Y \to \Spec(R)$ be the blowup along $V(f_1,\ldots,f_n)$ and let $V_i\subset Y$ be the open subset where $f_i$ divides $f_j$ for every $1\leq j \leq n$. Let $Y_0 \subset Y$ be the union of $V_i$ for all $1\leq i \leq r$. Let $\mfr{Y}$ (resp. $\mfr{Y}_0$, $\mfr{V}_i$) be the $(f_1,\ldots,f_n)$-adic completion of $Y$ (resp. $Y_0$, $V_i$). We have $\mfr{V}_i^\an \cong R(\tfrac{f_1,\ldots,f_n}{f_i})$ and $U \cong \mfr{Y}_0^\an$. Since $U\neq \mfr{X}^\an$, $Y_0 \neq Y$. Since $Y$ is proper over $R$, the image of $Y-Y_0 \to \Spec(R)$ is a nonempty closed subset, so there is a point $y\in Y-Y_0$ whose image in $\Spec(R)$ lies in $V(f_1,\ldots,f_n)$. As in the proof of \Cref{lem:existencerigpoint}, there is an adic morphism $\Spf(V)\to \mfr{Y}$ such that the closed point maps to $y$. By looking at the specialization map, we see that $\Spa(K,V) \to \mfr{Y}^\an \cong \Spa(R)^\an$ does not factor through $U$. 
\end{proof}

\begin{prop} \label{prop:formalmod}
    Let $R$ be a reduced excellent complete adic ring with the $\varpi$-adic topology for $\varpi \in R$. Let $\mfr{X}$ be a thick separated formal scheme topologically of finite type over $R$. Then, $\mfr{X} \to \Spf(R)$ is a formal modification if and only if $\mfr{X}^\an \to \Spa(R)^\an$ is an isomorphism. 
\end{prop}
\begin{proof}
    First, suppose that $\mfr{X}^\an \to \Spa(R)^\an$ is an isomorphism. Then, $\mfr{X}\to \Spf(R)$ is proper, rig-\'{e}tale and rig-surjective by \Cref{lem:dilatationproper} and \Cref{lem:rigetan}. Moreover, $\mfr{X}\to \mfr{X}\times_{\Spf(R)} \mfr{X}$ is rig-surjective by \Cref{lem:dilatationproper}. Thus, $\mfr{X}\to \Spf(R)$ is a formal modification. 

    Next, suppose that $\mfr{X}\to \Spf(R)$ is a formal modification. By \Cref{lem:rigetan}, $\mfr{X}^\an \to \Spf(R)^\an$ is \'{e}tale. By \cite[Proposition 1.6.8]{Hub96}, the diagonal $\mfr{X}^\an \to \mfr{X}^\an \times_{\Spf(R)^\an} \mfr{X}^\an$ is an open immersion. Since $\mfr{X}$ is separated, $\mfr{X}^\diamond$ is separated (see \cite[Proposition 4.17]{Gle24}) and $\lvert \mfr{X}^\an \rvert \to \lvert \mfr{X}^\an \times_{\Spf(R)^\an} \mfr{X}^\an \rvert$ is a closed immersion by \cite[Lemma 15.6]{Sch17}. Thus, $\mfr{X}^\an \to \mfr{X}^\an \times_{\Spf(R)^\an} \mfr{X}^\an$ is a closed and open immersion. By \Cref{lem:existencerigpoint}, the rig-surjectivity of $\mfr{X}\to \mfr{X}\times_{\Spf(R)} \mfr{X}$ implies $\mfr{X}^\an \to \mfr{X}^\an\times_{\Spa(R)^\an} \mfr{X}^\an$ is an isomorphism. Since $\mfr{X}^\an \to \Spa(R)^\an$ is \'{e}tale, it follows that $\mfr{X}^\an \to \Spa(R)^\an $ is an open immersion. Moreover, since $\mfr{X}$ is topologically of finite type over $R$, $\mfr{X}^\an \to \Spa(R)^\an$ is a quasicompact open immersion. Thus, since $\mfr{X}\to \Spf(R)$ is rig-surjective, it follows from \Cref{lem:complementrigpoint} that $\mfr{X}^\an \to \Spa(R)^\an $ is an isomorphism. 
\end{proof}

\section{The theory of $v$-sheaves and kimberlites}\label{sec:review-v}

\subsection{Basic facts on $v$-sheaves} \label{ssec:basicvsh}

In this section, we review the theory of $v$-sheaves developed in \cite{Sch17} and \cite{AGLR22}. 

Let $\Perf$ be the category of perfectoid spaces over $\bb{F}_p$. The $v$-topology is a Grothendieck topology on $\Perf$ introduced in \cite[Definition 8.1]{Sch17}. A sheaf (resp.\ stack) on $\Perf$ with the $v$-topology is called a $v$-sheaf (resp.\ $v$-stack). Via the Yoneda embedding, perfectoid spaces can be regarded as $v$-sheaves. 

A $v$-sheaf $X$ is small if there is a surjection $Y\to X$ of $v$-sheaves from a perfectoid space $Y$. Moreover, $X$ is quasicompact if we can take $Y$ to be quasicompact. Then, quasicompact maps and quasiseparated maps of $v$-sheaves can be defined (see \cite[Section 8]{Sch17}). 

A map of $v$-sheaves is said to be qcqs if it is quasicompact and quasiseparated. An important property of $v$-sheaves is that injectivity and surjectivity of qcqs maps can be checked on geometric points (see \cite[Lemma 12.11, Proposition 12.15]{Sch17}). The underlying space of a $v$-sheaf $X$ is denoted by $\lvert X \rvert$ (\cite[Proposition 12.7]{Sch17}). For a topological space $T$, the functor sending $S\in \Perf$ to the set of continuous maps from $\lvert S \rvert$ to $T$ is a $v$-sheaf and denoted by $\underline{T}$ (\cite[p.53]{Sch17}). For example, these properties are used in the following way. 

\begin{lem}\label{lem:eqqs}
    Let $S$ be a $v$-sheaf and let $X$ and $Y$ be $v$-sheaves over $S$. Let $f, g\colon X\to Y$ be two maps of $v$-sheaves over $S$. If $X$ is small, $Y$ is  quasiseparated over $S$, and $f(x)=g(x)$ for every $x\in X(C,C^+)$ with $C$ an algebraically closed perfectoid field with an open and bounded valuation subring $C^+\subset C$, we have $f=g$.  
\end{lem}
\begin{proof}
    The diagonal $\Delta_{Y/S}\colon Y\to Y\times_S Y$ is a quasicompact injection. It is enough to show that $f\times g\colon X\to Y\times_S Y$ factors through $\Delta_{Y/S}$. Let $Z=X\times_{Y\times_S Y, \Delta_{Y/S}} Y$. Since the map $Z\to X$ is quasicompact and $X$ is small, $Z$ is small. For every $x\in X(C,C^+)$ as above, we have $f(x)=g(x)$, so $x$ lies in $Z$. In particular, $\lvert Z \rvert \to \lvert X \rvert$ is surjective, so the claim follows from \cite[Lemma 12.11]{Sch17}. 
\end{proof}

Let $X$ be a small $v$-sheaf and let $Y$ be a subsheaf of $X$. The intersection of all closed subsheaves of $X$ containing $Y$ is called the $v$-closure of $Y$ in $X$ and denoted by $Y^\clos$ (see \cite[Definition 2.2]{AGLR22}). Every closed subsheaf of a small $v$-sheaf $X$ is of the form $\underline{S} \times_{\underline{\lvert X \rvert}} X$ with $S \subset \lvert X \rvert$ a weakly generalizing closed subspace (see \cite[Definition 2.3, Lemma 2.7]{AGLR22}). Moreover, for a subsheaf $Y \subset X$, we have $Y^\clos = \underline{\lvert Y \rvert^\wgc}\times_{\underline{\lvert X \rvert}} X$ with $\lvert Y \rvert^\wgc$ the weakly generalizing closure of $\lvert Y \rvert$ in $\lvert X \rvert$ (see \cite[Proposition 2.8]{AGLR22}). 

Let $\mfr{X}$ be a formal scheme over $\bb{Z}_p$. The functor taking $S\in \Perf$ to the set of pairs $((S^\sharp,\iota), f \colon S^\sharp \to \mfr{X}^\ad)$ with $(S^\sharp, \iota)$ an untilt of $S$ is a $v$-sheaf and denoted by $\mfr{X}^\diamond$. For an adic ring $R$ over $\bb{Z}_p$, $\Spf(R)^\diamond$ is denoted by $\Spd(R)$. 

In general, the functor $\mfr{X}\mapsto \mfr{X}^\diamond$ from formal schemes over $\bb{Z}_p$ to small $v$-sheaves over $\Spd(\bb{Z}_p)$ is not fully faithful. However, the $v$-sheafification is fully faithful if we suitably restrict the category of formal schemes (cf. \Cref{lem:perfdrep}, \Cref{lem:uniqmapmaxlgood}).

It is useful that $\Spd(f)\colon \Spd(S)\to \Spd(R)$ is qcqs for an adic homomorphism $f\colon R\to S$ of adic rings over $\bb{Z}_p$ (see \cite[Lemma 2.26]{Gle24}). This implies that $\mfr{X}^\diamond$ is qcqs over $\Spd(R)$ for a qcqs adic formal $R$-scheme $\mfr{X}$. 

There is a $v$-sheaf attached to a scheme $X$ over $\bb{Z}_p$ that is often denoted by $X^\diamondsuit$ (see \cite[Section 2.2]{AGLR22}). It is a $v$-sheafification of the functor taking an affinoid perfectoid pair $(S,S^+)$ over $\bb{F}_p$ to the set of pairs $((S^\sharp,\iota), x)$ with $(S^\sharp,\iota)$ an untilt of $S$ and $x\in X(S^\sharp)$. The following lemma exhibits a certain relation between $X^\diamond$ and $X^\diamondsuit$. 

\begin{lem}
    \label{lem:perfprop}
    Let $R$ be a $\bb{Z}_p$-algebra and let $X$ be a proper $R$-scheme. There is a natural isomorphism $X^\diamond \cong X^\diamondsuit \times_{\Spec(R)^\diamondsuit} \Spec(R)^\diamond$. 
\end{lem}
\begin{proof}
    It follows from the same reason as in the comment after \cite[Remark 2.11]{AGLR22}. There is a natural inclusion $X^\diamond \to X^\diamondsuit \times_{\Spec(R)^\diamondsuit} \Spec(R)^\diamond$, and it is surjective by the valuative criterion for properness. 
\end{proof}

Here, we record the following basic lemma. 

\begin{lem}\label{lem:mapSpd}
    Let $R\to S$ be an adic integral homomorphism of adic rings. Then, $\Spd(S)\to \Spd(R)$ is proper. Moreover, if $R\to S$ is injective, $\Spd(S)\to \Spd(R)$ is surjective. 
\end{lem}
\begin{proof}
    We use a valuative criterion for properness (\cite[Proposition 18.3]{Sch17}) to the qcqs map $\Spd(S)\to \Spd(R)$. Let $K$ be a perfectoid field in characteristic $p$ with an open and bounded valuation subring $K^+ \subset K$ and consider the following diagram. 
    \begin{center}
        \begin{tikzcd}
            \Spa(K,O_K) \ar[r] \ar[d] & \Spd(S) \ar[d]\\
            \Spd(K,K^+) \ar[r] & \Spd(R)
        \end{tikzcd}
    \end{center}
    Let $(K^{\sharp}, K^{+\sharp})$ be the untilt of $(K,K^+)$ with an identification $\iota\colon (K^\sharp)^\flat \cong K$ and let $S\to O_{K^\sharp}$ be the corresponding homomorphism. Since it maps $R$ to $K^{+\sharp}$ and $K^{+\sharp}$ is integrally closed in $K^\sharp$, we see that the map factors through $K^{+\sharp}$. Thus, we obtain a valuative criterion. 

    Next, suppose that $R\to S$ is injective. Take an arbitrary map $R\to C^+$ with $C$ an algebraically closed perfectoid field with an open and bounded valuation subring $C^+\subset C$. Since $R\to S$ is injective, the map $\Spec(S)\to \Spec(R)$ is surjective. Since $C$ is algebraically closed, the map $R\to C$ has a lift $S\to C$. Since $S$ is integral over $R$, it maps $S$ to $C^+$. Thus, the map $\Spd(C,C^+)\to \Spd(R)$ admits a lift to $\Spd(S)$. 
\end{proof}

\subsection{Kimberlites and thick closures} \label{ssc:kimber}
In this section, we review the theory of kimberlites developed in \cite{Gle24}, \cite{AGLR22} and \cite{Gle26}.

Let $\SchPerf$ be the category of perfect schemes. It is equipped with the $v$-topology (see \cite[Definition 2.1]{BS17}) and a sheaf on $\SchPerf$ with the $v$-topology is called a scheme-theoretic $v$-sheaf. 

Let $X$ be a $v$-sheaf. The functor taking $S\in \SchPerf$ to the set $\Hom(S^\diamond, X)$ is a scheme-theoretic $v$-sheaf and denoted by $X^\red$ (see \cite[Lemma 2.26]{Gle24}). If $X$ is small, $X^\red$ is also small. Then, formally adic morphisms can be defined as in \cite[Definition 3.20]{Gle24}. A map $Y\to X$ of $v$-sheaves is formally adic if $ (Y^\red)^\diamond  \cong (X^\red)^\diamond  \times_X Y$. 

We explain several terminology to review the definition of kimberlites. A small $v$-sheaf $X$ is formally separated if the diagonal $\Delta_X \colon X\to X\times X$ is a formally adic closed immersion (see \cite[Definition 3.27]{Gle24}). A small $v$-sheaf $X$ is $v$-formalizing if it satisfies the following equivalent conditions (see \cite[Definition 4.5, 4.6, Lemma 4.7]{Gle24}, \cite[Proposition 2.22]{AGLR22}). 
\begin{enumerate}
    \item There is a surjection of $v$-sheaves from a disjoint union of $\Spd(R^+)$ with $(R,R^+)$ a perfectoid Huber pair to $X$. 
    \item For every map $\Spa(R,R^+)\to X$ from an affinoid perfectoid space, there is a $v$-cover $\Spa(S,S^+) \to \Spa(R,R^+)$ that admits an extension $\Spd(S^+)\to X$. 
\end{enumerate}
If $X$ is formally separated, $\Spa(S,S^+)\to X$ admits at most one extension $\Spd(S^+)\to X$ (see \cite[Proposition 4.9]{Gle24}). Now, a small $v$-sheaf $X$ is a prekimberlite if 
\begin{enumerate}
    \item $X$ is formally separated and $v$-formalizing, and
    \item $X^\red$ is represented by a perfect scheme and $(X^\red)^\diamond \to X$ is a closed immersion. 
\end{enumerate}

For a prekimberlite $X$, the open subsheaf $X\backslash (X^\red)^\diamond$ is called the analytic locus of $X$ and denoted by $X^\an$. The $v$-sheaf $\mfr{X}^\diamond$ is a prekimberlite for a separated formal scheme $\mfr{X}$ over $\bb{Z}_p$ (see \cite[Proposition 4.17]{Gle24}) and we have $\mfr{X}^{\diamond, \red}=(\mfr{X}_\red^\perf)^\diamond$.

One important property is that a prekimberlite $X$ admits a specialization map $\spc_X \colon \lvert X \rvert \to \lvert X^\red \rvert$ (see \cite[Definition 4.12]{Gle24}). The restriction of $\spc_X$ to $X^\an$ is denoted by $\spc_{X^\an}$. When $\mfr{X}$ is a separated Noetherian formal scheme over $\bb{Z}_p$, $\spc_{(\mfr{X}^\an)^\diamond}$ coincides with the usual specialization map of $\mfr{X}$ (see e.g. \cite[Section 3.1]{FK18}). 

For an open subspace $U\subset X^\red$, the open subsheaf of $X$ corresponding to $\spc_X^{-1}(U)$ is denoted by $\widehat{X}_{/U}$. When $X$ is represented by a separated formal scheme $\mfr{X}$ over $\bb{Z}_p$, $\widehat{X}_{/U}$ is represented by the open formal subscheme of $\mfr{X}$ whose underlying space is $U$ (see the proof of \cite[Theorem 2.16]{AGLR22}). 

Let $\mfr{X}$ be a Noetherian formal scheme. The specialization map of $\mfr{X}$ is quasicompact and closed (see \cite[Theorem 3.1.2]{FK18}), and surjective if $\mfr{X}$ is distinguished in the sense of \cite[Definition 2.1.8]{FK18} (see \cite[Proposition 3.1.5]{FK18}). However, we need to put more assumptions on a prekimberlite $X$ so that $\spc_{X^\an}$ behaves similarly. 

For a perfect scheme $Y$, let $Y^{\diamond/\circ}$ be the analytic sheafification of the functor taking $\Spa(R,R^+)\in \Perf$ to $Y(R^+_\red)$. For a prekimberlite $X$, there is a map $\SP_X\colon X \to (X^\red)^{\diamond/\circ}$ introduced in \cite[Section 4.4]{Gle24}. We say that $X$ is valuative if $\SP_X$ is partially proper. A valuative prekimberlite $X$ is a kimberlite if $X^\an$ is a quasiseparated locally spatial diamond and $\spc_{X^\an}$ is quasicompact (see \cite[Definition 4.35]{Gle24}). For a separated formal scheme $\mfr{X}$ over $\bb{Z}_p$, $\mfr{X}^\diamond$ is a kimberlite and the specialization map of $\mfr{X}$ refers to the map $\spc_{(\mfr{X}^{^\an})^\diamond}$. The specialization map $\spc_{X^\an}$ of a kimberlite $X$ is a spectral closed map (see \cite[Theorem 4.40]{Gle24}). 

Following \cite[Definition 4.7]{Gle26}, we say that a prekimberlite $X$ is thick if $\spc_{X^\an}$ is surjective. Similarly, we say that a separated formal scheme $\mfr{X}$ over $\bb{Z}_p$ is thick if the specialization map of $\mfr{X}$ is surjective. We say that an adic ring $R$ over $\bb{Z}_p$ is thick if $\Spf(R)$ is thick. 

\begin{lem}\label{lem:flatpres}
    A complete adic ring $R$ with an ideal of definition $p\in I\subset R$ is thick if and only if for every $f\in R$ such that $f^n\cdot I^n \subset I^{n+1}$ for some $n\geq 1$, we have $f\in R^{\circ \circ}$. 
\end{lem}
\begin{proof}
    First, $\Spf(R)$ is not thick if and only if for some $f\in R - R^{\circ \circ}$, $\Spa(R[\tfrac{1}{f}], R[\tfrac{1}{f}])$ contains only discrete points. It is equivalent to the condition that the $I$-adic completion $R[\tfrac{1}{f}]^\wedge$ is discrete (see e.g. \cite[Proposition 7.49 (2)]{Wed12}). It is equivalent to $I^n \cdot R[\tfrac{1}{f}] = I^{n+1}\cdot R[\tfrac{1}{f}]$ for some $n\geq 1$. Since $I$ is finitely generated, it is equivalent to $f^n\cdot I^n \subset I^{n+1}$ for some $n\geq 1$. 
\end{proof}

Following \cite[Chapter II, Definition 2.1.8]{FK18}, we say that a complete adic ring $R$ is distinguished if $R$ is $I$-torsion free for some ideal of definition $I\subset R$, and a formal scheme $\mfr{X}$ is distinguished if it is Zariski locally isomorphic to the formal spectrum of a distinguished complete adic ring. By \cite[Proposition 3.1.5]{FK18}, a distinguished separated Noetherian formal scheme over $\bb{Z}_p$ is thick. Here, we verify that this holds without Noetherian assumption. 

\begin{lem}\label{cor:nzdgenerator}
    Let $R$ be a complete adic ring in which $p$ is topologically nilpotent. If $R$ is distinguished, $R$ is thick.  
\end{lem}
\begin{proof}
    First, suppose that $R$ is $\varpi$-adic for some non-zero-divisor $\varpi \in R$. Then, the claim can be checked by \Cref{lem:flatpres}. Let $f\in R$ be an element such that $f^n\cdot (\varpi^n) \subset (\varpi^{n+1})$ for some $n\geq 1$. We can write $f^n \varpi^n=\varpi^{n+1}x$ for some $x\in R$. Since $\varpi$ is a non-zero-divisor, we have $f^n=\varpi x$. Thus, we have $f\in R^{\circ \circ}$.

    The general case can be reduced to the above case. Let $I\subset R$ be a finitely generated ideal of definition and let $Y\to \Spec(R)$ be the blowup along $V(I)$. Since $R$ is $I$-torsion free, $\Spec(R)-V(I)$ is dense in $\Spec(R)$ and $Y\to \Spec(R)$ is surjective. Let $\mfr{Y}$ be the $I$-adic completion of $Y$. Since $\mfr{Y}$ is locally monogeneous and distinguished, $\mfr{Y}$ is thick. Since the specialization maps are functorial and $Y\to \Spec(R)$ is surjective, $\Spf(R)$ is thick. 
\end{proof}

\begin{defi}
    Let $X$ be a kimberlite and let $Z\subset X^\an$ be a closed subsheaf. When the weakly generalizing subset
    \[
        \lvert \spc_{X^\an}(Z)^\diamond \rvert \cup \lvert Z \rvert \subset \lvert X \rvert
    \]
    is closed, we say that the associated closed subsheaf of $X$ is the thick closure of $Z$ in $X$. 
\end{defi}


Here, $\spc_{X^\an}(Z) \subset X^\red$ is a closed subscheme since $X$ is a kimberlite. 

\begin{exa}
    Suppose that $X$ is represented by a separated formal scheme $\mfr{X}$ and let $\mfr{Z} \subset \mfr{X}$ be a thick closed formal subscheme. Then, $\mfr{Z}^\diamond$ is the thick closure of $(\mfr{Z}^\an)^\diamond \subset X^\an$. 
\end{exa}


Though thick closures and $v$-closures may differ in general (see \cite[footnote 10, Lemma 2.37]{AGLR22}), they will coincide under suitable Noetherian assumptions due to the following. 

\begin{lem} \label{lem:thickvclos}
    Let $R$ be a Noetherian adic ring with an ideal of definition $I\subset R$. If $R$ is $I$-torsion free, $\Spd(R)^\an$ is dense in $\Spd(R)$. 
\end{lem}
\begin{proof}
    This can be proved in the same way as \cite[Lemma 4.4]{Lou20}. Let $f_1,\ldots,f_n\in I$ be a set of generators and let $(S,S^+)=(R\llbracket s \rrbracket \langle \tfrac{f_1,\ldots,f_r}{s} \rangle, R\llbracket s \rrbracket \langle \tfrac{f_1,\ldots,f_r}{s} \rangle^+)$, which is a rational localization of $R\llbracket s \rrbracket$. Since $R$ is Noetherian and $I$-torsion free, $(S,S^+)$ is strongly Noetherian and $I$-torsion free. Then, for any rational localization $(T,T^+)$ of $(S,S^+)$, $T$ is flat over $S$, so $T$ is $I$-torsion free. In particular, the inverse image of $\Spa(R)^\an$ in $\Spa(S,S^+)$ is dense. Since $\Spd(S,S^+) \to \Spd(R)$ is surjective and $\lvert \Spd(S,S^+) \rvert \cong \lvert \Spa(S,S^+) \rvert$ by \cite[Proposition 15.4]{Sch17}, we see that $\Spd(R)^\an$ is dense in $\Spd(R)$. 
\end{proof}

\subsection{Geometric quotients} \label{ssc:geomquot}

In this section, we introduce geometric quotients of $v$-sheaves by group actions. The theory of $v$-sheaves is quite topological and the theory of geometric quotients works well. 

\begin{defi}
    Let $\Gamma$ be a group and let $X$ be a $v$-sheaf with a $\Gamma$-action. A map $\pi\colon X\to Y$ of $v$-sheaves is a geometric quotient of $X$ by $\Gamma$ if $\pi$ is surjective, and for every geometric point $x\in X(C,C^+)$ with $C$ an algebraically closed perfectoid field with an open and bounded valuation subring $C^+\subset C$, we have $\pi^{-1}(\pi(x)) = \{ \gamma \cdot x \vert \gamma \in \Gamma \}$. 
\end{defi}

\begin{lem} \label{lem:bcquot}
    Let $X$ be a $v$-sheaf with an action of a group $\Gamma$. Let $\pi\colon X\to Y$ be a geometric quotient of $X$ by $\Gamma$. For every map $Z\to Y$ of $v$-sheaves, the base change $X\times_{Y} Z\to Z$ is a geometric quotient of $X\times_Y Z$ by the natural $\Gamma$-action. 
\end{lem}
\begin{proof}
    It is immediate from the definition. 
\end{proof}

We need some assumptions to show the uniqueness of geometric quotients. For this, we fix a $v$-sheaf $S$ that plays the role of a base and consider small $v$-sheaves over $S$. 

\begin{prop}\label{prop:geomquotiscat}
    Let $X$ be a small $v$-sheaf with an action of a group $\Gamma$ over $S$. Let $\pi\colon X\to Y$ be a geometric quotient of $X$ by $\Gamma$ over $S$. If $Y$ is quasiseparated over $S$, then we have the following. 
    \begin{enumerate}
        \item For every $\gamma\in \Gamma$, we have $\gamma\cdot \pi=\pi$. 
        \item For every small $v$-sheaf $Z$ quasiseparated over $S$, a map $f\colon X \to Z$ over $S$ factors through $Y$ if and only if $\gamma\cdot f=f$ for every $\gamma\in \Gamma$. 
    \end{enumerate}
\end{prop}
\begin{proof}
    For (1), we can check the equality $\gamma\cdot\pi = \pi$ at each geometric point of $X$ by \Cref{lem:eqqs}, so the claim follows. 
    For (2), it is enough to show that $f$ factors through $Y$ if $\gamma\cdot f=f$ for every $\gamma\in \Gamma$.
    It is enough to check the equality of two maps $f\circ \mrm{pr}_i\colon X\times_{Y} X\rightrightarrows Z$ ($i=1,2$).
    Again, we can check the equality at each geometric point by \Cref{lem:eqqs}. 
    Since any geometric point of $X\times_{Y} X$ is of the form $(x,\gamma\cdot x)$ for some $\gamma\in \Gamma$, we get the claim. 
\end{proof}

The main source of geometric quotients by finite groups is as follows.  

\begin{prop} \label{prop:finquot}
    Let $R\to S$ be an injective adic homomorphism of complete adic rings. Suppose that a finite group $\Gamma$ acts on $S$ so that $R=S^\Gamma$. Then, $\Spd(S)\to \Spd(R)$ is a geometric quotient by $\Gamma$. 
\end{prop}
\begin{proof}
    Since $R=S^\Gamma$, $R\to S$ is injective and integral. Thus, $\pi\colon \Spd(S)\to \Spd(R)$ is surjective by \Cref{lem:mapSpd}. Take a geometric point $x\in \Spd(S)(C,C^+)$ with $C$ an algebraically closed perfectoid field with an open and bounded valuation subring $C^+\subset C$. It corresponds to an untilt $(C^\sharp, \iota)$ and a continuous homomorphism $S\to C^{+\sharp}$. Set $\pi^\sch\colon \Spec(S)\to \Spec(R)$ and let $x_{C^\sharp}\colon \Spec(C^\sharp)\to \Spec(S)$ be the geometric point corresponding to $R\to C^{\sharp}$. Since $[\Spec(S)/\Gamma]\to \Spec(R)$ is a coarse moduli space (see \cite[Theorem 3.1]{Con05}), $(\pi^\sch)^{-1}(\pi^\sch(x_{C^\sharp}))$ is the $\Gamma$-orbit of $x_{C^\sharp}$. Thus, for every $y\in \pi^{-1}(\pi(x))$, we have $y_{C^\sharp}=\gamma\cdot x_{C^\sharp}$ for some $\gamma\in \Gamma$, and $y=\gamma\cdot x$ in $\pi^{-1}(\pi(x))$ for such $\gamma$. 
\end{proof}

The condition in this example can be slightly relaxed in the following way. A merit of using the following condition is the compatibility with base change. 

\begin{lem} \label{lem:invanalogue}
    Let $R$ be a ring and let $S$ be an $R$-algebra with a group action of a finite group $\Gamma$ over $R$. The following conditions are equivalent.
    \begin{enumerate}
        \item For every $R$-algebra $A$ and every element $a\in A\otimes_R S$, $\prod_{\gamma \in \Gamma} \gamma\cdot a $ lies in the image of $A\to A\otimes_R S$. 
        \item For every element $a\in S[T]$, $\prod_{\gamma \in \Gamma} \gamma\cdot a $ lies in the image of $R[T] \to S[T]$. 
    \end{enumerate}
    Moreover, this condition is satisfied if $R=S^\Gamma$. 
\end{lem}
\begin{proof}
    We show $(2) \Rightarrow (1)$. Let $A$ be an arbitrary $R$-algebra and let $a\in A\otimes_R S$. We may write $a=\sum_{1\leq i \leq n} a_i \otimes s_i$ with $a_i\in A$ and $s_i\in S$. Then, $a$ is equal to the image of $\sum_{1\leq i \leq n} T_i \otimes s_i$ through $S[T_1,\ldots,T_n] \to A\otimes_R S$ sending $T_i$ to $a_i$. Thus, we may assume that $A=R[T_1,\ldots T_n]$ and $a=\sum_{1\leq i \leq n} T_i \otimes s_i$. We like to show that every coefficient of $\prod_{\gamma \in \Gamma} \gamma\cdot a$ lies in the image of $R\to S$. Take a positive integer $N > \# \Gamma$ and let $a'\in S[T]$ be the image of $a$ through $S[T_1,\ldots, T_n]\to S[T]$ sending $T_i$ to $T^{N^i}$. Then, every coefficient of $\prod_{\gamma \in \Gamma} \gamma\cdot a$ lies in the image of $R\to S$ if and only if every coefficient of $\prod_{\gamma \in \Gamma} \gamma\cdot a'$ lies in the image of $R\to S$. The latter follows from the assumption (2).  

    Since every coefficient of $\prod_{\gamma \in \Gamma} \gamma\cdot a $ is stable under the action of $\Gamma$ for every $a\in S[T]$, the condition (2) is satisfied when $R=S^\Gamma$.
\end{proof}

As we show in the following, the condition in \Cref{lem:invanalogue} is transitive and stable under restricting to subalgebras.  

\begin{lem} \label{lem:invtrans}
    Let $R\to M \to S$ be a homomorphism of rings and let $\Gamma$ be a finite group acting on $S$ over $R$. Let $H\subset \Gamma$ be a normal subgroup such that $H$ acts on $S$ over $M$ and there is an action of $\Gamma/H$ on $M$ over $R$ that is compatible with the action of $\Gamma$ on $S$. If the action of $\Gamma/H$ on $M$ over $R$ and the action of $H$ on $S$ over $M$ satisfy the condition in \Cref{lem:invanalogue}, the action of $\Gamma$ on $S$ over $R$ satisfies that condition. 
\end{lem}
\begin{proof}
    It is enough to check the condition (2) in \Cref{lem:invanalogue}. Let $a\in S[T]$. Since the action of $H$ on $S$ over $M$ satisfies that condition, $\prod_{\gamma \in H} \gamma\cdot a$ is in the image of $M[T]\to S[T]$. Take an element $b\in M[T]$ mapping to $\prod_{\gamma \in H} \gamma\cdot a$ in $S[T]$. Then, $\prod_{\gamma \in \Gamma} \gamma\cdot a$ is the image of $\prod_{\gamma \in \Gamma/H} \gamma\cdot b$ in $S[T]$, so it lies in the image of $R[T]\to S[T]$ since the action of $\Gamma/H$ on $M$ over $R$ satisfies the condition (2) in \Cref{lem:invanalogue}. 
\end{proof}

\begin{lem} \label{lem:invsubalg}
    Let $f\colon R\to S$ be a homomorphism of rings and let $\Gamma$ be a finite group acting on $S$ over $R$. Let $S'\subset S$ be a subring stable under $\Gamma$ and let $R'=f^{-1}(S')$. If the action of $\Gamma$ on $S$ over $R$ satisfies the condition in \Cref{lem:invanalogue}, the action of $\Gamma$ on $S'$ over $R'$ also satisfies that condition. 
\end{lem}
\begin{proof}
    For every element $a\in S'[T]$, $\prod_{\gamma \in \Gamma} \gamma\cdot a $ lies in the intersection of the image of $R[T]\to S[T]$ with $S'[T]$. That intersection is equal to the image of $R'[T]\to S'[T]$ as $R'=f^{-1}(S')$. 
\end{proof}

\begin{prop} \label{prop:affquot}
    Let $R$ be an adic ring with an ideal of definition $p\in I\subset R$. Let $S$ be an $I$-adic $R$-algebra with an action of a finite group $\Gamma$ over $R$. If $R\to S$ is injective and satisfies the condition in \Cref{lem:invanalogue}, $\Spd(S)\to \Spd(R)$ is a geometric quotient by $\Gamma$. 
\end{prop}
\begin{proof}
    By the condition in \Cref{lem:invanalogue}, for every element $a\in S$, $\prod_{\gamma\in \Gamma}(T-\gamma\cdot a)$ is a polynomial with coefficients in $R$. Thus, $S$ is integral over $R$ and $\Spd(S)\to \Spd(R)$ is surjective by \Cref{lem:mapSpd}. Let $\pi^\sch\colon \Spec(S) \to \Spec(R)$. As in the proof of \Cref{prop:finquot}, it suffices to show that for every geometric point $x\colon \Spec(C) \to \Spec(S)$ with $C$ an algebraically closed field, $(\pi^\sch)^{-1}(\pi^\sch(x))$ is equal to the $\Gamma$-orbit of $x$. Let $S_C = S\otimes_{R, \pi^{\sch}(x)^*} C$. Since $R\to S$ is injective, $\pi^\sch$ is surjective and $S_C\neq 0$. By the condition in \Cref{lem:invanalogue}, for every $a\in S_C$, we have $\prod_{\gamma \in \Gamma} \gamma\cdot a \in C$. Then, we see that $\Spec(S_C)^\Gamma\to \Spec(C)$ is bijective. Moreover, $C\to (S_C)^\Gamma$ is integral and $C$ is algebraically closed, so the reduction of $(S_C)^\Gamma$ is isomorphic to $C$. Since $[\Spec(S_C)/\Gamma] \to \Spec(S_C)^\Gamma$ is a coarse moduli space, $(\pi^\sch)^{-1}(\pi^\sch(x))$ is equal to the $\Gamma$-orbit of $x$. 
\end{proof}

A typical way of constructing a geometric quotient by a profinite group is as follows. 

\begin{prop}\label{prop:limquot}
    Let $\{\Gamma_i\}_{i\geq 0}$ be a sequence of groups with surjective transition maps $\Gamma_{i+1}\twoheadrightarrow \Gamma_i$ such that $\Gamma_0=\{1\}$. 
    Let $\{X_i\}_{i\geq 0}$ be a sequence of $v$-sheaves with transition maps $X_{i+1}\to X_i$ such that each $X_i$ is equipped with a $\Gamma_i$-action compatible with transition maps. 
    Let $\Gamma_\infty=\lim\limits_{i\geq 0} \Gamma_i$ and let $X_\infty=\lim\limits_{i\geq 0} X_i$. 
    Let $\Gamma_{i+1}^{i}=\mrm{Ker}(\Gamma_{i+1}\twoheadrightarrow \Gamma_i)$. 
   If $X_{i+1}\to X_i$ is a geometric quotient by $\Gamma_{i+1}^{i}$ for every $i\geq 0$, then $X_\infty\to X_0$ is a geometric quotient by $\Gamma_\infty$. 
\end{prop}
\begin{proof}
    Let $\pi_\infty\colon X_\infty \to X_0$. Since the $v$-site $\Perf$ is replete (by the same proof as in \cite[Lemma 2.6]{Heu21}) and each transition map is surjective, $\pi_\infty$ is surjective. Let $C$ be an algebraically closed perfectoid field with an open and bounded valuation subring $C^+\subset C$. 
    Let $x\in X_\infty(C,C^+)$ and let $y\in \pi_\infty^{-1}(\pi_\infty(x))$. 
    Let $\pi_\infty^i\colon X_\infty\to X_i$. 
    We construct a convergent sequence $\{\gamma_i\}_{i\geq 0}$ of $\Gamma$ so that
    $\pi_\infty^i(\gamma_j\cdot y)=\pi_\infty^i(x)$ for every $i\leq j$. 
    We do this by induction. 
    If $i=0$, we can take $\gamma_0=1$. 
    Suppose that we construct the $i$-th term $\gamma_i$. 
    Then, $\pi_\infty^{i+1}(\gamma_i\cdot y)$ is an element of $X_{i+1}(C,C^+)$ which maps to 
    $\pi_\infty^i(x)$ through the transition map $X_{i+1}\to X_i$. 
    Thus, there is an element $\delta\in \Gamma_{i+1}^{i}$ such that $\delta\cdot \pi_\infty^{i+1}(\gamma_i\cdot y)=\pi_\infty^{i+1}(x)$.
    We can take $\gamma_{i+1}$ to be the product of a lift of $\delta$ and $\gamma_i$. 
    If we let $\gamma=\lim \gamma_i\in \Gamma$, we have $\gamma\cdot y=x$. 
\end{proof}

Another way to construct geometric quotients is to pass to closed subsheaves. 

\begin{lem} \label{lem:quottosubsh}
    Let $\pi \colon X\to Y$ be a proper map of small $v$-sheaves that is a geometric quotient by a group $\Gamma$. Let $Z\subset X$ be a $\Gamma$-stable closed subsheaf and let $Z_Y \subset Y$ be the image of $Z$ under $\pi$. Then, we have $Z\cong X\times_Y Z_Y$ and $Z\to Z_Y$ is a geometric quotient by $\Gamma$. 
\end{lem}
\begin{proof}
    Since $\pi$ is proper, $Z_Y \subset Y$ is a closed subsheaf. For every geometric point $y$ of $X\times_Y Z_Y$, there is a geometric point $x$ of $Z$ such that $\pi(x)=\pi(y)$. Since $\pi$ is a geometric quotient by $\Gamma$, there is some $\gamma \in \Gamma$ such that $y=\gamma\cdot x$. Since $Z$ is stable under $\Gamma$, $y$ is a geometric point of $Z$. Thus, we have $Z=X\times_Y Z_Y$ as closed subsheaves of $X$. By \Cref{lem:bcquot}, $Z\to Z_Y$ is a geometric quotient by $\Gamma$. 
\end{proof}

\section{Perfectoid formal schemes} \label{sec:perfdformal}

The notion of $p$-adic perfectoid formal schemes is introduced in \cite{RC21}. It is constructed from formal spectra of perfectoid rings introduced in \cite{BMS18}. Here, we introduce not necessarily $p$-adic perfectoid formal schemes and establish some basic properties. 

\subsection{$F$-standard ideals}

In this section, we recall $F$-standard ideals introduced in \cite[Section 2.1]{Tak25} and explain the basic properties. 

An ideal $I$ of a perfect ring $A$ (in characteristic $p$) is $F$-standard if $I^p=\Frob(I)$. An $F$-standard ideal $I\subset A$ defines an ideal $[I] \subset W(A)$ consisting of elements $\sum_{n\geq 0} [a_n]p^n$ with $a_n\in I$. For $n=m/p^k \in \bb{Z}[\tfrac{1}{p}]_{\geq 0}$ with $m, k \geq 0$, we set $I^n=\Frob(I^m)^{-k}$. This is independent of the choice of $m$ and $k$ since $I$ is $F$-standard, and $I^n$ is also $F$-standard.

When a perfect ring $A$ is equipped with a linear topology of a finitely generated ideal, we say that $A$ is an adic perfect ring. By \cite[Lemma 2.5]{Tak25}, every adic perfect ring admits an $F$-standard ideal of definition. Note that such an ideal of definition is not necessarily finitely generated.

\begin{lem} 
    Let $A$ be a perfect ring and let $I\subset A$ be an $F$-standard ideal. The ideal $[I]\subset W(A)$ is $p$-adically closed and $W(A)/[I]$ is $p$-torsion free and $p$-adically complete. 
\end{lem}
\begin{proof}
    Since $[I]\cap (p^n) = p^n [I]$ and $[I]$ is $p$-adically complete, we see that $[I]\subset W(A)$ is a closed ideal and $W(A)/[I]$ is $p$-adically complete. Moreover, $W(A)/[I]$ is $p$-torsion free as $[I]\cap (p) = p [I]$. 
\end{proof}

\begin{lem} \label{lem:prodFstd}
    Let $A$ be a perfect ring. For $F$-standard ideals $I, J\subset A$, the ideals $I\cap J, IJ$ are $F$-standard. Moreover, $[I\cap J] = [I] \cap [J]$ and $[IJ]$ is the $p$-adic closure of $[I][J]$ in $W(A)$. 
\end{lem}
\begin{proof}
    Since we have $\Frob(I\cap J) \subset (I\cap J)^p \subset I^p\cap J^p = \Frob(I) \cap \Frob(J)$, $I\cap J$ is $F$-standard. By definition, we have $[I\cap J] = [I] \cap [J]$. On the other hand, since $(IJ)^p = I^pJ^p = \Frob(I) \Frob(J) = \Frob(IJ)$, $IJ$ is $F$-standard. Next, we show $[I][J] \subset [IJ]$. Let $a=\sum\limits_{n= 0}^\infty [a_n] p^n\in I$ and $b=\sum\limits_{n=0}^\infty [b_n] p^n \in J$. Then, $ab=\sum\limits_{n,m=0}^\infty [a_nb_m]p^{n+m}$, so $ab\in [IJ]$ as $a_nb_m \in IJ$. To prove the claim, it is enough to show that $[IJ] = [I][J] + p[IJ]$ since $[IJ]\subset W(A)$ is closed. Let $x=\sum\limits_{n=0}^\infty [a_n] p^n$ be an element of $[IJ]$. Then, $a_0\in IJ$, so we may write $a_0 = \sum\limits_{ 1\leq i \leq m} b_ic_i$ with $b_i\in I$ and $c_i\in J$ for each $1\leq i \leq m$. Then, $x-\sum\limits_{1 \leq i \leq m} [b_m][c_m]\in [IJ] \cap (p) = p [IJ]$, so we obtain $x\in [I][J] + p[IJ]$. 
\end{proof}

From this lemma, we have $(p^n,[I]^m) = (p^n, [I^m])$ for every $n, m\geq 0$. 

\begin{lem} \label{lem:Wittcomp}
    Let $A$ be a perfect ring and let $I\subset A$ be an $F$-standard ideal. Let $A^\wedge$ be the $I$-adic completion of $A$. The $(p,[I])$-adic completion of $W(A)$ is isomorphic to $W(A^\wedge)$. In particular, $W(A)$ is $(p,[I])$-adically complete if $A$ is $I$-adically complete. 
\end{lem}
\begin{proof}
    It is enough to show that the $[I]$-adic completion of $W_{n}(A)$ is isomorphic to $W_{n}(A^\wedge)$ for every $n\geq 0$. For every $m\geq 0$, there are natural maps $W_{n}(A/I^{p^nm}) \to W_{n}(A)/[I^m] \to W_{n}(A/I^m)$ since we have $\sum [a_n]p^n = (a_0,a_1^p,a_2^{p^2},\ldots)$ in $W(A)$. Thus, $\{W_{n}(A/I^m) \}_{m\geq 0}$ and $\{W_{n}(A)/[I^m] \}_{m\geq 0}$ are isomorphic as pro-systems. Thus, we have $\lim\limits_{m\geq 0} W_{n}(A)/[I]^m = \lim\limits_{m\geq 0} W_{n}(A/I^m) = W_{n}(A^\wedge)$. 
\end{proof}

\subsection{Adic perfectoid rings} \label{ssc:adicperfd}

In this section, we introduce affine building blocks of perfectoid formal schemes, which we call complete adic perfectoid rings. This notion is equivalent to perfectoid rings in the sense of Gabber-Romero (see \cite[Definition 16.3.1]{GR18}), but we use another term in order to clarify the terminology.

\begin{defi}
    An adic ring $R$ on which $p$ is topologically nilpotent is said to be an adic perfectoid ring if the underlying $p$-adic ring $R$ is a perfectoid ring as in \cite[Definition 3.5]{BMS18}. Moreover, if $R$ is a complete adic ring, we say that $R$ is a complete adic perfectoid ring. 
\end{defi}

When we simply say that $R$ is a perfectoid ring, $R$ is endowed with the $p$-adic topology. We will consider tilting from adic perfectoid rings to adic perfect rings. For a perfectoid ring $R$ and $f\in R^\flat$, the image of the Teichm\"{u}ller lift $[f]\in W(R^\flat)$ in $R$ is denoted by $f^\sharp$. The map $W(R^\flat)\to R$ is denoted by $\theta_R$ and a generator of $\Ker(\theta_R)$ is denoted by $\xi_R$. This element is distinguished in the sense of \cite[Lemma 2.33]{BS22}. The image of $\xi_R$ along $W(R^\flat)\to R^\flat$ is denoted by $\xi_{R,0}$. 

\begin{lem}
    Let $R$ be an adic perfectoid ring. There is an ideal of definition of $R$ generated by $p$ and $f_1^\sharp, \ldots, f_n^\sharp$ for some $f_1,\ldots,f_n\in R^\flat$. Moreover, the adic topology on $R^\flat$ generated by $\xi_{R,0}, f_1,\ldots,f_n$ is independent of the choice of $f_1,\ldots,f_n$. 
\end{lem}
\begin{proof}
    Let $I$ be a finitely generated ideal of definition of $R$ such that $p\in I$. Since the image of the map $f\mapsto f^\sharp$ is $p$-adically dense in $R$, we can take a set of generators of $I$ consisting of $p$ and $f_1^\sharp, \ldots, f_n^\sharp$. It is enough to show that the adic topology on $R^\flat$ generated by $\xi_{R,0}, f_1, \ldots, f_n$ is independent of the choice of $I$ and $f_1,\ldots, f_n$. 

    Let $I^\flat = (\xi_{R,0}, f_1,\ldots, f_n)$. The ideal $I^\flat$ is independent of the choice of $f_1,\ldots,f_n$ since $I^\flat$ is the inverse image of the ideal $I/pR$ along $R^\flat \to R^\flat/\xi_{R,0}R^\flat \cong R/pR$. From this description, we see that $J^\flat \subset I^\flat$ for any finitely generated ideal of definition $J$ of $R$ such that $p\in J \subset I$. Let $J$ be another finitely generated ideal of definition of $R$ such that $p\in J$. For sufficiently large $N\geq 0$, we have $(p)+J^N \subset I$ and $(p)+I^N \subset J$, so it implies that $(\xi_{R,0})+J^{\flat, N } \subset I^\flat$ and $(\xi_{R,0})+I^{\flat, N } \subset J^\flat$. Thus, the two adic topologies on $R^\flat$ given by $I^\flat$ and $J^\flat$ are equal. 
\end{proof}

Let $I$ be an ideal of $R$ such that $p\in I$. As in the above proof, the inverse  image of the ideal $I/pR$ along $R^\flat \to R^\flat/\xi_{R,0}R^\flat \cong R/pR$ is denoted by $I^\flat$. Note that since $\xi_{R}$ is distinguished, $p$ and $\xi_{R,0}^\sharp$ generate the same ideal in $R$, so $I$ is generated by $\xi_{R,0}^\sharp, f_1^\sharp,\ldots,f_n^\sharp$. For an ideal $J\subset R^\flat$, the ideal of $R$ generated by $f^\sharp$ for all $f\in J$ is denoted by $J^\sharp$. The above argument shows that for an ideal $p\in I \subset R$, we have $I=(I^\flat)^\sharp$. 

\begin{defi}
    Let $R$ be an adic perfectoid ring and let $I$ be a finitely generated ideal of definition of $R$ such that $p\in I$. The adic perfect ring $R^\flat$ with the $I^\flat$-adic topology is said to be the tilting of $R$, which is independent of the choice of $I$. 
\end{defi}

\begin{lem} \label{lem:Jsharp}
    Let $R$ be an adic perfectoid ring and let $J\subset R^\flat$ be an open $F$-standard ideal. The ideal $J^\sharp$ is equal to the image of $[J]\subset W(R^\flat)$ under $\theta_R$. 
\end{lem}
\begin{proof}
    Since $J$ is open, $\xi_{R,0}^n\in J$ for some $n\geq 0$, so we have $p^n \in J^\sharp$ since $\xi_R$ is distinguished. Then, the claim follows from the fact that $[J]\subset W(R^\flat)$ is the $p$-adic closure of the ideal generated by $[f]$ for all $f\in J$. 
\end{proof}

We will prove that the completeness of adic perfectoid rings is preserved under tilting.

\begin{lem}\label{lem:mulxi}
    Let $A$ be a perfect ring and let $\xi \in W(A)$ be a distinguished element. Let $\xi_0$ be the image of $\xi$ along $W(A)\to A$ and let $I\subset A$ be an $F$-standard ideal such that $\xi_0 \in I$. If we endow $W(A)$ with the $(p,[I])$-adic topology, the map $W(A)\to W(A)$ given by the multiplication by $\xi$ is a homeomorphism onto the image. 
\end{lem}
\begin{proof}
    It is easy to see that the multiplication by $\xi$ is continuous. We show that it is a homeomorphism onto its image. 

    Let $n\geq 1$ be an integer. We show that if $x\in W(A)$ satisfies $\xi x \in (p^n, [I^n])$, we have $x\in (p^s, [I^t])$ for every $s,t\geq 0$ with $s+t < n$. Suppose that $x$ is not contained in $(p^s, [I^t])$ and let $x=\sum_{m\geq 0} [a_m]p^m$. Then, there exists $0\leq k < s$ such that $a_k \notin I^t$. Take the minimal $k$ such that $a_k \notin I^t$. Then, we are in the situation $(\ast)$ that $a_0,\ldots, a_{k-1} \in I^\ell$ and $a_k \notin I^\ell$, but $\xi x \in (p^{k+2}, [I^\ell])$ for $k \leq s-1 \leq n-2$ and $\ell = t \geq 0$. 

    Suppose that we are in the above situation for some $k$ and $\ell\geq 1$. We may suppose that  $a_k \in I^{\ell-1}$. Let $\xi=\sum_{m\geq 0} [\xi_m]p^m$. Since $\xi$ is distinguished, $\xi_1$ is a unit and we may assume that $\xi_1=1$. Then, $\xi x$ is congruent to $\sum_{m_1+m_2 \leq k+1} [a_{m_1} \xi_{m_2}] p^{m_1+m_2}$ modulo $p^{k+2}$. It lies in $(p^{k+2}, [I^\ell])$, so we have $[\xi_0 a_{k}] p^{k} + ([\xi_0 a_{k+1}] + [a_{k}]) p^{k+1} \in (p^{k+2}, [I^\ell])$. Then, we have $\xi_0 a_k, \xi_0 a_{k+1}+a_k \in I^\ell$. Since $\xi_0\in I$ and $a_k\notin I^\ell$, we have $a_{k+1} \notin I^{\ell -1}$. Thus, if $k+1 \leq n-2$, the hypothesis $(\ast)$ holds for $k+1$ and $\ell-1$. 

    If we continue this procedure until $\ell=0$, we see that the hypothesis $(\ast)$ holds for some $k\leq s+t-1\leq n-2$ and $\ell=0$. This is a contradiction because $a_k \in A = I^0$. Thus, we see that $x\in (p^s, [I^t])$ and we obtain the claim. 
\end{proof}

\begin{prop} \label{prop:completeW}
    Let $R$ be a perfectoid ring and let $I\subset R$ be an ideal such that $p\in I$ and $I^\flat \subset R^\flat$ is $F$-standard. Then, $W(R^\flat)$ is $(p,[I^\flat])$-adically complete if and only if $R$ is $I$-adically complete. 
\end{prop}
\begin{proof}
    First, suppose that $W(R^\flat)$ is $(p,[I^\flat])$-adically complete. By \Cref{lem:mulxi}, $\Ker(\theta_R)$ is a closed ideal, so the quotient $R=W(R^\flat)/\Ker(\theta_R)$ is complete with respect to the quotient topology. The claim follows since $I$ is the image of $(p,[I^\flat])$ along $\theta_R$. 
    
    Next, suppose that $R$ is $I$-adically complete. First, we show that $W(R^\flat)$ is separated. Suppose that we have $a \in \bigcap_{n\geq 0} (p^n, [I^{\flat, n}])$. Then, $\theta_R(a) \in \bigcap_{n\geq 0} I^n$, so $\theta_R(a)=0$. If we write $a=\xi a'$, then we also have $a' \in \bigcap_{n\geq 0} (p^n, [I^{\flat, n}])$ by \Cref{lem:mulxi}. Thus, we have $\bigcap_{n\geq 0} (p^n, [I^{\flat, n}]) \subset \xi^m W(R^\flat)$ for every $m\geq 0$. Since $W(R^\flat)$ is $\xi$-adically complete by \Cref{lem:Wittcomp}, we have $\bigcap_{n\geq 0} (p^n, [I^{\flat, n}])=0$. Next, we show that $W(R^\flat)$ is complete. Let $(a_n)_{n\geq 0}$ be a sequence in $W(R^\flat)$ such that $a_{n+1}-a_n \in (p^n, [I^{\flat, n}])$. Since $R$ is $I$-adically complete, $(\theta_R(a_n))$ converges to an element $\ov{b}\in R$. Let $b$ be a lift of $\ov{b}$ in $W(R^\flat)$ and let $a'_n=a_n-b$. Then, $\theta_R(a'_n) \in (p^n,(I^{\flat, n})^\sharp)$, so by replacing $a_n$ with an element of $a_n + (p^n, [I^{\flat, n}])$, we may assume that $a'_n \in \Ker(\theta_R)$. By \Cref{lem:mulxi}, the sequence $(a'_n/\xi)$ is also a Cauchy sequence, so we may repeat this procedure to $(a'_n/\xi)$. After all, we see that there is a sequence $(b_m)_{m\geq 0}$ such that for every $k\geq 0$, $a_n - \sum_{0\leq m <k} b_m \xi^m$ modulo $\xi^k$ is convergent to $0$ in $W(R^\flat)/(\xi^k)$. Since $W(R^\flat)$ is $\xi$-adically complete, $\sum_{m\geq 0} b_m\xi^m$ is a limit of $(a_n)$. 
\end{proof}

\begin{cor} \label{cor:tiltcomplete}
    Let $R$ be an adic perfectoid ring. Then, $R$ is complete if and only if $R^\flat$ is complete. 
\end{cor}
\begin{proof}
    Let $I\subset R$ be an ideal of definition such that $p\in I$. By \Cref{prop:completeW}, $R$ is complete if and only if $W(R^\flat)$ is $(p,[I^\flat])$-adically complete, and $R^\flat$ is complete if and only if $W(R^\flat)$ is $(p,[I^\flat])$-adically complete. Thus, the claim follows. 
\end{proof}

\begin{cor} \label{cor:tiltofcomp}
    Let $R$ be an adic perfectoid ring and let $I\subset R$ be an ideal of definition such that $p\in I$. The $I$-adic completion $R^\wedge$ is a complete adic perfectoid ring. Its tilt $R^{\wedge, \flat}$ is isomorphic to the $I^\flat$-adic completion of $R^\flat$. 
\end{cor}
\begin{proof}
    Let $J_n = (p,[I^\flat])^n \cap (\xi_R)$. The short exact sequence $0\to (\xi_R) \to W(R^\flat) \to R\to 0$ gives rise to a short exact sequence $0\to (\xi_R)/J_n \to W(R^\flat)/(p,[I^\flat])^n \to R/I^n \to 0$. By \Cref{lem:mulxi}, the limit of $(\xi_R)/J_n$ is isomorphic to the $(p,[I^\flat])$-adic completion of $(\xi_R)$. Let $R^{\flat, \wedge}$ be the $I^\flat$-adic completion of $R^\flat$. By taking the limit of the above sequences, we obtain a short exact sequence $0\to \xi_R\cdot W(R^{\flat, \wedge}) \to W(R^{\flat, \wedge}) \to R^\wedge \to 0$ by \Cref{lem:Wittcomp}. As $\xi_R$ is distinguished, we see that $R^\wedge$ is a complete adic perfectoid ring with tilt $R^{\flat, \wedge}$. 
\end{proof}

The tilting equivalence holds in our setting as well. For an adic perfectoid ring $R$, an adic (resp.\ complete adic) perfectoid $R$-algebra refers to an adic (resp.\ complete adic) perfectoid ring $S$ with a continuous homomorphism $R\to S$. 

\begin{prop} \label{prop:conttiltequiv}
    Let $R$ be an adic perfectoid ring. The category of adic (resp.\ complete adic) perfectoid $R$-algebras is equivalent to the category of adic (resp.\ complete adic) perfect $R^\flat$-algebras. 
\end{prop}
\begin{proof}
    Let $R$ and $S$ be adic perfectoid rings. By the tilting equivalence of perfectoid rings (see \cite[Theorem 3.10]{BS22}), it is enough to show that a homomorphism $f\colon R\to S$ is continuous if and only if $f^\flat \colon R^\flat \to S^\flat$ is continuous. Since $R$ and $S$ are $p$-adic, $f$ is continuous if and only if $f/p \colon R/pR \to S/pS$ is continuous. Since $R^\flat$ and $S^\flat$ is $\xi_{R,0}$-adic, $f^\flat$ is continuous if and only if $f/\xi_{R,0}\colon R^\flat/\xi_{R,0} R^\flat \to S^\flat/\xi_{R,0}S^\flat$ is continuous. Thus, the claim follows from the isomorphism $R/pR \cong R^\flat/\xi_{R,0}R^\flat$.  
\end{proof}

For a continuous homomorphism $f\colon R\to S$ of adic perfectoid rings, the homomorphism $R^\flat \to S^\flat$ between the tiltings is denoted by $f^\flat$. 

\begin{lem} \label{lem:etperfd}
    Let $R$ be a complete adic perfectoid ring with an ideal of definition $p\in I \subset R$ and let $S$ be an $I$-adically complete $I$-completely \'{e}tale $R$-algebra. Then, $S$ is a complete adic perfectoid ring. 
\end{lem}
\begin{proof}
    We may assume that $I^\flat$ is $F$-standard. By the topological invariance of \'{e}tale sites, there is a unique \'{e}tale $W(R^\flat)/(p^n, [I^\flat]^n)$-algebra $S'_n$ such that $S'_n \otimes_{W(R^\flat)/(p^n, [I^\flat]^n)} (R/I) \cong S/IS$ for every $n\geq 0$. Then,  $S'_\infty = \lim_{n\geq 0} S'_n$ is $(p,[I^\flat])$-completely \'{e}tale over $W(R^\flat)$ and we have $S \cong S'_\infty/\xi_R S'_\infty$. By \cite[Lemma 2.18]{BS22}, $S'_\infty$ is equipped with the unique $\delta$-structure compatible with that of $W(R^\flat)$. Now, the Frobenius on $S'_\infty$ induces a homomorphism $S'_\infty/(p^n, [I^{\flat,1/p}]^{n}) \to S'_n$ after the base change to $W(R^\flat)/(p^n, [I^\flat]^n)$. It induces a universally homeomorphic \'{e}tale morphism, so it is an isomorphism. Thus, the Frobenius on $S'_\infty$ is an isomorphism, so $S \cong S'_\infty/\xi_R S'_\infty$ is a complete adic perfectoid ring. 
\end{proof}

\subsection{Continuous arc-topology}

In this section, we introduce continuous arc-topology on complete adic perfectoid rings and establish continuous arc-descent following \cite{BM21}, \cite[Section 8.2]{BS22} and \cite{Ito23}. The following continuous version was proposed in \cite[Section 2.2.1]{CS24}.  

\begin{defi}
    An adic homomorphism $R\to S$ of complete adic perfectoid rings is a continuous arc-cover if for every continuous homomorphism $R\to V$ to a valuation ring of rank $1$, there is an extension $V\to W$ of valuation rings of rank $1$ and a continuous homomorphism $S\to W$ lifting the composition $R\to V\to W$. 
\end{defi}

As in \cite[Remark 8.9]{BS22}, a continuous arc-cover can be constructed as a product of valuation rings of rank $1$ as follows. 

\begin{exa}
    Let $R$ be a complete adic perfectoid ring and let $I\subset R$ be a finitely generated ideal of definition. Let $X_R$ be a set of representatives of rank $1$ continuous valuations on $R$ and let $S=\prod_{x\in X_R} V_x$ be the product of $V_x$ with the $I$-adic topology. We can take each $V_x$ so that $V_x$ is $p$-complete and $\Frac(V_x)$ is algebraically closed. Then, $S$ is a complete adic perfectoid ring and $R\to S$ is a continuous arc-cover. 
\end{exa}

As usual, the continuous arc-site satisfies tilting equivalence. 

\begin{lem} \label{lem:tiltarc}
    Let $f\colon R\to S$ be a continuous homomorphism of complete adic perfectoid rings. Then, $f$ is a continuous arc-cover if and only if $f^\flat$ is a continuous arc-cover. 
\end{lem}
\begin{proof}
    Note that a $p$-adically complete valuation ring $V$ with an algebraically closed fraction field is perfectoid and $V^\flat$ is a perfect valuation ring with an algebraically closed fraction field. By \Cref{prop:conttiltequiv}, we see that for every continuous homomorphism $R \to V$, the existence of a lift $S\to W$ to an extension $V\to W$ is equivalent to the existence of a lift $S^\flat\to W^\flat$ to an extension $V^\flat\to W^\flat$. As in \Cref{prop:conttiltequiv}, $f$ is adic if and only if $f^\flat$ is adic.
\end{proof}

The argument of \cite{Ito23} for $\varpi$-complete arc-hyperdescent works in our setting as well. As in \cite[Definition 2.1]{Ito23}, hyperdescent for the continuous arc-topology is referred to as continuous arc-hyperdescent.  

\begin{prop} \label{prop:contarcdescent}
    The functors 
    \[ R \mapsto \Perf(R^\flat), \hspace{5pt} R \mapsto \Perf(W(R^\flat)), \hspace{5pt} R \mapsto \Perf(R) \]
    from the category of complete adic perfectoid rings to the $\infty$-category of small $\infty$-categories satisfy continuous arc-hyperdescent. 
\end{prop}
\begin{proof}
    As in the proof of \cite[Theorem 3.1]{Ito23}, it is enough to show that the bounded analogues of the given functors (where $\Perf$ is replaced by $\Perf_{[a,b]}$ for some $a\leq b$) satisfy arc-descent. 
    
    Let $R$ be a complete adic perfectoid ring. Let $I \subset R$ be a finitely generated ideal of definition such that $p\in I$ and let $f_1,\ldots,f_r\in I^\flat$ be a finite set of generators. Let $R\to S$ be a continuous arc-cover of complete adic perfectoid rings and let $R\to S^\bullet$ be the \v{C}ech nerve of the continuous arc-cover.

    First, we show the continuous arc-descent along $R\to S^\bullet$ of the functor $S^\bullet \mapsto \Perf_{[a,b]}(S^{\bullet,\flat})$. By \cite[Proposition 2.9 (2)]{Ito23}, it is enough to show that the arc-descent holds for the functor $S^\bullet \mapsto \Perf_{[a,b]}(S^{\bullet,\flat}/^{\bb{L}}(f_1^n,f_2^n, \ldots, f_r^n))$ for every $n\geq 0$. Since $R^\flat \to S^\flat$ is a continuous arc-cover by \Cref{lem:tiltarc}, $R^\flat \to S^\flat \times R^\flat[\tfrac{1}{f_1}] \times \cdots \times R^\flat[\tfrac{1}{f_r}]$ is an arc-cover of perfect rings. Thus, the claim follows from \cite[Theorem 3.1]{Ito23}. 

    The continuous arc-descent for $S^\bullet \mapsto \Perf_{[a,b]}(W(S^{\bullet,\flat}))$ follows from the same argument as in \cite[Section 4.1]{Ito23}. The continuous arc-descent for $S^\bullet \mapsto \Perf_{[a,b]}(S^\bullet/^{\bb{L}}(f_1^{\sharp,n},\ldots,f_r^{\sharp,n}))$ also follows from the same argument for every $n\geq 0$. Thus, the continuous arc-descent of the functor $S^\bullet \mapsto \Perf_{[a,b]}(S^\bullet)$ follows from \cite[Proposition 2.9 (2)]{Ito23}. 
\end{proof}

The following corollary follows formally from the continuous arc-descent of $R\mapsto \Perf(R)$ by looking at the space of endomorphisms of $R\in \Perf(R)$. 

\begin{cor} \label{cor:arcacyc}
    Let $R$ be a complete adic perfectoid ring with an ideal of definition $I\subset R$. Let $R\to S$ be a continuous arc-cover of complete adic perfectoid rings. The complex $0\to R \to S \to S\widehat{\otimes}_R S \to \cdots$ is acyclic. Here, $\widehat{\otimes}_R$ denotes the $I$-adically completed tensor product. 
\end{cor}

However, this descent is sometimes insufficient as it ignores the descent of topologies. We will show that the above descent holds as topological rings. 

\begin{defi}
    Let $R$ be a complete adic perfectoid ring. For an ideal of definition $I\subset R$ with $p\in I$, let $\vert \cdot \vert_I$ be the norm such that $\lvert x \rvert_I = p^{-n}$ with $n = \sup\{m \geq 0 \vert x\in I^m\}$, and let $\cl{M}_I(R)$ be the set of multiplicative norms on $R$ bounded by $\vert \cdot \vert_I$.
\end{defi}

\begin{lem} \label{lem:valuation_to_perfectoid_field}
    For every $\vert \cdot \vert \in \cl{M}_I(R)$, there is a continuous map $R \to V$ to a rank $1$ perfectoid valuation ring $V$ such that $\vert \cdot \vert$ is the pullback of a norm $\vert \cdot \vert_V$ on $V$. For any two such choices $R \to V_1$ and $R \to V_2$, there is a common refinement $R \to V_3$, in a sense that we have the following commutative diagram:
    \begin{center}
        \begin{tikzcd}
            & V_1 \ar[rd] & \\
            R \ar[ru] \ar[rd] \ar[rr] & & V_3. \\
            & V_2 \ar[ru] & 
        \end{tikzcd}
    \end{center}
\end{lem}
\begin{proof}
    Let $\mfr{p} = \supp \vert \cdot \vert \subset R$ and let $K = \Frac(R / \mfr{p})$. Then, $\vert \cdot \vert$ uniquely extends to a rank $1$ valuation on $K$. When the valuation on $K$ is trivial, $K$ is in characteristic $p$ and the map 
    \[
        R \to K \to (K\llbracket t \rrbracket)^{\perf, \wedge} = V
    \]
    satisfies the claim. When the valuation on $K$ is nontrivial, $V_0 = \vert \cdot \vert^{-1}([0,1]) \subset K$ is a rank $1$ valuation ring. We define $V$ as the completed perfection of $V_0$ when $V_0$ is in characteristic $p$, and as the completed algebraic closure of $V_0$ when $V_0$ is in mixed characteristic. Then, $R \to V$ satisfies the desired condition. 

    For the second claim, we apply \cite[Proposition 3.2 (1)]{Sch26} to $S = V_1 \otimes_R V_2$. Then, we can take a rank $1$ valuation $\vert \cdot \vert_S$ on $S$ extending the valuations on $V_1$ and $V_2$. By the previous argument for the first part, we get a continuous map $S \to V_3$ that realizes $\vert \cdot \vert_S$ as the restriction of the norm on $V_3$. 
\end{proof}

\begin{lem} \label{lem:tilting_equivalence_M(R)}
    Let $R$ be a complete adic perfectoid ring and let $I\subset R$ be an ideal of definition with $p\in I$. There is a homeomorphism
    \[
        \cl{M}_I(R) \cong \cl{M}_{I^\flat}(R^\flat) ,\quad
        \vert \cdot \vert \mapsto \vert \cdot \vert^\flat
    \]
    such that $\lvert f^\sharp \rvert = \lvert f \rvert^\flat$ for every $f \in R^\flat$. 
\end{lem}
\begin{proof}
    First, we show that the map $\vert \cdot \vert \mapsto \vert \cdot \vert^\flat$ is well-defined. First, since the map $f \mapsto f^\sharp$ is multiplicative, $\vert \cdot \vert^\flat$ is multiplicative. Moreover, for every $f \in I^\flat$, we have
    \[
        \vert f \vert^\flat = \vert f^\sharp \vert \leq p^{-1}, 
    \]
    so $\vert \cdot \vert^\flat$ is bounded by $\vert \cdot \vert_{I^\flat}$. Finally, since $\vert \cdot \vert$ is continuous and non-archimedean, we have
    \[
        \lvert (f+g)^\sharp \rvert = \lim_{n \to \infty} \lvert (f^{1/p^n})^\sharp + (g^{1 / p^n})^\sharp \rvert^{p^n} \leq \max(\lvert f^\sharp \rvert, \lvert g^\sharp \rvert).
    \]
    Thus, $\vert \cdot \vert^\flat$ is a multiplicative norm bounded by $\vert \cdot \vert_{I^\flat}$. 

    Thus, we get a map $\cl{M}_I(R) \to \cl{M}_{I^\flat}(R^\flat)$. It is continuous by construction. Since both spaces are compact and Hausdorff, it is enough to show that the map is bijective. 
    
    First, we show the surjectivity. By \Cref{lem:valuation_to_perfectoid_field}, every element of $\cl{M}_{I^\flat}(R^\flat)$ is obtained as the restriction along a continuous map $R^\flat \to V^\flat$ to a complete perfect rank $1$ valuation ring. By the tilting equivalence, we get a continuous map $R \to V$. Then, $V$ is a perfectoid rank $1$ valuation ring and it can be equipped with a norm $\vert \cdot \vert_V$ such that 
    \[
        \vert f^\sharp \vert_V = \vert f \vert_{V^\flat} ,\quad f \in V^\flat. 
    \]
    The induced norm on $R$ is multiplicative and bounded by $\vert \cdot \vert_I$ because for every $f \in I^\flat$, 
    \[
        \vert f^\sharp \vert_V = \vert f \vert_{V^\flat} \leq p^{-1}, 
    \]
    so $I = (I^\flat)^\sharp$ implies $\vert I \vert_V \subset [0, p^{-1}]$. By construction, $\vert \cdot \vert_V^\flat = \vert \cdot \vert_{V^\flat}$, so the surjectivity follows. 

    For the injectivity, choose two continuous maps
    \[
        R \to V_1 ,\quad R \to V_2
    \]
    to rank $1$ perfectoid valuation rings such that their tilts
    \[
        R^\flat \to V_1^\flat ,\quad R \to V_2^\flat
    \]
    induce the same norm on $R^\flat$. By the second claim of \Cref{lem:valuation_to_perfectoid_field}, we can take a common refinement $R^\flat \to V_3^\flat$ of their tilts. Then, its untilt $R \to V_3$ is a common refinement of the original two maps. Thus, they induce the same norm on $R$, which proves the injectivity. 
\end{proof} 

Let $\vert \cdot \vert_I^\spc$ denote the spectral norm
\[
    \lvert x \rvert^\spc_I = \lim_{n\to \infty} \lvert x^n \rvert_I^{1/n}. 
\]
Then, $\vert \cdot \vert_I^\spc$ is equal to the supremum over $\cl{M}_I(R)$ (see e.g. \cite[Lemma 1.5.22]{Ked19}). 

\begin{lem} \label{lem:preservation_spectral_norms}
    Let $R\to S$ be a continuous arc-cover of complete adic perfectoid rings and let $I \subset R$ be an ideal of definition with $p \in I$. Then, $\lvert x \rvert_I^\spc = \lvert x \rvert_{IS}^\spc$ for every $x \in R$. 
\end{lem}
\begin{proof}
    By \Cref{lem:valuation_to_perfectoid_field}, each $\vert \cdot \vert \in \cl{M}_I(R)$ corresponds to a continuous homomorphism $R\to V$ to a perfectoid rank $1$ valuation ring. Since $R\to S$ is a continuous arc-cover, we may assume by replacing $V$ that it admits a lift $S\to V$. Since the norm induced by $S\to V$ is bounded by $\vert \cdot \vert_{IS}$ if and only if the one induced by $R\to V$ is bounded by $\vert \cdot \vert_{I}$ (as they are equivalent to the condition that the valuations on $I$ are bounded by $p^{-1}$), we get $\lvert f \rvert_I^\spc = \lvert f \rvert_{IS}^\spc$ by taking the supremum over $\cl{M}_I(R)$ and $\cl{M}_{IS}(S)$. 
\end{proof}

\begin{prop} \label{prop:perfduniform}
    Let $R$ be a complete adic perfectoid ring and let $I\subset R$ be an ideal of definition with $p\in I$. The topology defined by $\vert \cdot \vert_I^\spc$ equals the $I$-adic topology on $R$. 
\end{prop}
\begin{proof}
    First, suppose that $R$ is in characteristic $p$. Let $J$ be the minimum $F$-standard ideal containing $I$. Since $\lvert x \rvert^\spc_I = \lim_{n\to \infty} \lvert x^{p^n} \rvert_I^{p^{-n}}$ and $J=\bigcup_{n \geq 0}\Frob^{-n}(I^{p^n})$, we see that for every $m \geq 1$, $x\in J^m$ implies $\lvert x \rvert^\spc_I \leq p^{-m}$, and $\lvert x \rvert^\spc_I \leq p^{-m-1}$ implies $x\in J^m$. Thus, the topology defined by $\vert \cdot \vert_I^\spc$ is equal to the $J$-adic topology on $R$. Thus, the claim follows from \cite[Lemma 2.5]{Tak25}. 

    Next, let $R$ be a complete adic perfectoid ring such that the image of the map $(-)^\sharp \colon R^\flat \to R$ is dense. It is obvious that the $I$-adic topology is finer than or equal to the topology defined by $\vert \cdot \vert_I^\spc$. We show the converse direction. Let $m\geq 1$ and let $J_m\subset R$ be the ideal consisting of elements $x\in R$ with $\lvert x \rvert^\spc_{I} \leq p^{-m}$. Since $p^m \in J_m$, it is enough to determine the set of $f\in R^\flat$ with $f^\sharp \in J_m$. Recall that $\vert \cdot \vert_I^\spc$ equals the supremum over $\cl{M}_I(R)$. Then, by \Cref{lem:tilting_equivalence_M(R)}, $f^\sharp \in J_m$ if and only if $\lvert f \rvert_{I^\flat}^\spc \leq p^{-m}$. Let $J_m^\flat \subset R^\flat$ be the ideal consisting of elements $x\in R^\flat$ with $\lvert x \rvert_{I^\flat}^\spc \leq p^{-m}$. Since $\vert \cdot \vert_{I^\flat}^\spc$ is power-multiplicative, $J_m^\flat$ is $F$-standard. The above argument shows that $J_m = (p^m, (J_m^\flat)^\sharp)$ and the argument in characteristic $p$ shows that the linear topology given by $\{J_m^\flat \}$ is equal to the $I^\flat$-adic topology. Since $I = (p,(I^\flat)^\sharp)$, we see that the linear topology given by $\{J_m\}_{m \geq 0}$ is equal to the $I$-adic topology. 

    Now, let $R$ be an arbitrary complete adic perfectoid ring. By Andr\'{e}'s flatness lemma (see \cite[Theorem 7.14]{BS22}), there is a $p$-completely faithfully flat map $R \to S$ of perfectoid rings such that the map $(-)^\sharp \colon S^\flat \to S$ is surjective. Let $S^\wedge$ be the $I$-adic completion of $S$. By \Cref{cor:tiltofcomp}, $S^\wedge$ is a complete adic perfectoid ring, and $R \to S^\wedge$ is $I$-completely faithfully flat. In particular, $R \to S^\wedge$ is a continuous arc-cover (see \cite[p.73 (2)]{CS24}) and the restriction of the $IS^\wedge$-adic topology to $R$ is equal to the $I$-adic topology. By \Cref{lem:preservation_spectral_norms}, the restriction of the topology given by $\vert \cdot \vert_{IS}^\spc$ to $R$ is equal to the topology given by $\vert \cdot \vert_{I}^\spc$. Since the claim holds for $S^\wedge$ by the previous argument, we get the claim for $R$ by restricting along $R \to S^\wedge$. 
\end{proof}

\begin{cor} \label{cor:arcsubsp}
    Let $R\to S$ be a continuous arc-cover of complete adic perfectoid rings. The topology on $R$ equals the subspace topology induced by the topology on $S$. 
\end{cor}
\begin{proof}
    The claim follows from \Cref{lem:preservation_spectral_norms} and \Cref{prop:perfduniform}. 
\end{proof}

\subsection{Perfectoid formal schemes}


\begin{defi}
    A perfectoid formal scheme is a formal scheme Zariski locally isomorphic to the formal spectrum of a complete adic perfectoid ring. A perfectoid formal scheme in characteristic $p$ is called a perfect formal scheme. 
\end{defi}

\begin{prop}
    A formal scheme $\mfr{X}$ is perfectoid if and only if every affine open formal subscheme of $\mfr{X}$ is the formal spectrum of a complete adic perfectoid ring. 
\end{prop}

\begin{proof}
    It is enough to show that if $\mfr{X}=\Spf(R)$ with $R$ a complete adic ring is perfectoid, then $R$ is a complete adic perfectoid ring. Let $I\subset R$ be a finitely generated ideal of definition such that $p\in I$. Since $\Spf(R)$ is perfectoid, $\Spec(R/I)$ is Zariski locally semiperfect, so $R/I$ is semiperfect. Note that a ring $A$ in characteristic $p$ is semiperfect if the Frobenius of $A$ is surjective, that is if the Frobenius of $\Spec(A)$ is a closed immersion, so this condition is Zariski local. 
    
    Let $R^\flat=\lim_{x\mapsto x^p} (R/I)$ and let $I^\flat=\Ker(R^\flat \twoheadrightarrow R/I)$. Then, $R^\flat$ is perfect and complete with respect to the linear topology given by $\{\Frob^n(I^\flat) \}_{n\geq 0}$. The quotient $R^\flat/\Frob(I^\flat) \to R^\flat/I^\flat$ is isomorphic to the Frobenius map $R/I \to R/I$. 
    
    For each affine open formal subscheme $\Spf(S) \hookrightarrow \Spf(R)$ with $S$ a complete adic perfectoid ring, we have $S/IS \cong S^\flat/(IS)^\flat$ and the kernel of the Frobenius map on $S^\flat/(IS)^\flat$ is finitely generated since $(IS)^\flat$ is finitely generated. Thus, $\Ker(\Frob\colon R/I \to R/I)$ is Zariski locally finitely generated, so it is finitely generated. In particular, there is a finitely generated ideal $J \subset R^\flat$ such that $I^\flat=J+\Frob(I^\flat)$. Then, for every $x \in I^\flat$, there is a sequence $y_1, y_2,\ldots \in J$ such that 
    \[
        x - (y_1 + y_2^p + \cdots + y_k^{p^{k-1}}) \in \Frob^k(I^\flat). 
    \]
    Since $R^\flat$ is complete with respect to the linear topology given by $\{\Frob^n(I^\flat) \}_{n\geq 0}$, we have $x = \sum_{i \geq 1} y_i^{p^{i-1}}$ and the sum converges in $J$ since $J$ is finitely generated. Thus, $I^\flat = J$ and $I^\flat$ is finitely generated. In particular, $R^\flat$ is $I^\flat$-adically complete. 

    We can take a finite Zariski open covering
    \[
        \Spf(R) = \bigcup_{\lambda \in \Lambda} \Spf(S_\lambda)
    \]
    so that $S_\lambda$ is a complete adic perfectoid ring and $\Spf(S_\lambda) = D_{f_\lambda^\sharp}(\Spf(R))$ for some $f_\lambda \in R^\flat$. This is possible since $D_{f^\sharp}(\Spf(R)) \cong D_{f}(\Spf(R^\flat))$ under $\lvert \Spf(R) \rvert \cong \lvert \Spf(R^\flat) \rvert$ and the set of $D_{f^\sharp}(\Spf(R))$ forms an open basis of $\lvert \Spf(R) \rvert$. Then, $\Spf(S_\lambda^\flat) = D_{f_\lambda}(\Spf(R^\flat))$, $(IS_\lambda)^\flat=I^\flat S_\lambda^\flat$ and $S_\lambda^\flat$ is a complete $I^\flat$-adic perfect ring. Now, $R^\flat \to \prod_{\lambda\in \Lambda} S_\lambda^\flat$ is a continuous arc-cover of complete adic perfect rings. By \Cref{prop:contarcdescent}, the homomorphism $\Ker(\theta_{S_\lambda}) \to W(S_\lambda^\flat)$ of invertible $W(S_\lambda^\flat)$-modules descends to a homomorphism $K\to W(R^\flat)$ with $K$ an invertible $W(R^\flat)$-module.   

    Now, we have a surjection $W(R^\flat) \twoheadrightarrow R$. The following diagram commutes for each $\lambda\in \Lambda$ by the functoriality of the construction. 
    \begin{center}
        \begin{tikzcd}
            K \ar[r]\ar[d] & W(R^\flat) \ar[r] \ar[d] & R \ar[d] \\
            \Ker(\theta_{S_\lambda})\ar[r] & W(S_\lambda^\flat) \ar[r] & S_\lambda 
        \end{tikzcd}
    \end{center}
    Now, $K\to \prod_{\lambda\in \Lambda} \Ker(\theta_{S_\lambda})$ and $W(R^\flat) \to \prod_{\lambda\in \Lambda} W(S_\lambda^\flat)$ are injective by \Cref{prop:contarcdescent}, and $R\to \prod_{\lambda\in \Lambda} S_\lambda$ is also injective by the sheaf condition. Moreover, the images of these maps can be described by gluing conditions. Thus, we see that $K$ maps isomorphically to $\Ker(W(R^\flat) \to R)$ in $W(R^\flat)$. It is enough to show that $K$ is generated by a distinguished element of $W(R^\flat)$.
    
    We proceed as in \cite[Lemma 3.6]{BS22}. Let $J\subset R^\flat$ be an $F$-standard ideal of definition and let $M=W(R^\flat)/(K^p + \varphi(K)\cdot W(R^\flat)+(p^2,[J]))$. Since $R^\flat \to \prod_{\lambda \in \Lambda} S^\flat_\lambda$ is $I^\flat$-completely faithfully flat, $W(R^\flat) \to \prod_{\lambda \in \Lambda} W(S_\lambda^\flat)$ is $(p,[J])$-completely faithfully flat. In particular, $M\to \prod_{\lambda\in \Lambda} M \otimes_{W(R^\flat)} W(S^\flat)$ is faithfully flat. Since $\Ker(\theta_{S_\lambda})$ is generated by a distinguished element, $p=0$ in $M \otimes_{W(R^\flat)} W(S^\flat)$ for every $\lambda \in \Lambda$ by \cite[Theorem 3.10]{BS22}. Thus, we have $p=0$ in $M$. Let $p = a + \varphi(b)c + p^2d + e$ with $a\in K^p$, $b \in K$, $c,d \in W(R^\flat)$ and $e\in [J]$. We have $p(1 - \delta(b)c - pd) - e \in K^p$. We show that $\delta(b)$ is a unit. If $\delta(b)$ is not a unit, there is a maximal ideal $J \subset \mfr{m}\in R^\flat$ such that $\delta(b) \in (p,[\mfr{m}])$. The map $R^\flat \to R^\flat/\mfr{m}$ factors through $S_\lambda^\flat$ for some $\lambda\in \Lambda$, so the image of $K$ in $W(R^\flat/\mfr{m})$ is generated by $p$. Since $\delta(b)$ maps to $p\cdot W(R^\flat/\mfr{m})$ and $e$ maps to $0$ in $W(R^\flat/\mfr{m})$, we have $p(1-pd)\in (p^2)$ in $W(R^\flat/\mfr{m})$, but this is a contradiction. Thus, $\delta(b)$ is a unit and $b\in K$ is a distinguished element. Then, $\Ker(\theta_{S_\lambda})$ is generated by $b$ for every $\lambda \in \Lambda$, and we see that $K$ is generated by $b$. 
\end{proof}

The above proof also justifies the following construction. 

\begin{defi}
    Let $\mfr{X}$ be a perfectoid formal scheme. The tilting of $\mfr{X}$ is a perfect formal scheme $\mfr{X}^\flat$ such that $\lvert \mfr{X}^\flat \rvert = \lvert \mfr{X} \rvert$ and $\cl{O}_{\mfr{X}^\flat}(U)=\cl{O}_{\mfr{X}}(U)^\flat$ for every affine open subspace $U\subset \lvert \mfr{X} \rvert$. 
\end{defi}

The tilting equivalence naturally extends to the category of perfectoid formal schemes. 

\begin{prop} \label{prop:tiltperfdfmsch}
    Let $R$ be a complete adic perfectoid ring. The category of perfectoid formal schemes over $\Spf(R)$ is equivalent to the category of perfect formal schemes over $\Spf(R^\flat)$. For a perfectoid formal scheme $\mfr{X}$ over $\Spf(R)$ with a tilting $\mfr{X}^\flat$ over $\Spf(R^\flat)$, there is an isomorphism $\mfr{X} \times_{\Spf(R)} \Spec(R/I) \cong \mfr{X}^\flat \times_{\Spf(R^\flat)} \Spec(R^\flat/I^\flat)$ for every ideal of definition $I\subset R$ with $p\in I$. Moreover, there is an isomorphism $\mfr{X}^\diamond \cong (\mfr{X}^\flat)^\diamond$. 
\end{prop}
\begin{proof}
    The first claim follows since the functor $\mfr{X} \mapsto \mfr{X}^\flat$ is an equivalence by \Cref{prop:conttiltequiv}. For the remaining claims, we may assume that $\mfr{X}=\Spf(S)$. Since $\mfr{X}^\flat = \Spf(S^\flat)$ and $S/IS \cong S^\flat/(IS)^\flat$, we have $\mfr{X} \times_{\Spf(R)} \Spec(R/I) \cong \mfr{X}^\flat \times_{\Spf(R^\flat)} \Spec(R^\flat/I^\flat)$. We have an isomorphism $\Spd(S)\cong \Spd(S^\flat)$ by \Cref{prop:conttiltequiv}.
\end{proof}

In the above situation, we say that $\mfr{X}$ is the untilt of $\mfr{X}^\flat$ over $R$. Since $\mfr{X}^\diamond \cong (\mfr{X}^\flat)^\diamond$, $\mfr{X}$ is thick if and only if $\mfr{X}^\flat$ is thick. 

\subsection{Fully faithfulness of $v$-sheafification} \label{ssc:perfdff}
In this section, we study the fully faithfulness of the $v$-sheafification functor. 
First, we recall the following result. 

\begin{prop}\textup{(\cite[Proposition 18.3.1]{SW20})}
    The functor $X\mapsto X^\diamond$ from perfect schemes of characteristic $p$ to small $v$-sheaves on $\Perf$ is fully faithful. 
\end{prop}

We extend this result to the category of perfectoid formal schemes. For this, we apply continuous arc-descent (\Cref{prop:contarcdescent}). 

\begin{lem} \label{lem:prodarcpts}
    Let $R$ be a complete adic perfectoid ring and let $I\subset R$ be a finitely generated ideal of definition. There exists a set of continuous homomorphisms $\{R\to V_\lambda\}_{\lambda \in \Lambda}$ with $V_\lambda$ an $I$-adically complete valuation ring with an algebraically closed fraction field such that the complete $I$-adic ring $S=\prod_{\lambda \in \Lambda} V_\lambda$ satisfies that $R\to S$ is a continuous arc-cover and $\Spd(S)\to \Spd(R)$ is a $v$-cover. 
\end{lem}
\begin{proof}
    For each point $x\in \lvert \Spd(R) \rvert$, take a continuous homomorphism $R\to V_x$ with $V_x$ a $p$-complete valuation ring with an algebraically closed fraction field such that $x$ is the image of the closed point along $\Spd(V_x[\tfrac{1}{\varpi}],V_x)\to \Spd(R)$ with $\varpi \in V_x$ a pseudo-uniformizer. Let $S=\prod_{x \in \lvert \Spd(R) \rvert} V_x$. The map $\lvert \Spd(S) \rvert \to \lvert \Spd(R) \rvert$ is surjective, so $\Spd(S)\to \Spd(R)$ is surjective by \cite[Lemma 2.26]{Gle24} and \cite[Lemma 12.11]{Sch17}. Note that $R\to S$ is automatically a continuous arc-cover since each homomorphism $R \to V$ to a valuation ring $V$ of rank $1$ corresponds to a point $\Spd(V[\tfrac{1}{\varpi}],V)\to \Spd(R)$ with $\varpi \in V$ a pseudo-uniformizer. 
\end{proof}

\begin{prop}\label{lem:perfdrep}
    Let $\mfr{X}$ be a perfectoid formal scheme and let $\mfr{Y}$ be a formal scheme over $\bb{Z}_p$.
    The map $\mrm{Hom}_{\Spf(\bb{Z}_p)}(\mfr{X},\mfr{Y})\to \mrm{Hom}_{\Spd (\bb{Z}_p)}(\mfr{X}^\diamond,\mfr{Y}^\diamond)$ is bijective. 
\end{prop}
\begin{proof}
    By following the argument given in the proof of \cite[Theorem 2.16]{AGLR22}, we may assume that $\mfr{X}$ and $\mfr{Y}$ are affine. Let $\mfr{X}=\Spf(R)$ and $\mfr{Y}=\Spf(A)$ with $R$ a complete adic perfectoid ring. Let $f\colon \Spd(R)\to \Spd(A)$ be an arbitrary map. Let $I\subset R$ be a finitely generated ideal of definition and take a continuous arc-cover $R\to S=\prod_{\lambda \in \Lambda} V_\lambda$ as in \Cref{lem:prodarcpts}. For each $\lambda\in \Lambda$, $\vert \cdot \vert_{\lambda}$ denotes the continuous valuation on $R$ corresponding to $V_\lambda$. Let $f_1,\ldots, f_r \in R$ be a finite set of generators of $I$. Take a decomposition $\Lambda = \coprod_{0\leq i \leq r} \Lambda_i$ so that $\vert \cdot \vert_\lambda$ is non-analytic for $\lambda \in \Lambda_0$ and $\lvert f_j \rvert_\lambda \leq \lvert f_i \rvert_\lambda \neq 0$ for $1\leq i,j \leq r$ and $\lambda \in \Lambda_i$. Let $S_i = \prod_{\lambda \in \Lambda_i} V_\lambda$. First, the composition $\Spd(S_0)\to \Spd(R) \to \Spd(A)$ is represented by a unique homomorphism $f_0\colon A\to A_\red^\perf \to R_\red \to S_0$ by \cite[Proposition 18.3.1]{SW20}. For $1 \leq i \leq r$, the composition $\Spd(S_i[\tfrac{1}{f_i}], S_i) \to \Spd(S_i) \to \Spd(R) \to \Spd(A)$ is represented by a unique continuous homomorphism $f_i\colon A\to S_i$. By \cite[Proposition 4.9]{Gle24}, $f_i$ represents the composition $\Spd(S_i)\to \Spd(R) \to \Spd(A)$. The product $(f_i)_{0\leq i \leq r} \colon A\to \prod_{0\leq i \leq r} S_i = S$ represents the composition $\Spd(S) \to \Spd(R) \to \Spd(A)$. Thus, we see that $\Spd(S)\to \Spd(R) \to \Spd(A)$ is represented by a unique continuous homomorphism $A\to S$. 

    Take a continuous arc-cover $g\colon S \widehat{\otimes}_R S \to S'$ as in \Cref{lem:prodarcpts}. Let $i_1, i_2\colon S \to S\widehat{\otimes}_R S$ be homomorphisms such that $i_1(s) = s\otimes 1$ and $i_2(s) = 1\otimes s$ for $s\in S$. Since the composition $\Spd(S') \to \Spd(R) \to \Spd(A)$ is represented by a unique continuous homomorphism $A \to S'$, we have $g\circ i_1 \circ (f_i) = g \circ i_2 \circ (f_i)$. Since $g$ is injective and the equalizer of $i_1$ and $i_2$ is $R$ by \Cref{cor:arcacyc}, we see that the product $(f_i)$ factors through $R\subset S$. By \Cref{cor:arcsubsp}, $(f_i)$ induces a continuous homomorphism $A\to R$ and we see that it represents $f$ by the surjectivity of $\Spd(S) \to \Spd(R)$. 
\end{proof}

\subsection{Analytic locus}

In this section, we study analytic loci of perfectoid formal schemes and prove a perfectoid analogue of \cite[Proposition 3]{Lou17}. As in the analytic case \cite[Lemma 6.5]{Sch12}, we need the following approximation lemma.

\begin{lem} \textup{(\cite[Lemma 5.5, 5.16]{Ked13})} \label{lem:approximation_lemma}
    Let $R$ be a complete adic perfectoid ring. For every $f \in R$ and $k \geq 0$, there is an expression $f = g^\sharp + p h$ with $g \in R^\flat$ and $h \in R$ such that for every $\vert \cdot \vert \in \Spa(R)$, we have 
    \[
        \lvert h \rvert \leq \lvert g^\sharp \rvert \quad \text{or} \quad
        \lvert h \rvert \leq \lvert \xi_{R, 0}^\sharp \rvert^k. 
    \]
\end{lem}
\begin{proof}
    The same argument as in loc. cit. works here. For the sake of completeness, we include the proof here with our notation. 

    Take an element $x_0 \in W(R^\flat)$ such that $\theta_R(x) = f$ and let $u = (p^{-1}(\xi_R - [\xi_{R, 0}]))^{-1}$. We construct a sequence $x_0, x_1, \ldots$ as follows: if we write $x_i = \sum_{j \geq 0} p^j [x_{i, j}]$, then we set
    \begin{equation} \label{eq:definition_xi}
        x_{i + 1} = x_i - p^{-1}(x_i - [x_{i, 0}])\cdot u \xi_R = [x_{i, 0}] - p^{-1}(x_i - [x_{i, 0}]) \cdot u [\xi_{R, 0}]. 
    \end{equation}
    Here, the second equality follows from the definition of $u$. For every $\vert \cdot \vert \in \Spa(R^\flat)$, we will show that for every $i \geq 0$, we have 
    \begin{equation} \label{eq:xij_condition}
        \lvert x_{i, j} \rvert \leq \lvert x_{i, 0} \rvert \quad (\forall j > 0)
        \quad \text{or} \quad 
        \lvert x_{i, j} \rvert \leq \lvert \xi_{R, 0} \rvert^i \quad (\forall j > 0). 
    \end{equation}
    Note that the latter holds automatically for $i = 0$. 

    Let $N \geq 0$ be the minimum integer (or $\infty$) such that $\lvert x_{N, 0} \rvert > \lvert \xi_{R, 0} \rvert^{N + 1}$. For $0 \leq i \leq N$, \eqref{eq:definition_xi} implies $ \lvert x_{i, j} \rvert \leq \lvert \xi_{R, 0} \rvert^i$ by induction on $i$. Then, \eqref{eq:definition_xi} implies $\lvert x_{N + 1, j} \rvert \leq \lvert x_{N + 1, 0} \rvert$. If $\lvert x_{i, j} \rvert \leq \lvert x_{i, 0} \rvert$ holds for some $i$, then \eqref{eq:definition_xi} implies that the same holds for $i + 1$. Thus, \eqref{eq:xij_condition} holds for every $i \geq 0$. 

    Now, for every $\vert \cdot \vert \in \Spa(R)$, we can take $\vert \cdot \vert^\flat \in \Spa(R^\flat)$ such that 
    $
        \lvert f \rvert^\flat = \lvert f^\sharp \rvert 
    $
    for $f \in R^\flat$ (cf.\ \Cref{lem:tilting_equivalence_M(R)}). We show that the claim holds if we set 
    \[
        g = x_{k, 0} ,\quad 
        h = \sum_{j \geq 0} p^j x_{k, j + 1}^\sharp. 
    \]
    If $\lvert x_{k, j} \rvert^\flat \leq \lvert x_{k, 0} \rvert^\flat$, we have $\lvert h \rvert \leq \lvert g^\sharp \rvert$. Otherwise, $\lvert x_{k, j} \rvert^\flat \leq \lvert \xi_{R, 0} \rvert^{\flat, k}$, so we get $\lvert h \rvert \leq \lvert \xi_{R, 0}^\sharp \rvert^k$. 
\end{proof}

\begin{lem} \label{lem:Jsharpn}
    Let $R$ be a complete adic perfectoid ring and let $f_1,\ldots,f_r$ be elements of $R^\flat$ such that $(f_1,\ldots,f_r)\subset R^\flat$ is open and $(f_1^\sharp, \ldots, f_r^\sharp)\subset R$ is open. Let $J\subset R^\flat$ be the minimum $F$-standard ideal containing $f_1,\ldots,f_r$. We have $(f_1^\sharp,\ldots, f_r^\sharp)^r\cdot (J^\sharp)^n \subset (f_1^\sharp,\ldots, f_r^\sharp)^n$ for every $n\geq 0$.
\end{lem}
\begin{proof}
    As in \cite[Definition 2.3]{Tak25}, an element of $J$ can be written as a linear combination of elements $f_1^{a_1}\cdots f_r^{a_r}$ with $a_i \in \bb{Z}[\tfrac{1}{p}]_{\geq 0}$ and $\sum_{1\leq i \leq r} a_i = 1$. Thus, as in \Cref{lem:prodFstd}, $[J]$ is the $p$-adic closure of the ideal of $W(R^\flat)$ generated by $[f_1]^{a_1}\cdots [f_r]^{a_r}$ with $a_1,\ldots, a_r$ as above. By \Cref{lem:Jsharp}, $J^\sharp$ is generated by elements $(f_1^{\sharp})^{a_1}\cdots (f_r^{\sharp})^{a_r}$ with $a_1,\ldots, a_r$ as above. Thus, $(J^\sharp)^n$ is generated by elements $(f_1^{\sharp})^{a_1}\cdots (f_r^{\sharp})^{a_r}$ with $a_i \in \bb{Z}[\tfrac{1}{p}]_{\geq 0}$ and $\sum_{1\leq i \leq r} a_i = n$. Since we have $\sum_{1\leq i \leq r} [a_i] \geq n-r$, we see that $(J^\sharp)^n \subset (f_1^\sharp, \ldots, f_r^\sharp)^{n-r}$ and the claim holds. 
\end{proof}

\begin{prop} \label{prop:Tatedomisperfd}
    Let $R$ be a complete adic perfectoid ring. For every Tate rational domain $U(\tfrac{T_1,\ldots,T_r}{s}) \subset \Spa(R)$, $(R\langle \tfrac{T_1,\ldots,T_r}{s} \rangle, R\langle \tfrac{T_1,\ldots,T_r}{s} \rangle^+)$ is a perfectoid Huber pair. 
\end{prop}
\begin{proof}
    First, we show that $s,T_1,\ldots,T_r$ can be taken inside the image of $(-)^\sharp\colon R^\flat \to R$. We may suppose that $T_1 = \xi_{R, 0}^{\sharp, n}$ for some $n \geq 0$. Then, we apply \Cref{lem:approximation_lemma} to $f = s$ with $k = n$ and let $s^\flat = g$. It is easy to check $U(\tfrac{T_1,\ldots, T_r}{s}) = U(\tfrac{T_1,\ldots, T_r}{(s^\flat)^\sharp})$ using the property of $s - (s^\flat)^\sharp$. In the same way, we apply \Cref{lem:approximation_lemma} to $f = T_2,\ldots, T_r$ and get $T_2^\flat, \ldots, T_r^\flat$. We set $T_1^\flat = \xi_{R, 0}^{n}$. Then, it is easy to see that $U(\tfrac{T_1,\ldots, T_r}{s}) = U(\tfrac{(T_1^\flat)^\sharp,\ldots, (T_r^\flat)^\sharp}{(s^\flat)^\sharp})$. Note that since $(T_i^\flat)^\sharp \equiv T_i \pmod{p}$ and $T_1 = \xi_{R, 0}^{\sharp, n}$, $(T_1^\flat)^\sharp,\ldots, (T_r^\flat)^\sharp$ generate an open ideal. Similarly, $T_1^\flat,\ldots, T_r^\flat$ generate an open ideal. From now on, we assume $s = (s^\flat)^\sharp$ and $T_i = (T_i^\flat)^\sharp$. 

    First, we show that $(R^\flat\langle \tfrac{T_1^\flat,\ldots,T_r^\flat}{s^\flat} \rangle, R^\flat\langle \tfrac{T_1^\flat,\ldots,T_r^\flat}{s^\flat} \rangle^+)$ is a complete perfect Tate Huber pair. By \Cref{prop:perfduniform}, the inverse of the Frobenius on $R^\flat$ is continuous, so $R^\flat \xrightarrow{x\mapsto x^{1/p}} R^\flat \to R^\flat \langle  \tfrac{T_1^\flat,\ldots,T_r^\flat}{s^\flat} \rangle$ factors through $R^\flat \to R^\flat \langle  \tfrac{T_1^\flat,\ldots,T_r^\flat}{s^\flat} \rangle$ by the universality of $R^\flat \langle  \tfrac{T_1^\flat,\ldots,T_r^\flat}{s^\flat} \rangle$. It gives the inverse to the Frobenius on $R^\flat \langle  \tfrac{T_1^\flat,\ldots,T_r^\flat}{s^\flat} \rangle$ and we see that $R^\flat \langle  \tfrac{T_1^\flat,\ldots,T_r^\flat}{s^\flat} \rangle$ is perfect. In particular,  $(R^\flat\langle \tfrac{T_1^\flat,\ldots,T_r^\flat}{s^\flat} \rangle, R^\flat\langle \tfrac{T_1^\flat,\ldots,T_r^\flat}{s^\flat} \rangle^+)$ is uniform by \cite[Proposition 6.1.6]{SW20}. 

    Next, we show that $(R\langle \tfrac{T_1,\ldots,T_r}{s} \rangle, R\langle \tfrac{T_1,\ldots,T_r}{s} \rangle^+)$ is uniform. The $s$-adic completion $R[\tfrac{T_1,\ldots,T_r}{s}]^\wedge$ of $R[\tfrac{T_1,\ldots,T_r}{s}]$ is a ring of definition of $R\langle \tfrac{T_1,\ldots,T_r}{s} \rangle$. It is enough to show that $s^k a \in R[\tfrac{T_1,\ldots,T_r}{s}]^\wedge$ holds for every $a\in R\langle \tfrac{T_1,\ldots,T_r}{s} \rangle^+$ for some $k\geq 0$. Since $R[1/s]$ is dense in $R\langle \tfrac{T_1,\ldots,T_r}{s} \rangle$, we may assume $a \in R[1/s]$ and let $a = b_0 / s^m$ with $b_0 \in R$. 

    Now, we will reduce to the case where $b_0 = (b_0^\flat)^\sharp$ for some $b_0^\flat \in R^\flat$. Since $a \in  R\langle \tfrac{T_1,\ldots,T_r}{s} \rangle^+$, $\lvert b_0 \rvert \leq \lvert s \rvert^m$ for every $\vert \cdot \vert \in U(\tfrac{T_1,\ldots,T_r}{s})$. We apply \Cref{lem:approximation_lemma} to $f = b_0$ with $k = nm$ and let $b_0^\flat = g$ and $b_1 = h$. It is easy to check that $\lvert (b_0^\flat)^\sharp \rvert, \lvert b_1 \rvert \leq \lvert s \rvert^m$ for every $\vert \cdot \vert \in U(\tfrac{T_1,\ldots,T_r}{s})$. By repeating this argument to $b_1$ inductively, we get an expression
    \[
        a = ((b_0^\flat)^\sharp + p (b_1^\flat)^\sharp + \cdots + p^{nm} b_{nm}) / s^m
    \] 
    with $(b_i^\flat)^\sharp / s^m \in R\langle \tfrac{T_1,\ldots,T_r}{s} \rangle^+$. Since $p^{nm} b_{nm}/ s^m \in R[\tfrac{T_1,\ldots,T_r}{s}]$, we are reduced to the case $a = b^\sharp / s^m$ for some $b \in R^\flat$. 
    
    Now, the tilting equivalence $\Spd(R) \cong \Spd(R^\flat)$ restricts to a homeomorphism $U(\tfrac{T_1,\ldots,T_r}{s}) \cong U(\tfrac{T_1^\flat,\ldots,T_r^\flat}{s^\flat})$ and $\lvert b \rvert \leq \lvert (s^\flat)^n \rvert$ for every $\vert \cdot \vert \in U(\tfrac{T_1^\flat,\ldots,T_r^\flat}{s^\flat})$. Then, $b/(s^\flat)^n\in R^\flat\langle \tfrac{T_1^\flat,\ldots,T_r^\flat}{s^\flat} \rangle^+$. Since $R^\flat\langle \tfrac{T_1^\flat,\ldots,T_r^\flat}{s^\flat} \rangle$ is uniform, we have $(s^\flat)^kb/(s^\flat)^n\in R^\flat[\tfrac{T_1^\flat,\ldots,T_r^\flat}{s^\flat}]$ for some $k\geq 0$ independent of $a$. Thus, by replacing $b$ and $n$, we may assume $(s^\flat)^kb \in (s^\flat,T_1^\flat,\ldots,T_r^\flat)^n$. Let $J\subset R^\flat$ be the minimum $F$-standard ideal containing $(s^\flat,T_1^\flat,\ldots,T_r^\flat)$. By \Cref{lem:prodFstd}, $[(s^\flat)^kb]$ lies in the $p$-adic closure of $[J]^n$ in $W(R^\flat)$. It follows that $s^kb^\sharp \in (J^\sharp)^n$. By \Cref{lem:Jsharpn}, we have $s^{k+r+1}b^\sharp \in (s,T_1,\ldots, T_r)^n$. It follows that $s^{k+r+1}b^{\sharp}/s^n \in R[\tfrac{T_1,\ldots,T_r}{s}]$. 

    Finally, we show that $(R\langle \tfrac{T_1,\ldots,T_r}{s} \rangle, R\langle \tfrac{T_1,\ldots,T_r}{s} \rangle^+)$ is a perfectoid Huber pair. Let $J\subset R^\flat$ be the minimum $F$-standard ideal containing $s^\flat,T_1^\flat, \ldots, T_r^\flat$. Then, $R^\flat[J/s^\flat]\subset R^\flat[1/s^\flat]$ is perfect and its $s^\flat$-adic completion $R^\flat[J/s^\flat]^\wedge$ is a ring of definition of $R^\flat\langle \tfrac{T_1^\flat,\ldots,T_r^\flat}{s^\flat} \rangle$. Moreover, $W(R^\flat[J/s^\flat]) \subset W(R^\flat[1/s^\flat])$ is the $p$-adic completion of $W(R^\flat)[[J]/[s^\flat]] \subset W(R^\flat)[1/[s^\flat]]$. By \Cref{lem:Wittcomp}, $W(R^\flat[J/s^\flat]^\wedge)$ is the $(p,[s^\flat])$-adic completion of $W(R^\flat)[[J]/[s^\flat]]$. We see from the tilting equivalence between $R$ and $R^\flat$ that the composition $W(R^\flat) \to R \to R\langle \tfrac{T_1,\ldots,T_r}{s} \rangle$ induces $W(R^\flat)[[J]/[s^\flat]] \to R\langle \tfrac{T_1,\ldots,T_r}{s} \rangle^+$. Since $(R\langle \tfrac{T_1,\ldots,T_r}{s} \rangle, R\langle \tfrac{T_1,\ldots,T_r}{s} \rangle^+)$ is uniform, it induces a continuous homomorphism $W(R^\flat[J/s^\flat]^\wedge)\to R\langle \tfrac{T_1,\ldots,T_r}{s} \rangle^+$. 
    
    Let $S$ be the untilt of $R^\flat[J/s^\flat]^\wedge$ over $R$. The previous map provides $S \to R\langle \tfrac{T_1,\ldots,T_r}{s} \rangle^+$. Since $s^\flat$ is a topologically nilpotent non-zero-divisor of $S^\flat$, $s$ is a topologically nilpotent non-zero-divisor of $S$. Let $S^+\subset S[1/s]$ be the integral closure of $S$. By \cite[Lemma 3.21]{BMS18}, $(S[1/s],S^+)$ is a perfectoid Huber pair and we have a continuous homomorphism $(S[1/s],S^+) \to (R\langle \tfrac{T_1,\ldots,T_r}{s} \rangle, R\langle \tfrac{T_1,\ldots,T_r}{s} \rangle^+)$. It is enough to show that this is an isomorphism. Since the composition $\Spa(S[1/s],S^+)\to \Spa(S)\to \Spa(R)$ factors through $U(\tfrac{T_1,\ldots,T_r}{s})$, we get a continuous homomorphism $(R\langle \tfrac{T_1,\ldots,T_r}{s} \rangle, R\langle \tfrac{T_1,\ldots,T_r}{s} \rangle^+) \to (S[1/s],S^+)$. It is the inverse to the previous map because both continuous maps
    \[
        S[1/s] \to R\langle \tfrac{T_1,\ldots,T_r}{s} \rangle, \quad
        R\langle \tfrac{T_1,\ldots,T_r}{s} \rangle \to S[1/s]
    \]
    are $R[1/s]$-linear and $R[1/s]$ is dense both in $S[1/s]$ and in $R\langle \tfrac{T_1,\ldots,T_r}{s} \rangle$. 
\end{proof}

\begin{cor}
    Let $\mfr{X}$ be a perfectoid formal scheme. Then, $\mfr{X}^\an$ is a perfectoid space. 
\end{cor}
\begin{proof}
    By \Cref{prop:Tatedomisperfd}, $\Spa(R)^\an$ is a perfectoid space, so $\mfr{X}^\an$ is perfectoid. 
\end{proof}


\begin{prop} \label{prop:thickan}
    Let $R$ be a thick complete adic perfectoid ring. The homomorphism $R=\Gamma(\Spf(R),\cl{O}) \to \Gamma(\Spa(R)^\an, \cl{O}^+)$ is a homeomorphism onto the image and the image contains $\Gamma(\Spa(R)^\an, \cl{O}^{\circ \circ})$. 
\end{prop}
\begin{proof}
    First, we show that $R\to \Gamma(\Spa(R)^\an, \cl{O}^+)$ is injective. Suppose that $a\in R$ maps to $0$. We take a continuous arc-cover $R\to S=\prod_{\lambda\in \Lambda} V_\lambda$ as in \Cref{lem:prodarcpts}. It is enough to show that the image of $a$ in $V_\lambda$ is zero for every $\lambda \in \Lambda$. Let $K_\lambda=\Frac(V_\lambda)$ and consider the map $\Spa(K_\lambda, V_\lambda) \to \Spa(R)$. If it factors through $\Spa(R)^\an$, the claim follows from our assumption on $a$. Otherwise, $\Spa(K_\lambda, V_\lambda)\to \Spa(R)$ factors through $\Spa(R_\red)$. It is enough to show that the restriction of $a$ to $\Spec(R_\red)$ is zero. Now, since $\Spf(R)$ is thick, the specialization map of $\Spf(R)$ is surjective. Thus, for every point $x\in \Spec(R_\red)$, there is a point $\Spa(K,V)\to \Spa(R)^\an$ specializing to $x$. Since the pullback of $a$ along $\Spf(V) \to \Spf(R)$ is zero, we see that $a$ is zero at $x$. Since $R_\red$ is reduced, the restriction of $a$ to $\Spec(R_\red)$ is zero and we get the claim. 

    Next, we show that the image of $R\to \Gamma(\Spa(R)^\an, \cl{O}^+)$ contains $\Gamma(\Spa(R)^\an, \cl{O}^{\circ \circ})$. For each $a\in \Gamma(\Spa(R)^\an, \cl{O}^{\circ \circ})$, we define $a_\lambda \in V_\lambda^{\circ \circ}$ for every $\lambda\in \Lambda$ so that $a_\lambda$ is the pullback of $a$ along $\Spa(K_\lambda, V_\lambda) \to \Spa(R)^\an$ if $\Spa(K_\lambda, V_\lambda)$ is an analytic point of $\Spa(R)$, and $a_\lambda=0$ otherwise. The collection $(a_\lambda)_{\lambda\in \Lambda}$ defines an element of $S$. We show that it descends to $R$. 

    Let $i_1, i_2\colon S \to S\widehat{\otimes}_R S$ be homomorphisms such that $i_1(s) = s\otimes 1$ and $i_2(s) = 1\otimes s$ for $s\in S$. Let $b=i_1((a_\lambda)) - i_2((a_\lambda)) \in S\widehat{\otimes}_R S$. The restriction of $(a_\lambda)$ to $\Spa(S)^\an$ is equal to the pullback of $a$ along $\Spa(S)^\an \to \Spa(R)^\an$, so the restriction of $b$ to $\Spa(S\widehat{\otimes}_R S)^\an$ is zero. Moreover, the restriction of $(a_\lambda)$ to $\Spec(S_\red)$ is zero since $a \in \Gamma(\Spa(R)^\an, \cl{O}^{\circ \circ})$ and $a_\lambda=0$ when $\Spa(K_\lambda,V_\lambda)$ is a discrete point of $\Spa(R)$. Thus, the restriction of $b$ to $\Spec((S\widehat{\otimes}_R S)_\red)$ is zero. By the same argument as in the proof of injectivity, we see that $b=0$. Thus, the element $(a_\lambda)$ descends to $R$. It follows from \cite[Theorem 8.7]{Sch17} that it maps to $a$ under $R\to \Gamma(\Spa(R)^\an, \cl{O}^+)$ since $\Spa(S)^\an \to \Spa(R)^\an$ is a $v$-cover of perfectoid spaces. 

    Finally, we prove that $R\to \Gamma(\Spa(R)^\an, \cl{O}^+)$ is a homeomorphism onto the image. Since $\Spa(R)^\an$ is a perfectoid space, and in particular uniform, $\Gamma(\Spa(R)^\an, \cl{O}^+) \to \prod_{\lambda \in \Lambda^\an} V_\lambda$ is a homeomorphism onto the image with $\Lambda^\an \subset \Lambda$ consisting of indices of analytic points $\Spa(K_\lambda,V_\lambda) \to \Spa(R)^\an$ (see e.g. \cite[Theorem 5.2.1]{SW20}). It suffices to show that the composition $R\to \prod_{\lambda \in \Lambda } V_\lambda \to \prod_{\lambda \in \Lambda^\an } V_\lambda$ is a homeomorphism onto the image. This follows from \Cref{prop:perfduniform} since discrete valuations on $R$ do not contribute to the supremum over $\cl{M}_I(R)$ for an ideal of definition $I\subset R$. 
\end{proof}

Recall from \Cref{ssc:kimber} that a complete adic ring $R$ is distinguished if $R$ is $I$-torsion free for some ideal of definition $I\subset R$. 

\begin{cor} \label{cor:Itorsfree}
    Let $R$ be a complete adic perfectoid ring with an ideal of definition $I\subset R$. Then, $R$ is thick if and only if $R$ is distinguished. 
\end{cor}
\begin{proof}
    Since $\Gamma(\Spa(R)^\an, \cl{O}^+)$ is $I$-torsion free, the claim follows from \Cref{cor:nzdgenerator} and \Cref{prop:thickan}. 
\end{proof}

\begin{rmk}
    Since $R$ is thick if and only if $R^\flat$ is thick, it follows that $R$ is distinguished if and only if $R^\flat$ is distinguished. 
\end{rmk}

\section{Dilatation theory of $v$-sheaves}
\label{sec:vdil}


\subsection{Very good covers} \label{ssec:verygoodcovers}

In this section, we introduce very good covers of complete adic rings as preliminaries to good covers. 

Let $R$ be a complete adic ring with an ideal of definition $p \in I\subset R$. Let $R_\bullet = (R=R_0 \hookrightarrow R_1 \hookrightarrow \cdots)$ be a sequence of rings with injective transition maps and let $\Gamma_\bullet = (\{1\} = \Gamma_0 \twoheadleftarrow \Gamma_1 \twoheadleftarrow \cdots)$ be a sequence of finite groups with surjective transition maps. We set $\Gamma_\infty = \lim_{n\geq 0} \Gamma_n$ and $\Gamma^n_m = \Ker(\Gamma_m \twoheadrightarrow \Gamma_n)$ for every $0\leq n \leq m \leq \infty$. 

\begin{defi} \label{defi:actofsystem}
    An action of $\Gamma_\bullet$ on $R_\bullet$ is a family of an action of $\Gamma_n$ on $R_n$ for each $n\geq 0$ such that the map $R_n \to R_{n + 1}$ is equivariant under $\Gamma_{n + 1}$ for every $n \geq 0$. Here, $\Gamma_{n + 1}$ acts on $R_n$ via the quotient $\Gamma_{n + 1} \twoheadrightarrow \Gamma_n$.
\end{defi}

In particular, the action of $\Gamma^n_m$ on $R_m$ is $R_n$-linear for every $0\leq n \leq m < \infty$. 

\begin{defi} \label{defi:verygoodcov}
    A pair $(R_\bullet, \Gamma_\bullet)$ with an action of $\Gamma_\bullet$ on $R_\bullet$ is a very good cover of a complete adic ring $R$ over $\bb{Z}_p$ if 
    \begin{enumerate}
        \item the action of $\Gamma_{n+1}^n$ on $R_{n+1}$ over $R_n$ satisfies the condition in \Cref{lem:invanalogue}, 
        \item $R_{n+1}/R_{n}$ is a finite projective $R_n$-module for every $n\geq 0$, and
        \item the $I$-completed colimit $R_\infty = \colim_{n\geq 0}^\wedge R_n$ is a complete adic perfectoid ring.
    \end{enumerate}
\end{defi}

We list some basic properties of very good covers. 

\begin{lem} \label{lem:appdirect}
    Let $(R_\bullet, \Gamma_\bullet)$ be a very good cover of a complete adic ring $R$ over $\bb{Z}_p$.
    \begin{enumerate}
        \item $R_m$ is finite projective over $R_n$ and the action of $\Gamma_{m}^n$ on $R_{m}$ over $R_n$ satisfies the condition in \Cref{lem:invanalogue} for every $0 \leq n \leq m < \infty$. 
        \item $R_n$ is $I$-adically complete and is a direct summand of $R_\infty$ as an $R_n$-module for every $n\geq 0$. 
    \end{enumerate}
\end{lem}
\begin{proof}
    Let $S_0=R_0=R$ and $S_n=R_n/R_{n-1}$ for $n\geq 1$. Since $S_n$ is finite projective over $R_n$, we can take a decomposition $R_{n+1}=R_n\oplus S_{n+1}$ and we have $R_n = S_0\oplus \cdots \oplus S_n$. In particular, $R_m = R_n \oplus S_{n+1} \oplus \cdots \oplus S_m$ and $R_m$ is finite projective over $R_n$. Then, (1) follows from \Cref{lem:invtrans}. Moreover, $R_n$ is finite projective over $R$ and $R_n$ is $I$-adically complete. Then, we have $R_\infty = \colim_{m \geq n}^\wedge (R_n\oplus S_{n+1} \oplus \cdots \oplus S_m) = R_n\oplus (\widehat{\oplus}_{m > n} S_m)$, so $R_n$ is a direct summand of $R_\infty$ as an $R_n$-module. 
\end{proof}

\begin{lem} \label{lem:transquot}
    Let $(R_\bullet, \Gamma_\bullet)$ be a very good cover of a complete adic ring $R$ over $\bb{Z}_p$. The map $\Spd(R_m)\to \Spd(R_n)$ is proper and a geometric quotient by $\Gamma_m^n$ for every $0\leq n \leq m \leq \infty$. 
\end{lem}
\begin{proof}
    As in the proof of \Cref{prop:affquot}, $R_m$ is integral over $R_n$ and we have the claim for $m<\infty$ by \Cref{lem:mapSpd} and \Cref{prop:affquot}. Since properness and geometric quotients are preserved under limits (apply \cite[Proposition 18.3]{Sch17} for the former, and see \Cref{prop:limquot} for the latter), we have the claim for $m=\infty$. 
\end{proof}

The main source of complete adic rings admitting very good covers is as follows. 

\begin{prop} \label{prop:exadmit}
    Let $E$ be a complete discrete valuation field over $\bb{Q}_p$ with perfect residue field and let $O_E$ be the ring of integers of $E$. For every $r\geq 0$, $O_E \langle T_1,\ldots, T_r \rangle$ admits a very good cover. 
\end{prop}

\begin{proof}
    Let $\pi \in O_E$ be a uniformizer of $E$. For $n\geq 0$, we set $E_n = E[\pi^{1/p^n}, \zeta_{p^n}]$ with $\zeta_{p^n}$ a $p^n$-th root of unity. We take a family $\{\zeta_{p^n}\}$ so that $\zeta_{p^{n+1}}^p=\zeta_{p^n}$. Then, $E_n$ is a finite Galois extension of $E$. Let $O_{E_n}$ be the ring of integers of $E_n$ and let $R_n = O_{E_n} \langle T_1^{1/p^n}, \ldots, T_r^{1/p^n} \rangle$, $\Gamma_n = \Gal(E_n/E) \ltimes (\bb{Z}/p^n)^{\oplus r}$. Here, the semidirect product is taken with respect to the cyclotomic character $\Gal(E_n / E) \to (\bb{Z}/p^n)^\times$. There is a natural inclusion $R_n \hookrightarrow R_{n+1}$ and a natural surjection $\Gamma_{n+1}\twoheadrightarrow \Gamma_n$ for every $n\geq 0$. Moreover, there is an action of $\Gamma_n$ on $R_n$ such that the action of $(\sigma, a_1,\ldots, a_r) \in \Gal(E_n/E)\ltimes (\bb{Z}/p^n)^{\oplus r}$ sends $x\in O_{E_n}$ to $\sigma(x)$, and sends $T_i^{1/p^n}$ to $\zeta_{p^n}^{a_i} T_i^{1/p^n}$ for $1\leq i \leq r$. This defines an action of $\Gamma_\bullet$ on $R_\bullet$. It is easy to see that $R_n = R_{n+1}^{\Gamma_{n+1}^n}$ and $R_{n+1}/R_n$ is finite free over $R_n$ (since $O_{E_{n+1}}/O_{E_n}$ is finite free over $O_{E_n}$). It is enough to show that $R_\infty$ is a perfectoid ring. Consider the $p$-completed colimit $O_{E_\infty} = \colim^\wedge_{n\geq 0} O_{E_n}$ and let $E_\infty  = O_{E_\infty}[\tfrac{1}{p}]$. Then, $E_\infty$ is a perfectoid field and $O_{E_{\infty}}$ is a perfectoid ring. Since $R_\infty$ is $\pi$-torsion free, it is enough by \cite[Lemma 3.10]{BMS18} to check that the Frobenius $R_{\infty}/\pi^{1/p} R_\infty\to R_\infty/\pi R_\infty$ is an isomorphism. Since $O_{E_\infty}$ is a perfectoid ring, it follows from the description $R_\infty/\pi R_\infty = \colim_{n\geq 0} (O_{E_n}/\pi O_{E_n})[T_1^{1/p^i},\ldots, T_r^{1/p^i}]$. 
\end{proof}

The class of complete adic rings admitting very good covers is closed under \'{e}tale localization and completion. 

\begin{lem} \label{lem:etperfcov}
    Let $R$ be a complete adic ring with an ideal of definition $p \in I\subset R$. Let $S$ be an $I$-adically complete $I$-completely \'{e}tale $R$-algebra. If $R$ admits a very good cover, $S$ admits a very good cover. 
\end{lem}
\begin{proof}
    Let $(R_\bullet, \Gamma_\bullet)$ be a very good cover of $R$. Let $S_\bullet = R_\bullet {\otimes}_{R} S$. The condition in \Cref{lem:invanalogue} is stable under base change, so the action of $\Gamma^n_{n+1}$ on $S_{n+1}$ over $S_n$ satisfies that condition. By construction, $S_{n+1}/S_n$ is finite projective over $S_n$. Moreover, $S_\infty = R_\infty \widehat{\otimes}_{R} S$ is a complete adic perfectoid ring by \Cref{lem:etperfd}. 
\end{proof}

\begin{lem} \label{lem:compperfcov}
    Let $R$ be a complete adic ring in which $p$ is topologically nilpotent, and let $J\subset R$ be a finitely generated open ideal. If $R$ admits a very good cover, the $J$-adic completion $R^{\wedge}$ admits a very good cover. Moreover, $R^{\wedge}$ admits a very good cover as a $p$-adic ring if $R$ is $p$-torsion free, $p$-adic and Noetherian and $R/pR$ is $F$-finite. 
\end{lem}
\begin{proof}
    Let $(R_\bullet, \Gamma_\bullet)$ be a very good cover of $R$. Let $R^\wedge_\bullet = R_\bullet {\otimes}_{R} R^\wedge$. The condition in \Cref{lem:invanalogue} is stable under base change, so the action of $\Gamma^n_{n+1}$ on $R^\wedge_{n+1}$ over $R^\wedge_n$ satisfies that condition. By construction, $R^\wedge_{n+1}/R^\wedge_n$ is finite projective over $R^\wedge_n$. Moreover, $R^\wedge_\infty = R_\infty \widehat{\otimes}_{R} R^\wedge$ is isomorphic to the $J$-adic completion of $R_\infty$, so $R^\wedge_\infty$ is a complete adic perfectoid ring by \Cref{cor:tiltofcomp}. Thus, $(R_\bullet^\wedge, \Gamma_\bullet)$ is a very good cover of the $J$-adic ring $R^\wedge$. 

    We show the second claim. It is enough to show that the $p$-adic completion $\colim^\wedge_{n\geq 0} R^\wedge_n$ is perfectoid. Since $R_\infty$ is perfectoid, we can take an element $\pi \in R_\infty$ such that $\pi^p=pu$ for some $u\in R_\infty^\times$. We may approximate $\pi$ by an element of $R_n$ for sufficiently large $n$ and replace $R$ with $R_n$ to assume that $\pi \in R$. 
    
    Since $R$ is Noetherian, $R\to R^\wedge$ is flat, so $R^\wedge_\bullet$ is $p$-torsion free. By \cite[Lemma 3.10 (ii)]{BMS18}, it is enough to show that the Frobenius $\colim R^\wedge_n/\pi R^\wedge_n \to \colim R^\wedge_n/\pi^p R^\wedge_n$ is an isomorphism. Note that $\colim R_n/\pi R_n \to \colim R_n/\pi^p R_n$ is an isomorphism since $R_\infty$ is perfectoid, .
    
    Now, we regard $R_n / \pi^p R_n$ and $R^\wedge_n/\pi^p R^\wedge_n$ as $R$-modules along the Frobenius map $R \to R / \pi R \to R / \pi^p R$. Then, since $R / p R$ is $F$-finite, $R_n / \pi^p R_n$ is finite over $R$. Thus, we have
    \[
        (R_n / \pi^p R_n) \otimes_R R^\wedge \cong \lim_{i \geq 0} (R_n / \pi^p R_n) \otimes_R R / J^i \cong \lim_{i \geq 0} R_n / (\pi^p, \Frob(J^i)) \cong R^\wedge_n/\pi^p R^\wedge_n. 
    \]
    We will deduce the claim from this base change and the flatness of $R \to R^\wedge$. 
    
    First, since $R_n$ is a direct summand of $R_\infty$, $R_n/\pi R_n\to \colim_{m\geq n} R_m/\pi R_m$ is injective, so the Frobenius $R_n/\pi R_n \to R_n/\pi^p R_n$ is injective. Since $R\to R^\wedge$ is flat, $R_n^\wedge/\pi R^\wedge_n \to R^\wedge_n/\pi^p R^\wedge_n$ is injective, so $\colim R^\wedge_n/\pi R^\wedge_n \to \colim R^\wedge_n/\pi^p R^\wedge_n$ is injective. On the other hand, since $R_n/\pi^p R_n$ is finite over $R$, for every $n\geq 0$, there is $m\geq n$ such that the image of the Frobenius $R_m/\pi R_m \to R_m/\pi^p R_m$ contains $R_n/\pi^p R_n$. By passing to the $J$-adic completion, we see that the image of the Frobenius $R^\wedge_m/\pi R^\wedge_m \to R^\wedge_m/\pi^p R^\wedge_m$ contains $R^\wedge_n/\pi^p R^\wedge_n$. In particular, $\colim R^\wedge_n/\pi R^\wedge_n \to \colim R^\wedge_n/\pi^p R^\wedge_n$ is surjective. 
\end{proof}

Thanks to \Cref{lem:etperfcov}, the following notion is purely local in the Zariski topology. 

\begin{defi}
    Let $\mfr{X}$ be a formal scheme over $\bb{Z}_p$. We say that $\mfr{X}$ admits a very good cover if $\mfr{X}$ is Zariski locally isomorphic to the formal spectrum of a complete adic ring admitting a very good cover. 
\end{defi}

The main source of formal schemes admitting very good covers is as follows. 

\begin{prop} \label{prop:smoothadmit}
    Let $E$ be a complete discrete valuation field over $\bb{Q}_p$ with perfect residue field and let $O_E$ be the ring of integers of $E$. Every $p$-adic smooth formal $O_E$-scheme admits a very good cover. 
\end{prop}
\begin{proof}
    After Zariski localization, every $p$-completely smooth $O_E$-algebra is $p$-completely \'{e}tale over $O_E\langle T_1,\ldots, T_r \rangle$ for some $r\geq 0$ (by the topological invariance of \'{e}tale sites). Thus, the claim follows from \Cref{prop:exadmit} and \Cref{lem:etperfcov}. 
\end{proof}

\begin{rmk} \label{rmk:exverygood}
    By \Cref{lem:compperfcov} and \Cref{prop:smoothadmit}, formal schemes admitting very good covers arise as the completion of smooth schemes. Here are arithmetically important examples: the $p$-adic completion of integral models of Shimura varieties at hyperspecial levels, and Rapoport-Zink spaces at hyperspecial levels with the $p$-adic uniformization (see \cite[Corollary 4.9]{Mie20}). 
\end{rmk}

\subsection{Good covers} \label{ssec:goodcov}

In this section, we introduce good covers of reduced excellent complete adic rings, which are refinements of very good covers modeled on analytic loci. They satisfy a minimal set of properties of very good covers needed in the proof of our representability criterion. 

Let $R$ be a reduced excellent complete adic ring with an ideal of definition $p \in I\subset R$. Let $R_\bullet = (R=R_0 \to R_1 \to \cdots)$ be a sequence of reduced finite $R$-algebras. Let $\Gamma_\bullet = (\{1\} = \Gamma_0 \twoheadleftarrow \Gamma_1 \twoheadleftarrow \cdots)$ be a sequence of finite groups with surjective transition maps that acts on $R_\bullet$ in the sense of \Cref{defi:actofsystem}. To simplify the notation, for every $n\geq 0$ and every Tate rational domain $U_n \subset \Spa(R_n)^\an$, the pullback of $U_n$ in $\Spa(R_m)^\an$ is denoted by $U_m$ for every $m \geq n$. Later, this notation is also used for $m=\infty$. 

As a refinement of the condition (1) in \Cref{defi:verygoodcov}, we first put the following condition: 
\begin{quote}
    (quot) for every $n\geq 0$, $\Spd(R_{n+1}) \to \Spd(R_n)$ is a geometric quotient by $\Gamma_{n+1}^n$, and for every Tate rational domain $U_n \subset \Spa(R_n)^\an$, the pullback $\cl{O}(U_n) \to \cl{O}(U_m)$ is injective and satisfies the condition in \Cref{lem:invanalogue} with respect to the $\Gamma_m^n$-action. 
\end{quote}

The latter condition of (quot) is analytically local on $\Spa(R_n)^\an$. 

\begin{lem} \label{lem:quotislocal}
    If there is a covering $\{V_{n,i}\}_{i\in I}$ of $\Spa(R_n)^\an$ by Tate rational domains such that $\cl{O}(V_{n,i}) \to \cl{O}(V_{m,i})$ is injective and satisfies the condition in \Cref{lem:invanalogue}, then every Tate rational domain $U_n\subset \Spa(R_n)^\an$ satisfies the same conditions. 
\end{lem}
\begin{proof}
    Since $R_m$ is finite over $R_n$, we have $\cl{O}(U_m)=\cl{O}(U_n)\otimes_{R_n} R_m$. First, suppose that $U_n \subset V_{n,i}$ for some $i\in I$. Then, we have $\cl{O}(U_m)=\cl{O}(U_n)\otimes_{\cl{O}(V_{n,i})}\cl{O}(V_{m,i})$. Since $V_{n,i}$ is strongly Noetherian, $\cl{O}(V_{n,i}) \to \cl{O}(U_n)$ is flat (see \cite[Lemma 1.7.6]{Hub96}). Thus, $\cl{O}(U_n) \to \cl{O}(U_m)$ is injective. It also satisfies the condition in \Cref{lem:invanalogue} since that condition is stable under base change. Thus, the case where $U_n \subset V_{n,i}$ is done. 
    
    Now, by refining the covering $\{V_{n,i}\}_{i\in I}$, we may suppose that $U_n$ is covered by a collection $\{V_{n,j}\}_{j \in J}$ for some $J\subset I$. Since $\cl{O}(U_n) \to \prod_{j\in J} \cl{O}(V_{n,j})$ is injective and the same holds for $U_m$, $\cl{O}(U_n) \to \cl{O}(U_m)$ is injective. Moreover, we have $\cl{O}(U_n)=\cl{O}(U_m) \cap \prod_{j \in J} \cl{O}(V_{n,j})$ in $\prod_{j \in J} \cl{O}(V_{m,j})$, so the condition in \Cref{lem:invanalogue} is satisfied by \Cref{lem:invsubalg}. 
\end{proof}

As a refinement of the condition (3) in \Cref{defi:verygoodcov}, we next put the following condition: 
\begin{quote}
    (perfd) $\lim_{n \geq 0} \Spd(R_n)$ is represented by a complete $I$-adic perfectoid $R$-algebra $R_\infty$. 
\end{quote}

By \Cref{lem:perfdrep}, we get a homomorphism $R_n\to R_\infty$ for every $n\geq 0$. Since $R_n$ is reduced and excellent, and $R_\infty$ is a complete adic perfectoid ring, we have a sequence of stably uniform analytic adic spaces (see \Cref{lem:excellentst} and \Cref{prop:Tatedomisperfd})
\[\Spa(R_\infty)^\an \to \cdots \to \Spa(R_n)^\an \to \cdots \to \Spa(R_1)^\an \to \Spa(R_0)^\an = \Spa(R)^\an.\]
By construction, the $v$-sheafification of this sequence is a limit sequence. We will impose that this is a limit sequence of analytic adic spaces. First, we put a density condition:
\begin{quote}
    (density) for every $n\geq 0$ and every Tate rational domain $U_n \subset \Spa(R_n)^\an$, the image of $\colim_{n \leq m < \infty} \cl{O}(U_m) \to \cl{O}(U_\infty)$ is dense. 
\end{quote}

Another condition on the sequence is the existence of functorial and uniform sections. It is one of the essential conditions and is a refinement of the condition (2) in \Cref{defi:verygoodcov}. 
\begin{quote}
    (unifsec) for every $n\geq 0$ and every Tate rational domain $U_n \subset \Spa(R_n)^\an$, we have a continuous $\cl{O}(U_n)$-linear section $s_{n,U_n} \colon \cl{O}(U_\infty) \to \cl{O}(U_n)$ of the pullback. The family $\{s_{n,U_n}\}$ is functorial in $U_n$ for fixed $n$, and uniform in $n$ for fixed $U_n$ in a sense that there is an open ideal $\cl{J} \subset \cl{O}^+(U_\infty)$ that maps to $\cl{O}^+(U_m)$ under $s_{m,U_m}$ for every $m\geq n$. 
\end{quote}

From now on, a continuous section $s_{n,U_n}$ of the pullback is always assumed to be $\cl{O}(U_n)$-linear. We will show that the above two conditions are analytically local.  

\begin{lem} \label{lem:densitys}
    Let $U_n \subset \Spa(R_n)^\an$ be a Tate rational domain with a uniform family of continuous sections $\{s_{m,U_m}\}$ for $m\geq n$. Then, the image of $\colim_{n \leq m < \infty} \cl{O}(U_m) \to \cl{O}(U_\infty)$ is dense if and only if $\{s_{m,U_m}(a)\}_{m\geq n}$ converges to $a$ in $\cl{O}(U_\infty)$ for every $a\in \cl{O}(U_{\infty,})$.  
\end{lem}
\begin{proof}
    The if direction is obvious. We show the only if direction. Let $a\in \cl{O}(U_\infty)$ be an arbitrary element. By assumption, there is a sequence $\{b_{m}\}_{m\geq n}$ with $b_{m}\in \cl{O}(U_m)$ that converges to $a$ in $\cl{O}(U_\infty)$. Let $\cl{J} \subset \cl{O}^+(U_\infty)$ be an open ideal that maps to $\cl{O}^+(U_m)$ under $s_{m,U_m}$ for every $m \geq n$. For every $N\geq 0$, there is $k\geq 0$ such that for every $m\geq k$, we have $a - b_{m} \in I^N\cdot \cl{J}$. Then, we have $s_{m,U_m}(a) - b_m \in I^N\cdot \cl{O}^+(U_m)$ and $a-s_{m,U_m}(a) \in I^N\cdot \cl{O}^+(U_\infty)$ for $m\geq k$. Since $\cl{O}^+(U_\infty)$ is $I$-adic, $\{s_{m,U_m}(a)\}$ converges to $a$. 
\end{proof}

\begin{lem} \label{lem:densunifislocal}
    The conditions (density) and (unifsec) hold if and only if there is a finite covering $\{V_{0,i}\}_{i\in I}$ of $\Spa(R)^\an$ by Tate rational domains with continuous sections $\{s_{n,V_{n,i}}\}_{n\geq 0, i \in I}$ such that 
    \begin{enumerate}
        \item the family $\{s_{n,V_{n,i}}\}_{n\geq 0}$ is uniform in $n$ for each $i\in I$, 
        \item $s_{n,V_{n,i}}$ and $s_{n,V_{n,j}}$ induce the same continuous section $\cl{O}(V_{\infty,i} \cap V_{\infty, j}) \to \cl{O}(V_{n,i} \cap V_{n,j})$ for every $i, j \in I$, and
        \item $\{s_{n,V_{n,i}}(a)\}_{n\geq 0}$ converges to $a$ in $\cl{O}(V_{\infty,i})$ for every $a\in \cl{O}(V_{\infty, i})$ and $i\in I$. 
    \end{enumerate}
\end{lem}

Before going to the proof, we explain the condition (2) more precisely. For each $i\in I$, we define $s_{n,U_n}$ for every Tate rational domain $U_n \subset V_{n,i}$. Let $\varpi \in \cl{O}^+(V_{n,i})$ be a topologically nilpotent unit. By the universality and uniformity of rational localization, we have $\cl{O}(U_\infty) \cong (\cl{O}^+(V_{\infty,i})\widehat{\otimes}_{\cl{O}^+(V_{n,i})} \cl{O}^+(U_n)/\overline{(\varpi^\infty \mathchar`- \torsion)})[\tfrac{1}{\varpi}]$. Take an open ideal $\cl{J}_i \subset \cl{O}^+(V_{\infty,i})$ as in (unifsec). Then, $s_{n,V_{n,i}} \colon \cl{J}_i \to \cl{O}^+(V_{n, i})$ induces 
\[
    \cl{J}_i \widehat{\otimes}_{\cl{O}^+(V_{n,i})} \cl{O}^+(U_n) \to \cl{O}^+(U_n)
\]
via base change, which induces a continuous section $s_{n,U_n}$ after inverting $\varpi$. The condition (2) says that this construction gives the same $s_{n,V_{n,i}\cap V_{n,j}}$ for different $i,j\in I$. Under this condition, $s_{n,U_n}$ is independent of the choice of $i\in I$ such that $U_n \subset V_{n,i}$. 

\begin{proof}
    The only if direction follows from \Cref{lem:densitys}. We will prove the converse direction. 

    By the above remark, the condition (2) enables us to define a functorial family of continuous sections $s_{n,U_n}$ for $U_n \subset V_{n,i}$ for some $i\in I$. For every $m\geq n$, the image of $\cl{J}\widehat{\otimes}_{\cl{O}^+(V_{m,i})} \cl{O}^+(U_m)$ contains the image of $\cl{J}\widehat{\otimes}_{\cl{O}^+(V_{n,i})} \cl{O}^+(U_n)$ in $\cl{O}^+(U_\infty)$, so the family $\{s_{n,U_n}\}_{n\geq 0}$ is uniform in $n$. Moreover, (density) holds for $U_n$ since we have $\cl{O}(U_\infty) \cong (\cl{O}^+(V_{\infty,i})\widehat{\otimes}_{\cl{O}^+(V_{n,i})} \cl{O}^+(U_n)/\overline{(\varpi^\infty \mathchar`- \torsion)})[\tfrac{1}{\varpi}]$. Thus, (density) and (unifsec) hold for $U_n \subset V_{n,i}$ for some $i\in I$. 

    By refining the covering $\{V_{n,i}\}_{i\in I}$, it is enough to treat the case where $U_n$ can be covered by $V_{n,j}$ for some $J\subset I$. Since the family $\{s_{n,V_{n,j}}\}_{j\in J}$ is functorial in $V_{n,j}$, we can glue it to a continuous section $s_{n,U_n}$. Moreover, the uniformity of $s_{n,U_n}$ in $n$ follows from that of $s_{n,V_{n,j}}$ for every $j\in J$. Since $\cl{O}(U_\infty)$ is equipped with the subspace topology of $\prod_{j\in J}\cl{O}(V_{\infty,j})$, the condition (3) for $U_n$ follows from the condition (3) for all $V_{n,j}$. Thus, (density) and (unifsec) hold for $U_n$. 
\end{proof}

\begin{defi}
    A pair $(R_\bullet, \Gamma_\bullet)$ is a good cover of a reduced excellent complete adic ring $R$ over $\bb{Z}_p$ if it satisfies (quot), (perfd), (density) and (unifsec). 
\end{defi}

\begin{lem} \label{lem:verygoodisgood}
    A very good cover $(R_\bullet, \Gamma_\bullet)$ of a reduced excellent complete adic ring $R$ over $\bb{Z}_p$ is a good cover of $R$. 
\end{lem}
\begin{proof}
    The condition (3) in \Cref{defi:verygoodcov} implies that the limit $\lim_{n \geq 0} \Spd(R_n)$ is represented by $R_\infty = \colim_{n \geq 0}^\wedge R_n$. Thus, $(R_\bullet, \Gamma_\bullet)$ satisfies (perfd). For every $n\geq 0$ and every Tate rational domain $U_n \subset \Spa(R_n)^\an$, $R_n\to \cl{O}(U_n)$ is flat by \cite[Lemma 1.7.6]{Hub96} and $\cl{O}(U_n) \to \cl{O}(U_m)$ is the base change of $R_n\to R_m$ for every $n\leq m <\infty$. Thus, (quot) holds by the condition (1) in \Cref{defi:verygoodcov}. 
    
    We will verify (density) and (unifsec). By \Cref{lem:appdirect} (2), there is a continuous section $s_n \colon R_{\infty} \to R_n$ of the natural inclusion. Let $U_n \subset \Spa(R_n)^\an$ be a Tate rational domain and let $\varpi \in \cl{O}^+(U_n)$ be a topologically nilpotent unit. Since we have $\cl{O}(U_\infty) \cong (R_\infty \widehat{\otimes}_{R_n} \cl{O}^+(U_n))[\tfrac{1}{\varpi}]$, (density) holds for $U_n$ and $s_n$ induces a continuous section $s_{n,U_n} \colon \cl{O}(U_\infty) \to \cl{O}(U_n)$. As the image of $R_\infty \widehat{\otimes}_{R_n} \cl{O}^+(U_n)$ in $\cl{O}^+(U_\infty)$ maps to $\cl{O}^+(U_n)$ under $s_{n,U_n}$, we see that $\{s_{n,U_n}\}_{n\geq 0}$ is uniform in $n$. Thus, $(R_\bullet, \Gamma_\bullet)$ satisfies (unifsec). 
\end{proof}

\begin{lem} \label{lem:goodcovetloc}
    Let $R$ be a reduced excellent complete adic ring with an ideal of definition $I\subset R$. Let $S$ be an $I$-adically complete $I$-completely \'{e}tale $R$-algebra. If $R$ admits a good cover, $S$ admits a good cover. 
\end{lem}
\begin{proof}
    Let $(R_\bullet, \Gamma_\bullet)$ be a good cover of $R$. Let $S_\bullet = R_\bullet \otimes_{R} S$. By Elkik's algebraization theorem (see \cite{Elk73}), $S_n$ is isomorphic to the $I$-adic completion of an \'{e}tale $R_n$-algebra. For every $0\leq n < \infty$, $S_n$ is a reduced finite $S$-algebra.

    Now, $S_\infty  = R_\bullet \widehat{\otimes}_{R} S$ is a complete adic perfectoid ring by \Cref{lem:etperfd} and represents $\Spd(R_\infty) \times_{\Spd(R)} \Spd(S)=\lim_{n\geq 0} \Spd(S_n)$. Since the conditions (quot), (density) and (unifsec) are local by \Cref{lem:quotislocal} and \Cref{lem:densunifislocal}, it is enough to check these conditions for sufficiently small Tate rational domains $U_n \subset \Spa(S_n)^\an$. In particular, we may assume that $U_n \to \Spa(R_n)^\an$ factors through a Tate rational domain $V_n \subset \Spa(R_n)^\an$ and is a composition of a finite \'{e}tale morphism to $V_n$ and a rational localization. Then, $U_\infty \cong V_{\infty} \times_{V_n} U_n$ and $s_{n,V_n}$ induces a continuous section $s_{n,U_n}$ as in \Cref{lem:densunifislocal}. It is easy to see that this construction is functorial and satisfies the density condition as in \Cref{lem:densunifislocal}. Moreover, the latter condition of (quot) follows from the fact that $\cl{O}(V_n) \to \cl{O}(U_n)$ is flat. 
\end{proof}


\begin{lem} \label{lem:analyticfinet}
    Let $R$ be a reduced excellent complete adic ring with an ideal of definition $I\subset R$. Let $S$ be a reduced finite $R$-algebra that is finite \'{e}tale over $R$ outside $V(I) \subset \Spec(R)$. If $R$ admits a good cover, $S$ admits a good cover. 
\end{lem}
\begin{proof}
    Let $(R_\bullet, \Gamma_\bullet)$ be a good cover of $R$. Let $S_n$ be the reduction of $R_n{\otimes}_{R} S$ for every $0\leq n < \infty$ and let $S^0_\infty$ be the perfectoidization of $R_\infty \otimes_R S$. By \cite[Theorem 10.9]{BS22}, $S^0_{\infty}$ is a perfectoid ring and $R_\infty\otimes_R S \to S_{\infty}^0$ is an isomorphism outside $V(I)$. Let $S_\infty=(S_\infty^0)^\wedge$ be the $I$-adic completion. By construction, $\Spd(S_\infty) \cong \Spd(R_\infty)\times_{\Spd(R)} \Spd(S)$ and $(S_\bullet,\Gamma_\bullet)$ satisfies (perfd). By the same argument as in \Cref{lem:goodcovetloc}, we see that $(S_\bullet, \Gamma_\bullet)$ satisfies (quot), (density) and (perfd). 
\end{proof}

Thanks to \Cref{lem:goodcovetloc}, the following notion is purely local in the Zariski topology. 

\begin{defi}
    Let $\mfr{X}$ be a reduced excellent formal scheme over $\bb{Z}_p$. We say that $\mfr{X}$ admits a good cover if $\mfr{X}$ is Zariski locally isomorphic to the formal spectrum of a reduced excellent complete adic ring admitting a good cover. 
\end{defi}

\subsection{Affine representability}
The aim of this section is to prove the following theorem. 

\begin{thm} \label{thm:affrep}
    Let $R$ be a reduced excellent complete adic ring admitting a good cover $(R_\bullet, \Gamma_\bullet)$ and let $I\subset R$ be an ideal of definition.    Let $S_\infty$ be a distinguished complete $I$-adic perfectoid $R_\infty$-algebra. Suppose that the following conditions hold. 
    \begin{enumerate}
        \item $\Gamma_\infty$ acts on $S_\infty$ compatibly with the action on $R_\infty$. 
        \item $\Spa(S_\infty)^\an \to \Spa(R_\infty)^\an$ is an open immersion. 
        \item $S_{\infty,\red}$ is perfectly of finite type over $R_{\infty,\red}$.
    \end{enumerate}
    Then, there is a good cover $(S_\bullet, \Gamma_\bullet)$ of a reduced excellent distinguished complete $I$-adic ring $S$ with the following conditions. 
    \begin{enumerate}
        \item There is a $\Gamma_\bullet$-equivariant homomorphism $R_\bullet \to S_\bullet$ such that for every $n\geq 0$, $S_n$ is topologically of finite type over $R_n$ and $\Spa(S_n)^\an\to \Spa(R_n)^\an$ is an open immersion. 
        \item There is a $\Gamma_\infty$-equivariant isomorphism $\Spd(S_\infty) \cong \lim_{n\geq 0} \Spd(S_n)$ such that the induced map $\Spd(S_\infty)\to \Spd(R_\infty)$ is represented by the given map $R_\infty \to S_\infty$. 
    \end{enumerate}
\end{thm}

\begin{proof}

For every $0\leq n \leq m \leq \infty$, the map $\Spa(R_m) \to \Spa(R_n)$ is denoted by $\pi_m^n$. First, we construct the open subspace of $\Spa(R_n)^\an$ that should be equal to $\Spa(S_n)^\an$. Let $U_n \subset \lvert \Spa(R_n)^\an \rvert$ be the image of $\lvert \Spa(S_\infty)^\an \rvert \subset \lvert \Spa(R_\infty)^\an \rvert$ under $\pi_\infty^n$.

\begin{lem} \label{lem:condUn}
    The subset $U_n\subset \lvert \Spa(R_n)^\an \rvert$ is open and $(\pi_m^n)^{-1}(U_n)=U_m$ for every $m \geq n$. 
\end{lem}
\begin{proof}
    By the condition (quot), $\Spd(R_n)$ is a geometric quotient of $\Spd(R_\infty)$ by $\Gamma_\infty^n$. Since $U_\infty \subset \lvert \Spa(R_\infty) \rvert$ is stable under $\Gamma_\infty$, we see that $(\pi_m^n)^{-1}(U_n)=U_m$ for every $m \geq n$. Since $\pi_\infty^n$ is surjective and $U_\infty$ is open, it follows from \cite[Proposition 12.9]{Sch17} that $U_n$ is open. 
\end{proof}

From now on, $U_n$ is identified with the open subspace of $\Spa(R_n)^\an$. Since $\Spa(S_\infty)^\an$ is quasicompact, $U_n$ is quasicompact. Thus, we can take a finite covering $\{V_{0,j}\}_{j\in J}$ of $U_0$ by Tate rational domains of $\Spa(R)^\an$. Its pullback $\{V_{n,j}\}_{j\in J}$ to $\Spa(R_n)^\an$ is a finite covering of $U_n$ by Tate rational domains of $\Spa(R_n)^\an$. Since the family of continuous sections $\{s_{n,V_{n,j}}\}$ is functorial, we get a continuous section $s_{n,U_n}\colon \cl{O}(U_\infty) \to \cl{O}(U_n)$. Moreover, since the family of continuous sections $\{s_{n,V_{n,j}}\}$ is uniform, there is an open ideal $\cl{J} \subset \cl{O}^+(U_\infty)$ such that $s_{n,U_n}(\cl{J}) \subset \cl{O}^+(U_n)$ for every $n\geq 0$. 

\begin{lem} \label{lem:densityU}
    The image of $\colim_{n<\infty} \cl{O}(U_n) \to \cl{O}(U_\infty)$ is dense. 
\end{lem}
\begin{proof}
    It is enough to see that $\{s_{n,U_n}(a)\}$ converges to $a$ for every $a \in \cl{O}(U_\infty)$. It follows from (density) and (unifsec) for each $V_{0, j}$ by \Cref{lem:densitys}. 
\end{proof}

Note that the existence of a continuous section implies that $\cl{O}(U_n)\to \cl{O}(U_\infty)$ is a closed embedding. Moreover, $S_\infty \to \cl{O}(U_\infty)$ is an open embedding by \Cref{prop:thickan}. 

Let $S_n=\cl{O}(U_n)\cap S_\infty$ where the intersection is taken inside $\cl{O}(U_\infty)$. Since $ \cl{O}^{\circ \circ}(U_\infty) \subset S_\infty \subset \cl{O}^+(U_\infty)$, we have $\cl{O}^{\circ \circ}(U_n) \subset S_n\subset \cl{O}^+(U_n)$. By \Cref{lem:densityU}, $\colim_{n<\infty} S_n \to S_\infty$ has a dense image. In particular, the completion of $\colim_{n<\infty} S_n$ with respect to the subspace topology from $S_\infty$ is isomorphic to $S_\infty$. This subspace topology can be identified as follows. 

\begin{lem} \label{lem:subspIadic}
    The subspace topology on $\colim_{n<\infty} S_n$ is the $I$-adic topology. Moreover,  the subspace topology on $S_n$ is the $I$-adic topology for every $n\geq 0$. 
\end{lem}
\begin{proof}
    Since $S_\infty$ is distinguished, $\cl{O}^+(U_\infty)$ is $I$-adic by \Cref{prop:thickan}. For every $n, m\geq 0$, we have $(I^m\cdot \cl{J}) \cap \cl{O}^+(U_n) \subset I^m\cdot \cl{O}^+(U_n) \subset (I^m\cdot \cl{O}^+(U_\infty))\cap \cl{O}^+(U_n)$. The first inclusion follows from the property of $s_{n,U_n}$. Thus, the subspace topology on $S_n$ is the $I$-adic topology as $S_n \supset \cl{O}^{\circ \circ}(U_n)$. The claim for $\colim_{n<\infty} S_n$ follows by the same argument. 
\end{proof}

As a consequence, we can verify the condition (perfd) on $S_\bullet$. 

\begin{cor} \label{cor:Slimit}
    $\Spd(S_\infty) \cong \lim_{n <\infty} \Spd(S_n)$. 
\end{cor}
\begin{proof}
    For every perfectoid Huber pair $(A,A^+)$ in characteristic $p$, a map from $\Spa(A,A^+)$ to $\lim_{n <\infty} \Spd(S_n)$ corresponds to an untilt $(A^\sharp, \iota)$ and a compatible family of continuous homomorphisms $S_n\to A^{\sharp +}$. Since $A^{\sharp +}$ is complete, this family corresponds to a map from the $I$-adic completion of $\colim_{n<\infty } S_n$ to $A^{\sharp +}$. Thus, we have $\Spd(S_\infty) \cong \lim_{n <\infty} \Spd(S_n)$. 
\end{proof}

Since $U_\infty$ is stable under $\Gamma_\infty$, $U_n$ is stable under $\Gamma_n$. Thus, $\Gamma_\infty$ acts on $\cl{O}(U_n)$ factoring through $\Gamma_n$. Since $S_\infty$ is stable under $\Gamma_\infty$, $\Gamma_\infty$ acts on $S_n$ factoring through $\Gamma_n$. It means that $\Gamma_\bullet$ acts on $S_\bullet$. By construction, we have a $\Gamma_\bullet$-equivariant homomorphism $R_\bullet \to S_\bullet$. We next verify the condition (quot) on $(S_\bullet, \Gamma_\bullet)$.

\begin{lem} \label{lem:Snquot}
    For every $0\leq n \leq m<\infty$, $S_n\to S_m$ is injective and satisfies the condition in \Cref{lem:invanalogue} with respect to the $\Gamma_m^n$-action. 
\end{lem}
\begin{proof}
    By construction, $S_n\to S_m$ is injective. It is enough to show the latter claim. By the condition (quot), $\prod_{j\in J} \cl{O}(V_{n,j}) \to \prod_{j\in J} \cl{O}(V_{m,j})$ satisfies the condition in \Cref{lem:invanalogue} with respect to the $\Gamma_m^n$-action. Since we have $\cl{O}^+(U_m) \cap \prod_{j\in J} \cl{O}(V_{n,j}) = \cl{O}^+(U_n)$, the claim follows from \Cref{lem:invsubalg}. 
\end{proof}

Thus, we verify the first condition of (quot) on $(S_\bullet, \Gamma_\bullet)$ by \Cref{prop:affquot}. In particular, we see from \Cref{prop:limquot} that $\Spd(S_n)$ is a geometric quotient of $\Spd(S_\infty)$ by $\Gamma_\infty^n$. 

\begin{cor} \label{cor:transsurj}
    For every $0\leq n \leq m \leq \infty$, $\Spd(S_m) \to \Spd(S_n)\times_{\Spd(R_n)} \Spd(R_m)$ is surjective. 
\end{cor}
\begin{proof}
    Since $\Spd(S_m) \to \Spd(S_n)\times_{\Spd(R_n)} \Spd(R_m)$ is qcqs by \cite[Lemma 2.26]{Gle24}, it is enough to show that it is surjective on points. By \Cref{lem:bcquot}, $\Spd(S_n)\times_{\Spd(R_n)} \Spd(R_m) \to \Spd(S_n)$ is a geometric quotient by $\Gamma_m^n$. Thus, for every geometric point $x\in \Spd(S_n)$, both fibers of $x$ in $\Spd(S_m)$ and in $\Spd(S_n)\times_{\Spd(R_n)} \Spd(R_m)$ consist of a single $\Gamma_m^n$-orbit. The claim follows since $\Spd(S_m) \to \Spd(S_n)\times_{\Spd(R_n)} \Spd(R_m)$ is $\Gamma_m^n$-equivariant. 
\end{proof}

\begin{cor} \label{cor:isomtoU}
    The map $\Spd(S_n)^\an  \to \Spd(R_n)^\an $ is an isomorphism onto $U_n$. 
\end{cor}
\begin{proof}
    By \Cref{lem:condUn}, $U_n$ is a geometric quotient of $U_\infty$ by $\Gamma_\infty^n$. Thus, $\Spd(S_n)^\an \to \Spd(R_n)^\an$ factors through $U_n$. Since $\Spd(S_n) \to \Spd(R_n)$ is qcqs by \cite[Lemma 2.26]{Gle24}, $\Spd(S_n)^\an \to U_n$ is qcqs. Since both spaces are geometric quotients of $U_\infty$ by $\Gamma_\infty^n$, it is an isomorphism by \Cref{prop:geomquotiscat}. 
\end{proof}

Next, we will verify that $S_n$ is topologically of finite type over $R_n$. For this, we apply \Cref{thm:topfin} to $S_n$. The first condition is verified in \Cref{lem:subspIadic} and the second condition follows from \Cref{cor:isomtoU} by \cite[Proposition 15.4]{Sch17}. It is enough to verify the last condition. 

\begin{lem}
    $(S_{n,\red})^\perf$ is perfectly of finite type over $(R_{n,\red})^\perf$. 
\end{lem}
\begin{proof}
    Since $S_\infty$ is isomorphic to the $I$-adic completion of $\colim_{n<\infty} S_n$, we have $S_{\infty,\red} \cong \colim_{n<\infty} S_{n,\red}$. Moreover, by the condition (perfd) and \cite[Proposition 18.3.1]{SW20}, we have $R_{\infty, \red} \cong \colim_{n<\infty} (R_{n,\red})^\perf$. Since $S_{\infty,\red}$ is perfectly of finite type over $R_{\infty, \red}$, for some $m\geq 0$, there is a homomorphism $P_m=R_{m,\red}[T_1,\ldots,T_r] \to S_{m,\red}$ such that if we set $P_n = R_{n,\red}[T_1,\ldots,T_r] \to S_{n,\red}$ as the base change of the map for $n \geq m$, $P_\infty^\perf \to S_{\infty,\red}$ is surjective. Then, the composition
    $\Spd(S_{\infty,\red}) \to \Spd(S_{n,\red})\times_{\Spd(R_{n,\red})} \Spd(R_{\infty,\red}) \to \Spd(P_n)\times_{\Spd(R_{n,\red})} \Spd(R_{\infty,\red}) \cong \Spd(P_\infty)$
    is a closed immersion for $n\geq m$. Then, $\Spd(S_{\infty,\red}) \to \Spd(S_{n,\red})\times_{\Spd(R_{n,\red})} \Spd(R_{\infty,\red})$ is injective and surjective by \Cref{cor:transsurj}, so it is an isomorphism. Since $\Spd(R_{\infty,\red}) \to \Spd(R_{n,\red})$ is a $v$-cover, $\Spd(S_{n,\red}^\perf) \to \Spd(P_n^\perf)$ is a closed immersion by \cite[Proposition 10.11]{Sch17}. By \cite[Lemma 3.31]{Gle24}, $P_n^\perf\to S_{n,\red}^\perf$ is surjective. Thus, $S_{n,\red}^\perf$ is perfectly of finite type over $R_{n,\red}^\perf$ for $n\geq m$.

    We treat the case with $n<m$. Since $R_m$ is finite over $R_n$, $S_{m,\red}^\perf$ is perfectly of finite type over $R_{n,\red}^\perf$. Since $\Gamma_m^n$ is finite, we can take a finitely generated $R_{n,\red}$-subalgebra $S^0 \subset S_{m,\red}^\perf$ stable under $\Gamma_m^n$ so that $S_0^\perf \cong S_{m,\red}^\perf$. Since $R_{n,\red}$ is Noetherian, $(S^0)^{\Gamma_m^n}$ is of finite type over $R_{n,\red}$, so $(S^0)^{\Gamma_m^n,\perf}$ is perfectly of finite type over $R_{n,\red}^\perf$. Since it represents a geometric quotient of $\Spd(S_{m,\red}^\perf)$ by $\Gamma_m^n$ by \Cref{prop:finquot}, $S_{n,\red}^\perf$ is isomorphic to $(S^0)^{\Gamma_m^n,\perf}$ by \cite[Lemma 2.26]{Gle24}, \Cref{prop:geomquotiscat} and \cite[Proposition 18.3.1]{SW20}. Thus, we get the claim. 
\end{proof}

Now, it follows from \Cref{thm:topfin} that $S_n$ is reduced and excellent, topologically of finite type over $R_n$ and $\Spa(S_n)^\an \to U_n$ is an isomorphism. Since $S_\infty$ is distinguished, $S_n$ is distinguished by construction. Let $S=S_0$. By \Cref{lem:Snquot}, $S_n$ is integral over $S$. Since $R_n$ is a finite $R$-algebra, $S_n$ is topologically of finite type over $R$, so $S_n$ is finite over $S$. 

By \Cref{lem:quotislocal} and \Cref{lem:densunifislocal}, we see that $(S_\bullet, \Gamma_\bullet)$ satisfies (quot), (density) and (unifsec). Thus, we see that $(S_\bullet, \Gamma_\bullet)$ is a good cover. 

\end{proof}

Here, we record an explicit construction of $(S_\bullet, \Gamma_\bullet)$ for a reference.

\begin{cor} \label{cor:expaffrep}
    In \Cref{thm:affrep}, $(S_\bullet, \Gamma_\bullet)$ can be taken so that $S_n = \cl{O}(U_n)\cap S_\infty$ in $\cl{O}(U_\infty)$.
\end{cor}

\subsection{Maximal good covers} \label{ssec:maxlgoodcov}

In this section, we study the properties of good covers satisfying the condition in \Cref{cor:expaffrep}. Here, we repeat the condition. 

\begin{quote}
    (maximal) $R_\infty$ is distinguished and $R_n \to R_\infty \cap \Gamma(\Spa(R_n)^\an,\cl{O}^+)$ is an isomorphism for every $n\geq 0$. 
\end{quote}

The intersection is taken inside $ \Gamma(\Spa(R_\infty)^\an, \cl{O}^+)$. Note that $R_\infty \to \Gamma(\Spa(R_\infty)^\an, \cl{O}^+)$ is an open immersion by \Cref{prop:thickan} and $\Gamma(\Spa(R_n)^\an, \cl{O}^+)\to \Gamma(\Spa(R_\infty)^\an, \cl{O}^+)$ is a closed immersion by the condition (unifsec). 

\begin{defi}
    A pair $(R_\bullet, \Gamma_\bullet)$ is a maximal good cover of a reduced excellent distinguished complete adic ring $R$ over $\bb{Z}_p$ if it is a good cover satisfying (maximal). 
\end{defi}

We say that a subring $A$ of a ring $B$ in characteristic $p$ is purely inseparably integrally closed in $B$ if $x^p \in A$ implies $x\in A$ for every $x\in B$. 

\begin{lem} \label{lem:propoftmaxlgood}
    Let $(R_\bullet,\Gamma_\bullet)$ be a maximal good cover of a reduced excellent distinguished complete adic ring $R$. For every $n\geq 0$, we have $\Gamma(\Spa(R_n)^\an, \cl{O}^{\circ \circ}) \subset R_n \subset \Gamma(\Spa(R_n)^\an, \cl{O}^+)$, $R_n$ is distinguished, and $R_{n,\red}$ is purely inseparably integrally closed in  $\Gamma(\Spa(R_n)^\an, \cl{O}^+)_\red$.
\end{lem}

\begin{proof}
    The first claim follows from \Cref{prop:thickan}, and since $\Gamma(\Spa(R_n)^\an, \cl{O}^+)$ is $I$-torsion free, $R_n$ is $I$-torsion free. We prove the second claim. Let $x \in \Gamma(\Spa(R_n)^\an, \cl{O}^+)_\red$ be an element such that $x^p \in R_{n,\red}$. Take an arbitrary lift $y\in \Gamma(\Spa(R_n)^\an, \cl{O}^+)$ of $x$. Then, $y^p \in R_n$, so in particular, $y^p\in R_\infty$. Since $R_{\infty,\red}$ is perfect, we have $y\in R_\infty$. Thus, $y\in R_\infty \cap \Gamma(\Spa(R_n)^\an, \cl{O}^+)=R_n$ and $x \in R_{n,\red}$. 
\end{proof}

\begin{lem} \label{lem:maxlgoodcovZarloc}
    Let $(R_\bullet,\Gamma_\bullet)$ be a maximal good cover of a reduced excellent distinguished complete adic ring $R$. For every element $f\in R$, $(R_\bullet[\tfrac{1}{f}]^\wedge, \Gamma_\bullet)$ is a maximal good cover of $R[\tfrac{1}{f}]^\wedge$. Here, $(-)^\wedge$ denotes the $I$-adic completion with $I\subset R$ an ideal of definition. 
\end{lem}

\begin{proof}
    Let $R'_n=R_n[\tfrac{1}{f}]^\wedge$. Since $R_\infty$ is thick, $R'_\infty$ is thick. Let $R''_n=R'_\infty \cap \Gamma(\Spa(R'_n)^\an, \cl{O}^+)$. By applying \Cref{cor:expaffrep} to $R'_\infty$, we see that $(R''_\bullet, \Gamma_\bullet)$ is a maximal good cover. It is enough to show that a natural homomorphism $R'_n \to R''_n$ is an isomorphism. By \Cref{lem:propoftmaxlgood} and \Cref{lem:Oooinimg}, we have $\Gamma(\Spa(R'_n)^\an, \cl{O}^{\circ \circ}) \subset R'_n \subset R''_n \subset \Gamma(\Spa(R'_n)^\an, \cl{O}^+)$.

    Since $\Spd(R'_n)$ and $\Spd(R''_n)$ are both geometric quotients of $\Spd(R'_\infty)$ by $\Gamma_\infty^n$ that are qcqs over $\Spd(R)$, we have $\Spd(R''_n) \cong \Spd(R'_n)$ by \Cref{prop:geomquotiscat}. In particular, $R'_{n,\red} \hookrightarrow R''_{n,\red}$ is an isomorphism up to perfection. Since we have $R'_{n,\red} \hookrightarrow R''_{n,\red} \hookrightarrow \Gamma(\Spa(R'_n)^\an, \cl{O}^+)_\red$, it is enough to show that $R'_{n,\red}$ is purely inseparably integrally closed in $\Gamma(\Spa(R'_n)^\an, \cl{O}^+)_\red$. 

    By the proof of \Cref{lem:Oooinimg}, $R'_{n,\red} \hookrightarrow \Gamma(\Spa(R'_n)^\an, \cl{O}^+)_\red$ is isomorphic to $R_{n,\red}[\tfrac{1}{f}] \hookrightarrow \Gamma(\Spa(R_n)^\an, \cl{O}^+)_\red[\tfrac{1}{f}]$. Thus, the claim follows from \Cref{lem:propoftmaxlgood}. 
\end{proof}

\begin{lem}
    Let $R$ be a reduced excellent $I$-adically complete ring. Let $S$ be a finite \'{e}tale $R$-algebra. If $R$ admits a maximal good cover, then $S$ admits a maximal good cover. 
\end{lem}
\begin{proof}
    Let $(R_\bullet, \Gamma_\bullet)$ be a maximal good cover of $R$. Then, $S_n = R_n \otimes_R S$ is reduced and $S_\infty = R_\infty \otimes_R S$ is a complete adic perfectoid ring by \Cref{lem:etperfd}. By the proof of \Cref{lem:analyticfinet}, $(S_\bullet, \Gamma_\bullet)$ is a good cover. It is enough to verify the condition (maximal) for $(S_\bullet, \Gamma_\bullet)$. 

    Since $R_\infty$ is distinguished, $S_\infty$ is also distinguished. Since $R_\infty \cap \Gamma(\Spa(R_n)^\an, \cl{O}) = R_\infty \cap \Gamma(\Spa(R_n)^\an, \cl{O}^+)$ and $\Gamma(\Spa(S_n)^\an, \cl{O}) \cong \Gamma(\Spa(R_n)^\an, \cl{O}) \otimes_R S$, we have 
    \[
        S_n \cong R_n \otimes_R S \cong (R_\infty \cap \Gamma(\Spa(R_n)^\an, \cl{O})) \otimes_R S \cong S_\infty \cap \Gamma(\Spa(S_n)^\an, \cl{O}). 
    \]
    Thus, $(S_\bullet, \Gamma_\bullet)$ satisfies (maximal). 
\end{proof}

Thanks to \Cref{lem:maxlgoodcovZarloc}, the following notion is purely local in the Zariski topology. 

\begin{defi}
    Let $\mfr{X}$ be a reduced excellent distinguished formal scheme over $\bb{Z}_p$. We say that $\mfr{X}$ admits a maximal good cover if $\mfr{X}$ is Zariski locally isomorphic to the formal spectrum of a reduced excellent distinguished complete adic ring admitting a maximal good cover. 
\end{defi}

\begin{lem} \label{lem:uniqmapmaxlgood}
    Let $\mfr{X}$ be a reduced excellent distinguished formal scheme over $\bb{Z}_p$ admitting a maximal good cover. Let $\mfr{Y}$ be a Noetherian formal scheme over $\bb{Z}_p$ and let $f\colon \mfr{X}^\diamond \to \mfr{Y}^\diamond$ be a map over $\Spd(\bb{Z}_p)$ such that the restriction of $f$ to $(\mfr{X}^\an)^\diamond$ is represented by a morphism $g\colon \mfr{X}^\an \to \mfr{Y}^\ad$ of adic spaces. Then, $f$ is represented by a unique morphism $\mfr{X}\to \mfr{Y}$ and its restriction to $\mfr{X}^\an$ is equal to $g$. 
\end{lem}
\begin{proof}
    By following the argument given in the proof of \cite[Theorem 2.16]{AGLR22}, we may assume that $\mfr{X}$ and $\mfr{Y}$ are affine. Let $\mfr{X}=\Spf(R)$ and $\mfr{Y}=\Spf(A)$ with $R$ a reduced excellent distinguished complete adic ring admitting a maximal good cover $(R_\bullet,\Gamma_\bullet)$. By \Cref{lem:perfdrep}, the composition $\Spd(R_\infty) \to \Spd(R) \xrightarrow{f} \Spd(A)$ is represented by a continuous homomorphism $A\to R_\infty$. The induced morphism $\Spa(R_\infty)^\an \to \Spa(A)$ is equal to the composition $\Spa(R_\infty)^\an \to \Spa(R)^\an \xrightarrow{g} \Spa(A)$ since $\Spd(R_\infty)^\an \to \Spd(A)$ is uniquely represented by a morphism $\Spa(R_\infty)^\an \to \Spa(A)$ as $\Spa(R_\infty)^\an$ is a perfectoid space. In particular, $A$ maps to $R_\infty \cap \Gamma(\Spa(R)^\an, \cl{O}^+)=R$. Since $\Spd(R_\infty) \to \Spd(R)$ is surjective, this homomorphism $A\to R$ represents $f$. Since $R\to R_\infty$ is injective, the uniqueness of $f$ follows from \Cref{lem:perfdrep}. By construction, the restriction of $f$ to $\Spa(R)^\an$ is equal to $g$. 
\end{proof}

As a corollary, we obtain a certain uniqueness of the representability of $v$-sheaves. 

\begin{cor} \label{cor:uniqrepexc}
    Let $\mfr{X}$ be a Noetherian formal scheme and let $\mfr{Y}$ and $\mfr{Z}$ be reduced excellent distinguished formal schemes over $\mfr{X}$ admitting maximal good covers. Suppose that $\mfr{Y}^\an \to \mfr{X}^\ad$ and $\mfr{Z}^\an \to \mfr{X}^\ad$ are open immersions of adic spaces and isomorphic to each other. Then, any isomorphism $\mfr{Y}^\diamond \cong \mfr{Z}^\diamond$ over $\mfr{X}^\diamond$ is representable by a unique isomorphism $\mfr{Y}\cong \mfr{Z}$ over $\mfr{X}$. 
\end{cor}
\begin{proof}
    Let $f \colon \mfr{Y}^\diamond \to \mfr{Z}^\diamond$ be an isomorphism over $\mfr{X}^\diamond$. Since the restriction of $f$ to $\mfr{Y}^\an$ is represented by $\mfr{Y}^\an \cong \mfr{Z}^\an \hookrightarrow \mfr{Z}^\ad$, $f$ is uniquely represented by a morphism $f_0 \colon \mfr{Y} \to \mfr{Z}$ lifting $\mfr{Y}^\an \cong \mfr{Z}^\an$ by \Cref{lem:uniqmapmaxlgood}. Moreover, $f_0$ is a morphism over $\mfr{X}$ again by \Cref{lem:uniqmapmaxlgood}. By applying the same argument to the inverse of $f$, we obtain the inverse of $f_0$ over $\mfr{X}$ and we see that $f_0$ is an isomorphism. 
\end{proof}

\subsection{Examples of $p$-adic formal schemes with maximal good covers} \label{ssc:padicexa}

In this section, we present an example of $p$-adically complete rings admitting maximal good covers other than \Cref{prop:smoothadmit}. The result of this section will not be used in the rest of the paper. In this section, the $p$-adic completion is denoted by $(-)^\wedge$. 

\begin{prop} \label{prop:rigetgoodcov}
    Let $R$ be a reduced excellent $p$-adically complete ring. Let $S$ be a $p$-torsion free finite type $R$-algebra that is integrally closed in $S[\tfrac{1}{p}]$ and $S[\tfrac{1}{p}]$ is standard \'{e}tale over $R[\tfrac{1}{p}]$. If $R$ admits a good cover, $S^\wedge$ admits a maximal good cover. 
\end{prop}

\begin{proof}
    Let $(R_\bullet, \Gamma_\bullet)$ be a good cover of $R$. For each $n\geq 0$, let $S_n$ be the integral closure of $S\otimes_R R_n$ in $(S\otimes_R R_n)[\tfrac{1}{p}]$. Since $S_n[\tfrac{1}{p}]$ is \'{e}tale over $R_n[\tfrac{1}{p}]$, $S_n[\tfrac{1}{p}]$ is reduced, so $S_n$ is a reduced finite $S$-algebra. We will show that $(S^\wedge_\bullet, \Gamma_\bullet)$ is a maximal good cover of $S^\wedge$. 

    Sine $S[\tfrac{1}{p}]$ is standard \'{e}tale over $R[\tfrac{1}{p}]$, we may write as $S[\tfrac{1}{p}] = R[x]_{(ph)}/(g)$ with $g$ a monic polynomial in $R[x]$ and $h\in R[x]$ such that $g'$ is invertible in $R[x]_{(ph)}$. We may assume that $x\in S$. Let $f_1,\ldots,f_r \in S$ be generators of $S$ as an $R$-algebra. For some $N\geq 0$, we have $(ph)^N f_i \in R[x]/(g)$ for every $1\leq i \leq r$. Since $(ph)^{-N} \in S[\tfrac{1}{p}]$, we may assume that $(ph)^N f_1$ is a power of $p$. Then, $\Spa(S^\wedge[\tfrac{1}{p}], S^\wedge) \to \Spa(R[x]_{(p)}/(g), R[x]/(g))$ is isomorphic to a rational localization $R(\tfrac{(ph)^N f_1,\ldots,(ph)^N f_r}{(ph)^N})$. 

    By \cite[Theorem 10.11]{BS22}, $\Spd(R_\infty) \times_{\Spd(R)} \Spd(R[x]/(g))$ is represented by a perfectoid ring $R'_\infty$. Then, $\Spd(S^\wedge[\tfrac{1}{p}], S^\wedge) \times_{\Spd(R[x]/(g))} \Spd(R'_\infty)$ is represented by an affinoid perfectoid pair $(S_\infty[\frac{1}{p}], S_\infty)$ obtained as a rational localization $R'_\infty\langle \tfrac{(ph)^N f_1,\ldots,(ph)^N f_r}{(ph)^N}\rangle$. In particular, the pullback of an \'{e}tale morphism $\Spd(S^\wedge[\tfrac{1}{p}], S^\wedge) \to \Spd(R[\tfrac{1}{p}], R)$ to $\Spd(R_\infty[\tfrac{1}{p}], R_\infty)$ is represented by an affinoid perfectoid pair $(S_\infty[\frac{1}{p}], S_\infty)$. 

    Since (density) and (unifsec) are local by \Cref{lem:densunifislocal}, we may argue as in \Cref{lem:goodcovetloc} to show that we have a functorial and uniform family of continuous sections $\{s_{n,U_{n}}\}$ with the density condition. By setting $U_n = \Spa(S_n^\wedge[\tfrac{1}{p}], S_n^\wedge)$, the density condition for $\{s_{n,U_n}\}$ implies that $S_\infty = \colim^\wedge S_n^\wedge$.  In particular, we have the condition (perfd). The condition (quot) can be verified as in \Cref{lem:goodcovetloc}. By construction, $(S_\bullet,\Gamma_\bullet)$ satisfies (maximal). 
\end{proof}

Let $E$ be a complete discrete valuation field over $\bb{Q}_p$ with perfect residue field and let $O_E$ be the ring of integers of $E$. 

\begin{cor} \label{cor:maxlgood}
    Let $R$ be a $p$-torsion free normal finite type $O_E$-algebra. If $R[\tfrac{1}{p}]$ is standard \'{e}tale over a polynomial algebra over $E$, $R^\wedge$ admits a maximal good cover.
\end{cor}
\begin{proof}
    By assumption, there is a homomorphism $P\to R$ from a polynomial algebra $P$ over $O_E$ such that $P[\tfrac{1}{p}] \to R[\tfrac{1}{p}]$ is standard \'{e}tale. Thus, the claim follows from \Cref{prop:exadmit} and \Cref{prop:rigetgoodcov}. 
\end{proof}

For applications to integral models of Shimura varieties, it is natural to hope for a generalization to the case where $R[\tfrac{1}{p}]$ is smooth over $E$. The main technical difficulty lies in the condition (unifsec). It implies that $(R^\wedge[\tfrac{1}{p}], R^\wedge)$ is sousperfectoid in the sense of \cite{HK22}, but it is not known yet whether every smooth affinoid $E$-algebra is sousperfectoid. Here, we record some weaker results. 

\begin{cor}
    Let $R$ be a $p$-torsion free normal finite type $O_E$-algebra. If $R[\tfrac{1}{p}]$ is smooth over $E$, there is a proper morphism $Y\to \Spec(R)$ such that $Y[\tfrac{1}{p}] \cong \Spec(R[\tfrac{1}{p}])$ and $Y^\wedge$ admits a maximal good cover. 
\end{cor}
\begin{proof}
    We can take a set of elements $f_1,\ldots,f_r \in pR$ so that $\Spec(R[\tfrac{1}{p}])$ is covered by $D(f_i)$ for $1\leq i \leq r$ and each $D(f_i)$ is standard \'{e}tale over a polynomial algebra over $E$. Let $Y\to \Spec(R)$ be the blowup along $V(f_1,\ldots,f_r)$ and let $Y^+$ be the normalization of $Y$ in $\Spec(R[\tfrac{1}{p}])$. Let $V_i^+ \subset Y^+$ be the locus where $f_i$ divides $f_j$ for $1\leq j \leq r$. We have $V_i^+[\tfrac{1}{p}]=D(f_i)$, so $(V_i^+)^\wedge$ admits a maximal good cover by \Cref{cor:maxlgood}. Thus,  $(Y^+)^\wedge$ admits a maximal good cover. 
\end{proof}

\begin{cor}
    Let $R$ be a $p$-torsion free normal domain of finite type over $O_E$. If $R[\tfrac{1}{p}]$ is smooth over $E$, and $R/pR$ and $R[\tfrac{1}{p}]$ are of the same dimension, there is a nonempty open subset $U\subset \Spec(R/pR)$ such that $\Spec(R)^\wedge_{/U}$ admits a maximal good cover. 
\end{cor}
\begin{proof}
    We can take an affine open subset $V\subset \Spec(R[\tfrac{1}{p}])$ that is standard \'{e}tale over an affine space over $E$. Let $Z$ be the Zariski closure of $\Spec(R[\tfrac{1}{p}])-V$ in $\Spec(R)$. Since $R/pR$ is of the same dimension as $R[\tfrac{1}{p}]$, $Z\otimes_{O_E} (O_E/pO_E)$ has a strictly smaller dimension than $\Spec(R/pR)$. In particular, $\Spec(R/pR)-Z$ is nonempty. Thus, we can take a distinguished open subset $D_f \subset \Spec(R)-Z$ such that $D_f \cap \Spec(R/pR)$ is nonempty. By construction, $D_f[\tfrac{1}{p}]$ is standard \'{e}tale over an affine space over $E$, so $D_f^\wedge$ admits a maximal good cover. 
\end{proof}

\subsection{Relative representability}

The aim of this section is to prove our main representability criterion by showing that the construction in \Cref{cor:expaffrep} can be glued in the Zariski topology.

We say that a perfect scheme $X$ over an affine perfect scheme $\Spec(A)$ is perfectly quasi-projective over $A$ if $X$ is perfectly of finite type over $A$ and admits an ample line bundle. 

\begin{prop} \label{thm:dilatation}
    Let $\mfr{X}$ be a reduced excellent formal scheme over $\bb{Z}_p$ admitting a good cover. Let $Y\to \mfr{X}^\diamond$ be a quasiseparated map of small $v$-sheaves. Let $U \subset \mfr{X}^\an$ be a quasicompact open immersion with a lift $U^\diamond \to Y$. Suppose that the following condition holds. 
    \begin{quote}
        For every complete adic perfectoid ring $R$ with an adic morphism $\Spf(R)\to \mfr{X}$, there is a unique distinguished perfectoid formal $R$-scheme $\mfr{Y}_R$ over $Y$ such that $\mfr{Y}_R$ is adic over $R$, $\mfr{Y}_{R, \red}$ is perfectly quasi-projective over $R_\red$, $\mfr{Y}_R^\diamond \to Y\times_{\mfr{X}^\diamond} \Spd(R)$ is a closed immersion and there is an isomorphism $\mfr{Y}_R^\an \cong U\times_{\mfr{X}^\an} \Spd(R)^\an$ that represents the base change of $U^\diamond \to Y$ to $\Spd(R)^\an$. 
    \end{quote}
    Then, the $v$-closure of $U^\diamond$ in $Y$ is representable by a unique reduced excellent distinguished formal scheme $\mfr{Y}$ admitting a maximal good cover with $\mfr{Y}^\an \cong U$. Moreover, $\mfr{Y}$ is separated and topologically of finite type over $\mfr{X}$. 
\end{prop}

\begin{rmk}
    The uniqueness of $\mfr{Y}_R$ is a mild condition. For example, if $Y\times_{\mfr{X}^\diamond} \Spd(R)$ is a kimberlite, $\mfr{Y}_R$ represents the thick closure of $U\times_{\mfr{X}^\an} \Spd(R)^\an$ in $Y\times_{\mfr{X}^\diamond} \Spd(R)$ and the uniqueness follows from \Cref{lem:perfdrep}. 
\end{rmk}

\begin{proof}
    First, suppose that $\mfr{X}=\Spf(R)$ where $R$ admits a good cover $(R_\bullet, \Gamma_\bullet)$. Let $\mfr{Y}_\infty$ be the unique perfectoid formal $R_\infty$-scheme over $Y$ satisfying the given condition. For every $\gamma \in \Gamma_\infty$, $\mfr{Y}_\infty \otimes_{R_\infty, \gamma} R_\infty$ also satisfies the given condition, so there is a unique isomorphism $\mfr{Y}_\infty \otimes_{R_\infty, \gamma} R_\infty \cong \mfr{Y}_\infty$. Thus, we have an action of $\Gamma_\infty$ on $\mfr{Y}_\infty$. 

    Since $\mfr{Y}_{\infty, \red}$ is perfectly of finite type over $R_{\infty, \red}$, we can take a finite affine open covering $\{ \mfr{V}_{\infty,i} \}_{i\in I}$ of $\mfr{Y}_\infty$. Let $U_n=U\times_{\mfr{X}^\an} \Spa(R_n)^\an$ and $V_{\infty,i}=\mfr{V}_{\infty,i}^\an \subset \Spa(R_\infty)^\an$. Then, $\{V_{\infty,i}\}_{i\in I}$ is a finite cover of $U_\infty$. 
    By \cite[Lemma 12.17]{Sch17}, that cover can be defined over $U_m$ for sufficiently large $m\geq 0$. In particular, $V_{\infty,i}$ is stable under $\Gamma_\infty^m$ for every $i\in I$. Since $\mfr{V}_{\infty,i}$ is distinguished, we see by looking at the specialization map that $\mfr{V}_{\infty,i}$ is also stable under $\Gamma_\infty^m$. Thus, we may apply \Cref{thm:affrep} to $\mfr{V}_{\infty,i}$ and obtain a maximal good cover $(S_{\bullet, i}, \Gamma_{\bullet}^m)_{\bullet \geq m}$ by \Cref{cor:expaffrep}. 

    We show that the collection $\{\Spf(S_{m,i})\}_{i\in I}$ can be glued to a formal scheme $\mfr{Y}_m$. For every $i,j\in I$, $\mfr{V}_{\infty,i} \cap \mfr{V}_{\infty,j}$ is stable under $\Gamma_\infty^m$. Since $(S_{m,i})_{\red}^\perf \to (S_{\infty, i})_\red$ is integral, $\Spf(S_{\infty,i}) \to \Spf(S_{m,i})$ is closed, so there is an open subspace $W_{m,i}^j \subset \Spf(S_{m,i})$ whose pullback $W_{\infty,i}^j \subset \mfr{V}_{\infty,i}$ equals $\mfr{V}_{\infty,i} \cap \mfr{V}_{\infty,j}$. Similarly, there is an open subspace $W_{m,j}^i \subset \Spf(S_{m,j})$ whose pullback $W_{\infty,j}^i \subset \mfr{V}_{\infty,j}$ equals $\mfr{V}_{\infty,i} \cap \mfr{V}_{\infty,j}$. By construction, $W_{m,i}^j$ and $W_{m,j}^i$ admit maximal good covers, and $W_{m,i}^{j,\an} \cong W_{m,j}^{i,\an}$ as open subspaces of $\Spa(R_n)^\an$. Moreover, $(W_{m,i}^j)^\diamond$ and $(W_{m,j}^i)^\diamond$ are geometric quotients of $(\mfr{V}_{\infty,i} \cap \mfr{V}_{\infty,j})^\diamond$ by $\Gamma_\infty^m$ that are qcqs over $\Spd(R_m)$, so we have $(W_{m,i}^j)^\diamond \cong (W_{m,j}^i)^\diamond$ by \Cref{prop:geomquotiscat}. It follows from \Cref{cor:uniqrepexc} that there is a unique isomorphism $W_{m,i}^j \cong W_{m,j}^i$ of formal schemes compatible with the isomorphisms on the analytic loci and the $v$-sheafifications. Thus, we obtain a formal scheme $\mfr{Y}_m$ by gluing the affine formal schemes $\Spf(S_{m,i})$ for $i\in I$.
    
    By construction, $\mfr{Y}_m^\diamond$ is a geometric quotient of $\mfr{Y}_\infty^\diamond$ by $\Gamma_\infty^m$ that is qcqs over $\Spd(R)$. Since $Y$ is quasiseparated over $\Spd(R)$, we obtain a map $\mfr{Y}_m^\diamond \to Y\times_{\Spd(R)} \Spd(R_m)$ by \Cref{prop:geomquotiscat}. By \Cref{cor:transsurj}, $\mfr{Y}_\infty^\diamond \to \mfr{Y}_m^\diamond \times_{\Spd(R_m)} \Spd(R_\infty)$ is surjective. It is also injective since $\mfr{Y}_\infty^\diamond \to Y\times_{\Spd(R)} \Spd(R_\infty)$ factors through $\mfr{Y}_m^\diamond \times_{\Spd(R_m)} \Spd(R_\infty)$. Thus, the base change of $\mfr{Y}_m^\diamond \to Y\times_{\Spd(R)} \Spd(R_m)$ to $\Spd(R_\infty)$ is a closed immersion, so it is a closed immersion itself by \cite[Proposition 10.11]{Sch17}. By construction, $\mfr{Y}_m^\an \cong U_m$ and $\mfr{Y}_m$ is distinguished, so $U_m^\diamond$ is dense in $\mfr{Y}_m^\diamond$ by \Cref{lem:thickvclos}. Thus, the $v$-closure of $U_m$ in $Y\times_{\Spd(R)} \Spd(R_m)$ is representable by $\mfr{Y}_m$. By \Cref{thm:affrep}, $\mfr{Y}_m$ admits a maximal good cover and is topologically of finite type over $R_m$. Moreover, since $\mfr{Y}_\infty$ is quasi-projective over $R_\red$, and in particular, separated, $\mfr{V}_{\infty,i} \cap \mfr{V}_{\infty,j}$ is affine. By applying \Cref{thm:affrep} to $\mfr{V}_{\infty,i} \cap \mfr{V}_{\infty,j}$, we see that $W_{m,i}^j$ is affine for every $i,j\in I$ and $\mfr{Y}_m$ is separated. The uniqueness of $\mfr{Y}_m$ follows from \Cref{cor:uniqrepexc}. 

    Now, we will show that we can take $m=0$ if we choose a suitable affine open cover $\{\mfr{V}_{\infty,i}\}_{i\in I}$. It is enough to choose so that each $\mfr{V}_{\infty,i}$ is stable under $\Gamma_{\infty}$. Since $\mfr{Y}_{\infty,\red}$ is perfectly quasi-projective over $R_{\infty,\red}$, $\mfr{Y}_{\infty,\red}$ admits an ample line bundle. Since $\mfr{Y}_{\infty,\red} = \lim_{n\geq m} \mfr{Y}_{n,\red}^\perf$, that line bundle can be defined over $\mfr{Y}_{n,\red}^\perf$ for sufficiently large $n\geq m$. Then, $\mfr{Y}_{n,\red}$ is quasi-projective over $R_\red$, so any $\Gamma_n$-orbit of $\mfr{Y}_{n,\red}$ admits a $\Gamma_n$-stable affine open neighborhood. In particular, $\mfr{Y}_n$ admits a $\Gamma_n$-stable finite affine covering. Its pullback to $\mfr{Y}_\infty$ yields a desired $\Gamma_\infty$-stable finite affine covering. 

    Finally, we treat the general case. Take an affine covering of $\mfr{X}$ by $\Spf(R)$ with $R$ admitting a good cover. The $v$-closure of $U^\diamond\times_{\mfr{X}^\diamond} \Spd(R)$ in $Y\times_{\mfr{X}^\diamond} \Spd(R)$ is representable by a unique reduced excellent distinguished formal scheme $\mfr{Y}_R$ admitting a maximal good cover with $\mfr{Y}_R^\an \cong U\times_{\mfr{X}^\an} \Spa(R)^\an$. By the uniqueness of $\mfr{Y}_R$, the collection $\{\mfr{Y}_R\}$ can be glued to a formal scheme $\mfr{Y}$ over $\mfr{X}$. It is easy to see that $\mfr{Y}$ satisfies the desired condition. 
\end{proof}

\begin{rmk}
    If $U=\mfr{X}^\an$, $\mfr{Y}$ is automatically proper by \Cref{lem:dilatationproper}. In addition, if $\mfr{X}$ is locally monogeneous, e.g. if $\mfr{X}$ is $p$-adic, $\mfr{Y}\to \mfr{X}$ is a formal modification by \Cref{prop:formalmod}. 
\end{rmk}

\begin{lem} \label{rmk:existenceY}
    In the setting of \Cref{thm:dilatation}, suppose $\mfr{X} = \Spf(R)$ with $R$ admitting a good cover $(R_\bullet, \Gamma_\bullet)$. Then,  the same claim holds even if the given condition holds only for $\Spf(R_\infty) \to \mfr{X}$. Moreover, if $Y \times_{\Spd(R)} \Spd(R_\red)$ is representable by a separated perfectly finite type $R_\red$-scheme, we do not need the perfect quasi-projectivity assumption on $\mfr{Y}_{R_\infty, \red}$. 
\end{lem}
\begin{proof}
    The first claim is immediate from the proof. For the second claim, let $Y^\red$ be the perfect $R_\red$-scheme representing $Y \times_{\Spd(R)} \Spd(R_\red)$. It is enough to show that $\mfr{Y}_{R_\infty, \red}$ admits a $\Gamma_\infty$-stable affine open covering. Since $\mfr{Y}_{R_\infty, \red}^\diamond \to (Y^\red\otimes_{R^\perf_\red} R_{\infty, \red})^\diamond$ is a closed immersion, $\mfr{Y}_{R_\infty, \red} \to Y^\red \otimes_{R^\perf_\red} R_{\infty, \red}$ is also a closed immersion by \cite[Lemma 3.31]{Gle24}. In particular, $\mfr{Y}_{R_\infty, \red} \to Y^\red$ is affine. Thus, if we take a finite affine open covering of $Y^\red$, its pullback to $\mfr{Y}_{R_\infty, \red}$ is a $\Gamma_\infty$-stable finite affine open covering. 
\end{proof}


The following corollary shows the representability of $v$-sheaf theoretic modifications. 

\begin{cor} \label{cor:vshfmodif}
    Let $R$ be a reduced excellent complete adic ring admitting a good cover $(R_\bullet,\Gamma_\bullet)$. Let $Y$ be a thick prekimberlite formally adic over $\Spd(R)$ with $Y^\an \cong \Spd(R)^\an$. Suppose that $Y^\red$ is perfectly of finite type over $R_\red$ and $Y\times_{\Spd(R)} \Spd(R_\infty)$ is representable by a perfectoid formal scheme adic over $R_\infty$. Then, $Y$ is representable by a unique proper formal $R$-scheme $\mfr{Y}$ admitting a maximal good cover with $\mfr{Y}^\an \cong \Spa(R)^\an$. 
\end{cor}
\begin{proof}
    Since $Y$ is formally separated, $Y^\red$ is separated. By \Cref{rmk:existenceY}, it is enough to check the condition for $R_\infty$. Let $\mfr{Y}_\infty$ be the perfectoid formal scheme representing $Y\times_{\Spd(R)} \Spd(R_\infty)$. It is enough to show that $\mfr{Y}_\infty$ is thick. By \Cref{lem:perfdrep}, $\Gamma_\infty$ acts on $\mfr{Y}_\infty$. The open subspace $V=\mfr{Y}_{\infty, \red} - \spc_{\mfr{Y}_\infty}(\mfr{Y}_\infty^\an) \subset \mfr{Y}_{\infty, \red}$ is stable under $\Gamma_\infty$. By construction, $V^\diamond \subset \mfr{Y}_\infty^\diamond$ is open and stable under $\Gamma_\infty$. Let $V_0 \subset \lvert Y \rvert$ be the image of $V^\diamond$. By \cite[Proposition 12.9]{Sch17}, $\lvert \mfr{Y}_\infty^\diamond \rvert \to \lvert Y \rvert$ is a quotient map, so $V_0$ is open. By construction, $V_0^\red \cap \spc_Y(Y^\an) = \phi$, so $V_0^\red$ is empty as $Y$ is thick. Since there is a map $V\to V_0^\red$, we see that $V$ is empty. 
\end{proof}

\subsection{Variant}

In this section, we provide another form of our relative representability criterion that is suitable for application. Here, we only deal with $p$-adic formal schemes. For a $v$-sheaf $X$ over $\Spd(\bb{Z}_p)$, $X_\eta$ denotes $X\times_{\Spd(\bb{Z}_p)} \Spd(\bb{Q}_p,\bb{Z}_p)$ and $X_s$ denotes $X\times_{\Spd(\bb{Z}_p)} \Spd(\bb{F}_p)$. Recall the notation in \Cref{ssec:basicvsh} and \Cref{ssc:adicperfd}. 

\begin{prop} \label{prop:finetrep}
    Let $\mfr{X}$ be a reduced excellent $p$-adic formal scheme admitting a good cover. Let $Y\to \mfr{X}^\diamond$ be a quasiseparated map of small $v$-sheaves. Let $Z \subset Y_\eta$ be a closed subsheaf that is finite \'{e}tale over $\mfr{X}_\eta^\diamond$. Suppose that the following condition holds. 
    \begin{quote}
        For every perfectoid ring $R$ with $\Spf(R)\to \mfr{X}$, there is a perfectly proper $R^\flat$-scheme $Y_{R^\flat}$ and a finite \'{e}tale $R^\flat[1/\xi_{R,0}]$-scheme $Z_{R^\flat[1/\xi_{R,0}]}$ with a morphism $Z_{R^\flat[1/\xi_{R,0}]} \to  Y_{R^\flat}$ such that we have the following commutative diagram
        \begin{center}
            \begin{tikzcd}
                Z \times_{\mfr{X}_\eta^\diamond} \Spd(R)^\an \ar[r] \ar[d,"\cong"] & Y \times_{\mfr{X}^\diamond} \Spd(R) \ar[d,"\cong"] \\
                Z_{R^\flat[1/\xi_{R,0}]}^{\diamondsuit}\times_{\Spec(R^\flat[1/\xi_{R,0}])^\diamondsuit} \Spd(R^\flat)^\an \ar[r] & Y_{R^\flat}^\diamond\times_{\Spec(R^\flat)^\diamond} \Spd(R^\flat). 
            \end{tikzcd}
        \end{center}
        that is compatible with tilting equivalence $\Spd(R) \cong \Spd(R^\flat)$. 
    \end{quote}
    Then, $Z^\clos$ is representable by a unique reduced excellent distinguished $p$-adic formal scheme $\mfr{Y}$ admitting a maximal good cover that is proper over $\mfr{X}$ with $\mfr{Y}_\eta$ finite \'{e}tale over $\mfr{X}_\eta$. 
\end{prop}

\begin{proof}
    First, suppose that $\mfr{X}=\Spf(R)$ with $R$ admitting a good cover $(R_\bullet, \Gamma_\bullet)$. Let $Z_{R_\infty^\flat}$ be the scheme-theoretic image of $Z_{R_\infty^\flat[1/\xi_{R_\infty,0}]} \to Y_{R_\infty^\flat}$. Let $\mfr{Y}_{R_\infty^\flat}$ (resp.\ $\mfr{Z}_{R_\infty^\flat}$) be the $\xi_{R_\infty,0}$-adic completion of $Y_{R_\infty^\flat}$ (resp.\ $Z_{R_\infty^\flat}$). Since $Z_{R_\infty^\flat}$ is $\xi_{R_\infty,0}$-torsion free, $\mfr{Z}_{R_\infty^\flat}$ is distinguished. Thus, $\mfr{Z}_{R_\infty^\flat}$ is thick by \Cref{cor:nzdgenerator}. Let $\mfr{Y}_{R_\infty}$ (resp.\ $\mfr{Z}_{R_\infty}$) be the untilt of $\mfr{Y}_{R_\infty^\flat}$ (resp.\ $\mfr{Z}_{R_\infty^\flat}$) over $R_\infty$ (see \Cref{prop:tiltperfdfmsch}). Then, $\mfr{Y}_{R_\infty}$ represents $Y\times_{\mfr{X}^\diamond} \Spd(R_\infty)$ and $\mfr{Z}_{R_\infty}^\diamond \subset \mfr{Y}_{R_\infty}^\diamond$ is a thick closed subsheaf. We will show that $\mfr{Z}_{R_\infty,\eta}$ represents $Z\times_{\mfr{X}_\eta^\diamond} \Spd(R_\infty)^\an$. 

    By \Cref{lem:perfprop}, we have $\mfr{Z}_{R_\infty^\flat}^{\diamond,\an} \cong Z_{R_\infty^\flat}^\diamond \times_{\Spec(R_\infty^\flat)^\diamond} \Spd(R_\infty^\flat)^\an \cong Z_{R_\infty^\flat}^\diamondsuit \times_{\Spec(R_\infty^\flat)^\diamondsuit} \Spd(R_\infty^\flat)^\an \cong (Z_{R_\infty^\flat}[1/\xi_{R_\infty,0}])^{\diamondsuit}\times_{\Spec(R_\infty^\flat[1/\xi_{R,0}])^\diamondsuit} \Spd(R_\infty^\flat)^\an$. Since $Z_{R_\infty^\flat[1/\xi_{R,0}]} \to Z_{R_\infty^\flat}[1/\xi_{R_\infty,0}]$ is finite and surjective, $Z_{R_\infty^\flat[1/\xi_{R,0}]}^\diamondsuit \to Z_{R_\infty^\flat}[1/\xi_{R_\infty,0}]^\diamondsuit$ is surjective on geometric points. In particular, $Z\times_{\mfr{X}_\eta^\diamond} \Spd(R_\infty)^\an \to \mfr{Z}_{R_\infty,\eta}^\diamond$ is surjective on geometric points. Since $Z\subset Y_\eta$, it is also injective, so $Z\times_{\mfr{X}_\eta^\diamond} \Spd(R_\infty)^\an \cong \mfr{Z}_{R_\infty,\eta}^\diamond$. Thus, $\mfr{Z}_{R_\infty}$ represents the thick closure of $Z\times_{\mfr{X}_\eta^\diamond} \Spd(R_\infty)^\an$ in $Y\times_{\mfr{X}^\diamond} \Spd(R_\infty)$. 

    By the uniqueness of thick closures, $\mfr{Z}_{R_\infty}^\diamond \subset Y\times_{\mfr{X}^\diamond} \Spd(R_\infty)$ is stable under $\Gamma_\infty$. Since $\Spd(R_\infty) \to \Spd(R)$ is proper (as in \Cref{lem:transquot}), the image $Z^\ints\subset Y$ of $\mfr{Z}_{R_\infty}^\diamond$ under the projection to $Y$ is closed. By \Cref{lem:quottosubsh}, we have $\mfr{Z}_{R_\infty}^\diamond \cong Z^\ints \times_{\mfr{X}^\diamond} \Spd(R_\infty)$ and $Z^\ints$ is a geometric quotient of $\mfr{Z}_{R_\infty}^\diamond$ by $\Gamma_\infty$. 

    Let $R^+ \subset R[\tfrac{1}{p}]$ be the integral closure of $R$. Since $Z$ is finite \'{e}tale over $\mfr{X}_\eta^\diamond$, there is a unique finite \'{e}tale homomorphism $(R[\tfrac{1}{p}],R^+) \to (S^+[\tfrac{1}{p}],S^+)$ such that $Z \cong \Spd(S^+[\tfrac{1}{p}],S^+)$ by \cite[Lemma 15.6]{Sch17}. Let $(R_\infty[\tfrac{1}{p}],R_\infty^+) \to (S_\infty^+[\tfrac{1}{p}],S_\infty^+)$ be the base change of the finite \'{e}tale homomorphism. By \Cref{prop:thickan}, we have $(S_\infty^+)^{\circ \circ} \subset \Gamma(\mfr{Z}_{R_\infty}, \cl{O}) \subset S_\infty^+$. Let $S_0 =R+(S^+)^{\circ \circ}\subset S^+$. Since $R$ is reduced and excellent, $S^+$ is finite over $R$ and $S_0$ is finite over $R$. Moreover, we have a $\Gamma_\infty$-invariant morphism $\mfr{Z}_{R_\infty} \to \Spf(S_0)$, so by \Cref{prop:geomquotiscat}, we get a morphism $Z^\ints \to \Spd(S_0)$. By the proof of \Cref{lem:analyticfinet}, the reduction of $(S_0 \otimes_R R_\bullet, \Gamma_\bullet)$ is a good cover and let $S_\infty$ be as in the condition (perfd). Since $\Spd(S_\infty) \cong \Spd(R_\infty)\times_{\Spd(R)} \Spd(S_0)$, we have $Z^\ints \times_{\Spd(S_0)} \Spd(S_\infty) \cong \mfr{Z}_{R_\infty}^\diamond$. By \Cref{rmk:existenceY}, the proof of \Cref{thm:dilatation} applies to $Z^\ints \to \Spd(S_0)$ and we obtain the desired representability. 
    
    As in \Cref{cor:uniqrepexc}, the uniqueness of $\mfr{Y}$ follows from \Cref{lem:uniqmapmaxlgood} thanks to \cite[Lemma 15.6]{Sch17}. Thus, in the general case, we can glue $\mfr{Y}_R$ for each $\Spf(R) \hookrightarrow \mfr{X}$ as above and the claim follows. 
\end{proof}

\section{Applications to local shtukas} \label{sec:locsht}

\subsection{Local shtukas} \label{ssec:introlocsht}

In this section, we set up some notation and recall the definition of local shtukas introduced in \cite{SW20}. 

Let $F$ be a non-archimedean local field over $\bb{Q}_p$ and let $O_F$ be the ring of integers of $F$ with a residue field $k$. Let $\pi$ be a uniformizer of $F$ and let $q$ be the number of elements of $k$. For a perfect $k$-algebra $R$, we set $W_{O_F}(R)=W(R)\otimes_{W(k)} O_F$. 

Let $S=\Spa(R,R^+)$ be an affinoid perfectoid space over $k$ and let $\varpi\in R^+$ be a pseudo-uniformizer. We set $\cl{Y}_S=\Spa(W_{O_F}(R^+))-V([\varpi])$. It is an analytic adic space (see \cite[Proposition II.1.1]{FS21}). Let $\cl{Y}_{S,[0,n]}\subset \cl{Y}_S$ be the locus $\{\lvert \pi^n \rvert \leq \lvert [\varpi] \rvert \neq 0 \}$ for a positive integer $n$. By the proof of \cite[Proposition II.1.1]{SW20}, $\cl{Y}_{S,[0,n]}=\Spa(B_{S,[0,n]},B^+_{S,[0,n]})\subset \cl{Y}_S$ is an affinoid open subspace and the $p$-adic completion $B_{S,[0,n]}^\wedge$ of $B_{S,[0,n]}$ is isomorphic to $W_{O_F}(R)$. By \cite[Proposition II.1.3]{FS21}, we can glue this construction analytic locally and obtain an analytic adic space $\cl{Y}_S$ for every perfectoid space $S$ over $k$. 

Let $G$ be a connected reductive group over $F$ and let $\cl{G}$ be a parahoric group scheme of $G$ over $O_F$.

\begin{defi}(\cite[Section 23.1]{SW20})
    Let $S$ be a perfectoid space over $k$. A local $\cl{G}$-shtuka over $S$ is a tuple $((S^\sharp,\iota), \cl{P}, \varphi_{\cl{P}})$ such that $(S^\sharp, \iota)\in \Spd(O_F)(S)$ is an untilt of $S$ over $O_F$, $\cl{P}$ is a $\cl{G}$-torsor over $\cl{Y}_S$ and $\varphi_{\cl{P}}\colon \varphi^*\cl{P}\vert_{\cl{Y}_S\backslash S^\sharp} \cong \cl{P}\vert_{\cl{Y}_S\backslash S^\sharp}$ is an isomorphism of $\cl{G}$-torsors meromorphic along $S^\sharp$. 
\end{defi}

Let $\Perf_k$ be the $v$-site of perfectoid spaces over $k$. The functor taking a perfectoid space $S$ over $k$ to the groupoid of local $\cl{G}$-shtukas over $S$ is a $v$-stack on $\Perf_k$ by \cite[Proposition 19.5.3]{SW20} and denoted by $\Sht_\cl{G}$ (cf.\ \cite[Definition 23.1.1]{SW20}). There is a natural map $\Sht_{\cl{G}}\to \Spd(O_F)$ of $v$-stacks on $\Perf_k$. Note that a $v$-stack on $\Perf_k$ is identified with a $v$-stack on $\Perf$ with a map to $\Spd(k)$. For a $v$-stack $X$ on $\Perf_k$, we say that a map $X\to \Sht_{\cl{G}}$ of $v$-stacks on $\Perf_k$ is a local $\cl{G}$-shtuka over $X$. For a $v$-stack $X$ over $\Spd(O_F)$, we say that a map $X\to \Sht_{\cl{G}}$ of $v$-stacks over $\Spd(O_F)$ is a local $\cl{G}$-shtuka over $X_{/O_F}$. 

\begin{prop}\label{prop:mapLW}
    Let $(L_W^+\cl{G})^\diamondsuit$ be the functor taking an affinoid perfectoid space $\Spa(R,R^+)$ over $k$ to the group $\cl{G}(W_{O_F}(R))$. It is a small $v$-sheaf on $\Perf_k$ and there is a natural map $\Sht_{\cl{G}}\to [\ast/(L_W^+\cl{G})^\diamondsuit]$ of $v$-stacks on $\Perf_k$. Here, $\ast$ denotes the $v$-sheaf $\Spd(k)$. 
\end{prop}

\begin{proof}
    Since $\cl{G}$ is affine, $(L_W^+\cl{G})^\diamondsuit$ is a small $v$-sheaf on $\Perf_k$ by \cite[Theorem 8.7]{Sch17}. The quotient $v$-stack $[\ast/(L_W^+\cl{G})^\diamondsuit]$ classifies the set of $\cl{G}$-torsors on $\Spec(W_{O_F}(R))$ over an affinoid perfectoid space $S=\Spa(R,R^+)$ over $k$ since $\cl{G}$-torsors on $W_{O_F}(R)$ satisfy $v$-descent by \cite[Corollary 17.1.9]{SW20}. As in \cite[Section 20.3]{SW20}, we obtain a map $\Sht_{\cl{G}}\to [\ast/(L_W^+\cl{G})^\diamondsuit]$ by pulling back $\cl{G}$-torsors on $\cl{Y}_{S,[0,n]}$ to $W_{O_F}(R)$. 
\end{proof}

Let $R$ be a perfectoid ring over $O_F$. The Frobenius map on $W_{O_F}(R^\flat)$ is denoted by $\varphi_R$. By \cite[Lemma 2.4.2]{Ito25b}, the kernel of the surjection $\theta_R\colon W_{O_F}(R^\flat)\to R$ is principal. By an abuse of notation, a generator of $\Ker(\theta_R)$ is also denoted by $\xi_R$. The image of $\xi_R$ under the projection $W_{O_F}(R^\flat) \to R^\flat$ is denoted by $\xi_{R,0}$. 

\begin{defi}
    A $\cl{G}$-Breuil-Kisin-Fargues ($\cl{G}$-BKF) module over a perfectoid ring $R$ over $O_F$ is a pair $(P,\varphi_P)$ such that $P$ is a $\cl{G}$-torsor over $W_{O_F}(R^\flat)$ and $\varphi_P$ is an isomorphism $(\varphi_R^*P)[1/\xi_R] \cong P[1/\xi_R]$. 
\end{defi}

\begin{prop}\textup{(\cite{Gut23})}\label{prop:GBKF}
    For a perfectoid ring $R$ over $O_F$, the category of local $\cl{G}$-shtukas over $\Spd(R)_{/O_F}$ is equivalent to the category of $\cl{G}$-BKF modules over $R$. 
\end{prop}

The above result is also proved in \cite{GIZ25} when $R$ is in characteristic $p$. 

\subsection{Level structures}

In this section, we introduce parahoric level structures on local shtukas. Let $\cl{G}'$ be another parahoric group scheme of $G$ over $O_F$ with a morphism $\cl{G}'\to \cl{G}$ over $O_F$. It corresponds to a parahoric subgroup $\cl{G}'(O_F)\subset \cl{G}(O_F)$. 

\begin{lem}\label{lem:kerincl}
    For a $p$-torsion free $O_F$-algebra $A$, we have $\Ker(\cl{G}(A)\to \cl{G}(A/\pi A)) \subset \cl{G}'(A)$. 
\end{lem}
\begin{proof}
    Let $\cl{G}_{1}$ be the first congruence group scheme of $\cl{G}$ as in \cite[Definition A.5.12]{KP23}. Since $A$ is $p$-torsion free, we have $\cl{G}_1(A)=\Ker(\cl{G}(A)\to \cl{G}(A/\pi A))$. We show that there is a morphism $\cl{G}_1\to \cl{G}'$ over $O_F$. Let $\breve{F}$ be the completion of the maximal unramified extension of $F$ and let $O_{\breve{F}}$ be the ring of integers in $\breve{F}$. By \cite[Corollary 2.10.10]{KP23}, it is enough to show that $\Ker(\cl{G}(O_{\breve{F}}) \to \cl{G}(O_{\breve{F}}/\pi O_{\breve{F}})) \subset \cl{G}'(O_{\breve{F}})$. The claim follows from \cite[Corollary 8.4.12, Corollary 8.4.13]{KP23}. 
\end{proof}

\begin{defi}
    Let $H_1\to H_2$ be a homomorphism of group schemes over a scheme $S$ and let $P_2$ be an $H_2$-torsor over $S$. An $H_1$-subtorsor of $P_2$ over $S$ is a pair $(P_1,\iota)$ such that $P_1$ is an $H_1$-torsor over $S$ and $\iota$ is an isomorphism $P_1\times^{H_1} H_2 \cong P_2$ of $H_2$-torsors. 
\end{defi}

For a section $s\colon S\to P_2$, we say that the $H_1$-subtorsor $(H_1, s\cdot (-)\colon H_2 \cong P_2)$ is the $H_1$-subtorsor generated by $s$. For any $h_1\in H_1(S)$, the $H_1$-subtorsor generated by $s\cdot h_1$ is isomorphic to the $H_1$-subtorsor generated by $s$. 

By \cite[Theorem 8.4.19 (2)]{KP23}, the image of $\cl{G}'_k$ in $\cl{G}_k$ is a parabolic subgroup of $\cl{G}_k$. Let $\cl{H}\subset \cl{G}_k$ be the parabolic subgroup.

\begin{lem}\label{lem:G'subtors}
    Let $A$ be a $p$-torsion free and $p$-complete $O_F$-algebra and let $P$ be a $\cl{G}$-torsor over $A$. There is a natural bijection between the set of $\cl{G}'$-subtorsors of $P$ and the set of $\cl{H}$-subtorsors of $P_{A/\pi A}$. 
\end{lem}
\begin{proof}
    The one direction is clear by taking the base change to $A/\pi A$. 
    Suppose that we have an $\cl{H}$-subtorsor of $P_{A/\pi A}$. 
    Since $A$ is $p$-torsion free, there are no nontrivial automorphisms of $\cl{G}'$-subtorsors of $P$, so we may $p$-completely \'{e}tale locally localize $A$. 
    Thus, we may assume that $P$ is trivial and the $\cl{H}$-subtorsor of $P_{A/\pi A}$ is generated by the reduction of an element $x\in P(A)$. 
    It is enough to show that the $\cl{G}'$-subtorsor generated by $x$ is independent of the choice of $x$. 
    Let $x'\in P(A)$ be another choice. 
    The reduction of $x^{-1}x'$ to $A/\pi A$ lies in $\cl{H}(A/\pi A)$. Now, the map $\cl{G}'(A)\to \cl{H}(A/\pi A)$ is surjective because $\cl{G}(A)\to \cl{G}(A/\pi A)$ is surjective and $\cl{G}'$ is the dilatation (see \cite[Section A.5]{KP23}) of $\cl{G}$ along $\cl{H}$ by \cite[Corollary 2.10.10]{KP23} and \cite[Theorem 8.4.19 (3)]{KP23}. Thus, we may assume that the reduction of $x^{-1}x'$ is trivial. Then, we have $x^{-1}x'\in \cl{G}'(A)$ by \Cref{lem:kerincl}. 
\end{proof}

The following lemma is the basis of the representability of the moduli of subtorsors. 

\begin{lem}\label{lem:quotient}
    Let $k$ be a field. Let $G$ be a connected linear algebraic group over $k$ and let $H$ be a parabolic subgroup of $G$. Let $R$ be a $k$-algebra and let $P$ be a $G$-torsor over $R$. The quotient fpqc sheaf $P/H$ is representable by a projective $R$-scheme. 
\end{lem}
\begin{proof}
    Since the unipotent radical of $G$ is contained in $H$, we may assume that $G$ is reductive. A regular dominant cocharacter $\lambda$ of $H$ defines a $G$-equivariant ample line bundle $O(\lambda)$ on $G/H$. 
    Take an \'{e}tale $R$-algebra $R'$ so that $P_{R'}\cong G_{R'}$. 
    It gives a $1$-cocycle valued in $G$ and it induces an \'{e}tale descent datum $((G/H)_{R'},O(\lambda))$ corresponding to $P/H$. 
    Since $O(\lambda)$ is ample, it is effective (see \cite[Theorem 7]{BLR90}). 
\end{proof}

\begin{defi}
    Let $\cl{P}$ be a local $\cl{G}$-shtuka on a $v$-stack $X$. A $\cl{G}'$-level structure on $\cl{P}$ is a commutative diagram
    \begin{center}
        \begin{tikzcd}
            X \ar[r, "\cl{P}'"] \ar[rd, "\cl{P}"'] & \Sht_{\cl{G}'} \ar[d]\\
            & \Sht_\cl{G}. 
        \end{tikzcd}
    \end{center}
    In other words, it consists of a local $\cl{G}'$-shtuka $\cl{P}'$ with an isomorphism $\cl{P}'\times^{\cl{G'}} \cl{G} \cong \cl{P}$. 
\end{defi}

We have the following commutative diagram by \Cref{prop:mapLW}. 

\begin{center}
    \begin{tikzcd}
        \Sht_{\cl{G}'} \ar[r] \ar[d] & \lbrack \ast/(L_W^+\cl{G'})^\diamondsuit \rbrack \ar[d]\\
        \Sht_\cl{G} \ar[r] & \lbrack \ast/(L_W^+\cl{G})^\diamondsuit \rbrack 
    \end{tikzcd}
\end{center}

\begin{lem} \label{lem:fibSht}
    The following holds in the above diagram. 
    \begin{enumerate}
        \item $[\ast/(L_W^+\cl{G}')^\diamondsuit] \to [\ast/(L_W^+\cl{G})^\diamondsuit]$ is separated. 
        \item $\Sht_{\cl{G'},\eta}$ is finite \'{e}tale over $\Sht_{\cl{G},\eta}$ and locally isomorphic to the set $\cl{G}(O_F)/\cl{G}'(O_F)$. 
        \item $\Sht_{\cl{G}',s}$ is isomorphic to $\Sht_{\cl{G},s}\times_{[\ast/(L_W^+\cl{G})^\diamondsuit]} [\ast/(L_W^+\cl{G}')^\diamondsuit]$. 
    \end{enumerate}
\end{lem}
\begin{proof}
    By \Cref{lem:G'subtors} and \Cref{lem:quotient}, the fiber $\ast \times_{[\ast/(L_W^+\cl{G})^\diamondsuit]} [\ast/(L_W^+\cl{G}')^\diamondsuit]$ is represented by the projective $k$-scheme $\cl{G}_k/\cl{H}$. The claim (1) follows from the $v$-descent of separatedness in \cite[Proposition 10.11]{Sch17}. 

    Let $S$ be a perfectoid space over $F$ and let $\cl{P}$ be a local $\cl{G}$-shtuka over $S_{/O_F}$. By \cite[Proposition 22.6.1]{SW20}, there is a pro-\'{e}tale $\underline{\cl{G}(O_F)}$-torsor $\bb{P}$ over $S$ corresponding to $\cl{P}$. The fiber product $\Sht_{\cl{G}'}\times_{\Sht_{\cl{G}},\cl{P}} S$ is represented by the finite \'{e}tale covering $\bb{P}/\cl{G'}(O_F)$ of $S$. This proves (2). 
    
    On the other hand, let $S=\Spa(R,R^+)$ be an affinoid perfectoid space over $k$ and let $\cl{P}$ be a local $\cl{G}$-shtuka over $S$. For every $\cl{G}'$-subtorsor $\cl{P}'$ of $\cl{P}$ on $\cl{Y}_S$, the Frobenius of $\cl{P}$ gives the structure of a local $\cl{G}'$-shtuka on $\cl{P}'$. Since the map $\cl{P}'\to \cl{P}$ is an isomorphism after inverting $p$, we see by applying the Beuville-Laszlo lemma (in the form of \cite[Lemma 5.2.9]{SW20}) to $B_{S,[0,n]}$ that $\cl{P}'$ is determined by the restriction to $W_{O_F}(R)$. Thus, we have (3). 
\end{proof}

\begin{prop}\label{prop:Shtclsd}
    The $v$-stack $\Sht_\cl{G'}$ is a closed substack of $\Sht_\cl{G}\times_{[\ast/(L_W^+\cl{G})^\diamondsuit]} [\ast/(L_W^+\cl{G}')^\diamondsuit]$. 
\end{prop}
\begin{proof}
    Let $S=\Spa(R,R^+)$ be an affinoid perfectoid space over $k$ with an untilt $S^\sharp = \Spa(R^\sharp,R^{+\sharp})$ over $O_F$. Let $\cl{P}$ be a local $\cl{G}$-shtuka over $S^\sharp_{/O_F}$. A $\cl{G}'$-level structure $\cl{P}'$ on $\cl{P}$ is determined by the corresponding $\cl{G}'$-subtorsor on $\cl{Y}_{S,[0,n]}$. By the same argument as in \Cref{lem:fibSht}, we see that $\cl{P}'$ is determined by the restriction to $W_{O_F}(R)$ and $\Sht_\cl{G'}$ is a substack of $\Sht_\cl{G}\times_{[\ast/(L_W^+\cl{G})^\diamondsuit]} [\ast/(L_W^+\cl{G}')^\diamondsuit]$. We show that it is a closed substack. 

    A section of $\Sht_\cl{G}\times_{[\ast/(L_W^+\cl{G})^\diamondsuit]} [\ast/(L_W^+\cl{G}')^\diamondsuit]$ over $S^\sharp$ corresponds to a $\cl{G}'$-subtorsor $\cl{P}'$ of $\cl{P}$ by the Beuville-Laszlo lemma. We determine when its pullback along a map $T = \Spa(A,A^+)\to S$ factors through $\Sht_{\cl{G}'}$. Since $\varphi_{\cl{P}}$ is meromorphic along $S^\sharp$, its pullback to $\cl{Y}_{T,[0,n]}$ is given by an isomorphism $(\varphi^*\cl{P})\vert_{B_{T,[0,n]}[1/\xi_{R^\sharp}]} \cong \cl{P}\vert_{B_{T,[0,n]}[1/\xi_{R^\sharp}]}$. The $\cl{G}'$-subtorsor $\cl{P}'$ defines a local $\cl{G}'$-shtuka over $T$ if and only if $\varphi_{\cl{P}}\vert_{B_{T,[0,n]}[1/\xi_{R^\sharp}]}$ induces an isomorphism $(\varphi^*\cl{P'})\vert_{B_{T,[0,n]}[1/\xi_{R^\sharp}]} \cong \cl{P'}\vert_{B_{T,[0,n]}[1/\xi_{R^\sharp}]}$. By applying the Beauville-Laszlo lemma to $B_{T,[0,n]}[1/\xi_{R^\sharp}]$, it can be written as a condition that two $\cl{G}'$-subtorsors of $\cl{P}$ are equal over the $p$-adic completion of $W_{O_F}(A)[1/\xi_{R^\sharp}]$. Now, we have $W_{O_F}(A)[1/\xi_{R^\sharp}]/(\pi) \cong A[1/\xi_{R^\sharp,0}]$. Let $P$ be a $\cl{G}$-torsor over $R$ corresponding to $\cl{P}\vert_S$. By \Cref{lem:G'subtors}, $\cl{G}'$-subtorsors of $\cl{P}$ over the $p$-adic completion of $W_{O_F}(A)[1/\xi_{R^\sharp}]$ correspond to $\cl{H}$-subtorsors of $P\otimes_R A[1/\xi_{R^\sharp,0}]$. By \Cref{lem:quotient}, they correspond to sections of $P/\cl{H}$ over $A[1/\xi_{R^\sharp,0}]$. Now, through the isomorphism $(\varphi^*\cl{P})\vert_{B_{S,[0,n]}[1/\xi_{R^\sharp}]} \cong \cl{P}\vert_{B_{S,[0,n]}[1/\xi_{R^\sharp}]}$, each of $\cl{P}'$ and $\varphi^*\cl{P}'$ corresponds to a section of $P/\cl{H}$ over $R[1/\xi_{R^\sharp,0}]$. Since $P/\cl{H}$ is separated, the equalizer locus of the two sections is the zero locus of an ideal $I\subset R[1/\xi_{R^\sharp,0}]$. Let $J\subset R$ be an ideal such that $I=J[1/\xi_{R^\sharp,0}]$. Since $A$ is perfect, the composition $R\to A \to A[1/\xi_{R^\sharp,0}]$ factors through $R[1/\xi_{R^\sharp,0}]/I$ if and only if $\xi_{R^\sharp,0}\cdot J \subset \Ker(R\to A)$. Thus, the locus where $\cl{P}'$ defines a local $\cl{G}'$-shtuka is a Zariski closed subset of $S$ given by the ideal $\xi_{R^\sharp,0}\cdot J$. 
\end{proof}

We can check that the condition in \Cref{prop:finetrep} holds in the following situation. 

\begin{prop} \label{prop:repatinfty}
    Let $R$ be a perfectoid ring and let $\cl{P}$ be a local $\cl{G}$-shtuka over $\Spd(R)_{/O_F}$. There is a perfectly projective $R^\flat$-scheme $Y$ and a finite \'{e}tale $R^\flat[1/\xi_{R,0}]$-scheme $Z$ with a morphism $Z \to  Y$ such that we have the following commutative diagram
    \begin{center}
        \begin{tikzcd}
            \Sht_{\cl{G}',\eta} \times_{\Sht_{\cl{G},\eta}, \cl{P}_\eta} \Spd(R)^\an \ar[r] \ar[d,"\cong"] & \lbrack \ast/(L_W^+\cl{G}')^\diamondsuit \rbrack \times_{\lbrack \ast/(L_W^+\cl{G})^\diamondsuit \rbrack, \cl{P}} \Spd(R) \ar[d,"\cong"] \\
            Z^{\diamondsuit}\times_{\Spec(R^\flat[1/\xi_{R,0}])^\diamondsuit} \Spd(R^\flat)^\an \ar[r] & Y^\diamond\times_{\Spec(R^\flat)^\diamond} \Spd(R^\flat). 
        \end{tikzcd}
    \end{center}
    that is compatible with tilting equivalence $\Spd(R) \cong \Spd(R^\flat)$. 
\end{prop}
\begin{proof}
    By \Cref{prop:GBKF}, $\cl{P}$ is given by a $\cl{G}$-BKF module $P$ over $R$. Let $Y=(P\otimes_{W_{O_F}(R^\flat)} R^\flat)/\cl{H}$. By \Cref{lem:quotient}, $Y$ is projective over $R^\flat$. For each perfect $R^\flat$-algebra $A$, $Y(A)$ classifies the set of $\cl{G}'$-subtorsors of $P\otimes_{W_{O_F}(R^\flat)} W_{O_F}(A)$ by \Cref{lem:G'subtors}. Thus, we see from \Cref{lem:perfprop} that $Y^\diamond \times_{\Spec(R^\flat)^\diamond} \Spd(R^\flat) \cong Y^\diamondsuit \times_{\Spec(R^\flat)^\diamond} \Spd(R^\flat)$ is isomorphic to $\lbrack \ast/(L_W^+\cl{G}')^\diamondsuit \rbrack \times_{\lbrack \ast/(L_W^+\cl{G})^\diamondsuit \rbrack, \cl{P}} \Spd(R)$. 

    By \Cref{lem:fibSht}, $\Sht_{\cl{G}',\eta} \times_{\Sht_{\cl{G},\eta}, \cl{P}_\eta} \Spd(R)^\an$ is finite \'{e}tale over $\Spd(R)^\an$. Let $R^{\flat, +} \subset R^\flat[1/\xi_{R,0}]$ be the integral closure of $R^\flat$ inside $R^\flat[1/\xi_{R,0}]$. By \cite[Lemma 15.6]{Sch17}, there is a unique finite \'{e}tale homomorphism $(R^\flat[1/\xi_{R,0}], R^{\flat, +}) \to (S,S^+)$ such that $\Spd(S,S^+) \cong \Sht_{\cl{G}',\eta} \times_{\Sht_{\cl{G},\eta}, \cl{P}_\eta} \Spd(R)^\an$ under tilting equivalence. Let $Z=\Spec(S)$. By construction, we have a $\cl{G}'$-level structure $\cl{P}'$ of $\cl{P}$ over $\Spd(S,S^+)$. By restricting to $W_{O_F}(S)$, $\cl{P}'$ gives a $\cl{G}'$-subtorsor of $P\otimes_{W_{O_F}(R^\flat)} W_{O_F}(S)$. By the above argument, it corresponds to a morphism $Z \to Y$ and this morphism makes the given diagram commute by construction. 
\end{proof}

Now, we can define the flat moduli space of $\cl{G}'$-level structures on local $\cl{G}$-shtukas over reduced excellent $p$-adic formal schemes admitting good covers. 
 
\begin{thm} \label{thm:replevel}
    Let $\mfr{X}$ be a reduced excellent $p$-adic formal $O_F$-scheme admitting a good cover. Let $\cl{P}$ be a local $\cl{G}$-shtuka over $\mfr{X}^\diamond_{/O_F}$. The $v$-closure of the generic fiber in $\Sht_{\cl{G}'}\times_{\Sht_{\cl{G}},\cl{P}} \mfr{X}^\diamond$ is representable by a unique reduced excellent distinguished $p$-adic formal scheme $\mfr{Y}$ admitting a maximal good cover that is proper over $\mfr{X}$ with $\mfr{Y}_\eta$ finite \'{e}tale over $\mfr{X}_\eta$. 
\end{thm}

We say that $\mfr{Y}$ is the flat moduli space of $\cl{G}'$-level structures on $\cl{P}$.

\begin{proof}
    Let $Y=[\ast/(L_W^+\cl{G}')^\diamondsuit] \times_{[\ast/(L_W^+\cl{G})^\diamondsuit], \cl{P}} \mfr{X}^\diamond$ and $Z=\Sht_{\cl{G}',\eta} \times_{\Sht_{\cl{G},\eta}, \cl{P}_\eta} \mfr{X}^\diamond_\eta$. By \Cref{lem:fibSht}, $Y$ is separated over $\mfr{X}^\diamond$ and $Z$ is finite \'{e}tale over $\mfr{X}_\eta^\diamond$. By \Cref{prop:repatinfty}, we may apply \Cref{prop:finetrep} to $Z \hookrightarrow Y$ to see the representability of the $v$-closure of $Z$ in $Y$. By \Cref{prop:Shtclsd}, the $v$-closure of $Z$ in $Y$ lies in $\Sht_{\cl{G}'}\times_{\Sht_{\cl{G}},\cl{P}} \mfr{X}^\diamond$, so the claim follows. 
\end{proof}

\subsection{Integral models of local Shimura varieties and formal completion}

In this section, we review the basic properties of local Shimura varieties and their $v$-sheaf theoretic integral models introduced in \cite[Lecture 24, 25]{SW20}. For simplicity, we assume that $G$ is unramified. 

Fix an algebraic closure $\ov{k}$ of $k$ and let $O_{\breve{F}}=W_{O_F}(\ov{k})$ and $\breve{F}=O_{\breve{F}}[\tfrac{1}{p}]$. Let $B(G)$ denote the set of $\sigma$-conjugacy classes of $G(\breve{F})$. For each cocharacter $\mu \colon \bb{G}_m \to G_{\breve{F}}$, there is a subset $B(G, \mu) \subset B(G)$ of $\mu$-bounded $\sigma$-conjugacy classes. For each minuscule $\mu$ and $[b] \in B(G, \mu)$, the triple $(G, b, \mu)$ is called a local Shimura datum. Note that we adopt the opposite sign convention for $\mu$ compared to \cite{SW20}. 

For each compact open subgroup $K \subset G(F)$, a local Shimura variety $\cl{M}_{G, b, \mu, K}$ is defined as a rigid analytic variety over $\breve{F}$. When $K = \cl{G}(O_F)$, it is also denoted by $\cl{M}_{\cl{G}, b, \mu}$ and it parametrizes a pair $(\cl{P}, \iota)$ where $\cl{P}$ is a local $\cl{G}$-shtuka bounded by $-\mu$ and 
$
    \iota \colon \cl{P}\vert_{\cl{Y}_{[r, \infty)}} \cong \cl{E}^b\vert_{\cl{Y}_{[r, \infty)}}
$
for sufficiently large $r$ is a quasi-isogeny to $\cl{E}^b$. This moduli extends to perfectoid spaces over $O_{\breve{F}}$ by replacing the boundedness condition as follows. 

\begin{defi}\textup{(\cite[Conjecture 21.4.1]{SW20}, \cite[Theorem 1.2]{AGLR22})}
    Let $\cl{M}_{\cl{G}, -\mu}^\loc \subset \Gr_{\cl{G}}$ be the $v$-closure of the flag space $\Flag^\diamond_{G, -\mu}$ inside the Beilinson-Drinfeld Grassmannian $\Gr_{\cl{G}}$. Let $M^\loc_{\cl{G}, -\mu}$ be the unique flat projective normal $O_F$-scheme representing $\cl{M}^\loc_{\cl{G}, -\mu}$. 
\end{defi}

\begin{defi}\textup{(\cite[Definition 25.1.1]{SW20})} \label{defi:v_sheaf_integral_model}
    The $v$-sheaf theoretic integral model $\cl{M}_{\cl{G},b,\mu}^\ints$ of $\cl{M}_{\cl{G}, b, \mu}$ is a $v$-sheaf sending a perfectoid space $S$ over $O_{\breve{F}}$ to the set of pairs $(\cl{P}, \iota)$ where 
    \begin{enumerate}
        \item $\cl{P}$ is a local $\cl{G}$-shtuka whose Frobenius 
        $\varphi_{\cl{P}}$ is bounded by $\cl{M}_{\cl{G}, -\mu}^\loc$, and
        \item $\iota$ is a quasi-isogeny from $\cl{P}$ to $\cl{E}^b$. 
    \end{enumerate}
\end{defi}

\begin{rmk}
    In most cases of abelian type, it is known to be representable by a formal scheme over $O_{\breve{F}}$ (see \cite{PR22a}), but much less is known in general, especially in exceptional cases. In a recent ongoing work of Lee and Madapusi \cite{Si24}, the representability at hyperspecial levels is announced in full generality. 
\end{rmk}

In general, $\cl{M}_{\cl{G}, b, \mu}^{\ints}$ is known to be a prekimberlite and the underlying space $(\cl{M}_{\cl{G}, b, \mu}^{\ints})_\red$ is equal to the affine Deligne-Lusztig variety $X_\mu(b)$ (see \cite[2.61, 2.63]{Gle22a}). Then, for each closed point $x \in X_\mu(b)$, we can define a formal completion $(\cl{M}^\ints_{\cl{G},b,\mu})^\wedge_{/x}$ by \cite[Definition 4.18]{Gle24}: it is a $v$-subsheaf of $\cl{M}^\ints_{\cl{G},b,\mu}$ consisting of geometric points whose rank $1$ generalizations map to $x$ under the specialization map. When $\cl{M}^\ints_{\cl{G}, b, \mu}$ is representable by a formal scheme $\cl{S}$, $(\cl{M}^\ints_{\cl{G},b,\mu})^\wedge_{/x}$ is represented by the formal completion of $\cl{S}$ at $x$. 

In particular, the representability of $(\cl{M}^\ints_{\cl{G},b,\mu})^\wedge_{/x}$ is milder than the whole representability, and it is in fact proved when $\cl{G}$ is reductive by \cite{Bar22} and \cite{Ito25a}. For later application, we briefly review the description of \cite{Ito25a} in terms of prismatic $(\cl{G}, \mu)$-display.

\subsubsection{Prismatic displays and universal deformations} \label{sssec:prism_display}

From now on, we assume that $\cl{G}$ is reductive. Take an integral lift $\mu \colon \bb{G}_m \to \cl{G}_{O_{\breve{F}}}$. In \cite[Section 5.2]{Ito25b}, prismatic $(\cl{G},\mu)$-displays are introduced for $O_F$-prisms $(A, I)$ over $O_{\breve{F}}$. We refer to \cite[Section 2]{Ito25b} for $O_F$-prisms. An $O_F$-prism $(A, I)$ is over $O_{\breve{F}}$ if it is equipped with a map $O_{\breve{F}} \to A / I$. Since $A$ is $(p, I)$-complete, 
such a map uniquely lifts to $O_{\breve{F}} \to A$. 

Let $(A, I)$ be an $O_F$-prism over $O_{\breve{F}}$ with $I = (d)$. The display group attached to $(A, I)$ is 
\[
    \cl{G}_\mu(A, I) = \{ g \in G(A) \mid \mu(d) g \mu(d)^{-1} \in \cl{G}(A) \}
\]
and it admits a twisted right action 
\[
    \cl{G}(A) \times \cl{G}_\mu(A, I) \to \cl{G}(A), \quad 
    (X, g) \mapsto g^{-1} X \varphi(\mu(d) g \mu(d)^{-1})
\]
(see \cite[(5.2)]{Ito25b}). In order to get a description independent of $d$, we need to replace $\cl{G}(A)$ by
$
    \cl{G}(A)_I = \varprojlim_{d} \cl{G}(A)
$
under suitable transition maps. For a general $O_F$-prism $(A, I)$ over $O_{\breve{F}}$, we have \'{e}tale sheaves
\[
    \cl{G}_{\mu, (A, I)} \colon (A, I)_{\Prism, \acute{e}t} \to \Grp, \quad 
    \cl{G}_{\Prism, (A, I)} \colon (A, I)_{\Prism, \acute{e}t} \to \Set
\]
with a right action $\cl{G}_{\Prism, (A, I)} \circlearrowleft \cl{G}_{\mu, (A, I)}$ by patching $\cl{G}(A)_I \circlearrowleft \cl{G}_\mu(A, I)$ \'{e}tale locally. 

\begin{defi}\textup{(\cite[Definition 5.2.1]{Ito25b})}
    A prismatic $(\cl{G},\mu)$-display over an $O_F$-prism $(A,I)$ over $O_{\breve{F}}$ is a pair $(\cl{Q}, \alpha_{\cl{Q}})$ where $\cl{Q}$ is a $\cl{G}_{\mu, (A, I)}$-torsor over $(A, I)_{\acute{e}t}^{\op}$ and $\alpha \colon \cl{Q} \to \cl{G}_{\Prism, (A, I)}$ is a $\cl{G}_{\mu, (A, I)}$-equivariant map. 
\end{defi}

By pushing forward along the natural inclusion $\cl{G}_{\mu}(A, I) \subset \cl{G}(A)$, a $\cl{G}$-torsor $\cl{Q}_{(A, I)}$ over $A$ is attached to a prismatic $(\cl{G},\mu)$-display $\cl{Q}$ over $(A, I)$. Its pullback to $A / I$ is denoted by $\cl{Q}_{A / I}$. Similar constructions work over the absolute $O_F$-prismatic site $(R)_{\Prism, O_F}$ for an $O_{\breve{F}}$-algebra $R$. Here, $(R)_{\Prism, O_F}$ consists of $O_F$-prisms $(B, J)$ equipped with a structure map $g_{(B, J)} \colon R \to B / J$ endowed with the \'{e}tale topology. 

\begin{rmk} \label{rmk:display_vs_BKF}
    By \cite[Proposition 5.3.8]{Ito25b}, a prismatic $(\cl{G}, \mu)$-display $\cl{Q}$ over $(A, I)$ uniquely corresponds to a $\cl{G}$-BKF module $\cl{Q}_{\BK}$ of type $\mu$ over $(A, I)$. In particular, if $(A, I)$ is perfect, $\cl{Q}$ is equivalent to a local $\cl{G}$-shtuka over $\Spd(A / I)_{/O_F}$. By \cite[Proposition 5.3.4]{Ito25b}, there is a natural isomorphism $\varphi^* \cl{Q}_{\BK} \cong \cl{Q}_{(A, I)}$ of $\cl{G}$-torsors over $A$. 
\end{rmk}

\begin{exa} \textup{(\cite[Example 5.3.7]{Ito25b})}
    \label{rmk:banal_display_BKF}
    Fix a generator $d \in I$. For each $X \in \cl{G}(A)$, $\cl{Q}_X$ denotes the prismatic $(\cl{G}, \mu)$-display associated to the trivial $\cl{G}_{\mu, (A, I)}$-torsor and the map $\cl{G}_{\mu, (A, I)} \to \cl{G}_{\Prism, (A, I)}$ sending $1$ to $X$. Then, $(\cl{Q}_X)_{\BK}$ is associated to the map 
    \[  
        (\alpha_X)_\BK \colon \varphi^*\cl{G}_A[1/I] \cong \cl{G}_A[1/I] ,\quad 
        1 \mapsto \mu(d)X. 
    \]
    Moreover, the natural isomorphism $\alpha'\colon \varphi^*\cl{G}_A \cong \cl{G}_A$ associated to $\cl{Q}_X$ sends $1$ to $X$ by the proof of \cite[Proposition 5.3.4]{Ito25b}. In particular, $(\alpha_X)_\BK \circ \alpha'^{-1}$ sends $1$ to $\mu(d)$. 
\end{exa}

Fix a closed point $x \in X_\mu(b)(\ov{k})$. The local $\cl{G}$-shtuka at $x$ corresponds to a prismatic $(\cl{G}, \mu)$-display $\cl{Q}_x$ over $(O_{\breve{F}}, (\pi))$. We will introduce the universal deformation ring $R_{\cl{G}, \mu}$ of $\cl{Q}_x$. For simplicity, fix a maximal torus $T \subset G_{\breve{F}}$ such that $\mu \in X_*(T)$ and let $\Phi_{\mu < 0}$ denote the set of roots $\alpha$ such that $\langle \mu, \alpha \rangle = -1$. Then, 
\[
    (\mfr{S},\cl{E}) = (O_{\breve{F}}\llbracket t, u_\alpha \mid \alpha \in \Phi_{\mu<0} \rrbracket, \pi-t) ,\quad R_{\cl{G}, \mu} = \mfr{S}/\cl{E}
\]
is an (oriented) Breuil-Kisin $O_F$-prism associated with $R_{\cl{G},\mu}$. Let $\mfr{m}$ be the maximal ideal of $R_{\cl{G}, \mu}$ and let $R_{\cl{G}, \mu}^\wedge$ be the complete $\mfr{m}$-adic ring whose underlying ring is $R_{\cl{G}, \mu}$.

\begin{thm}\textup{\cite[Theorem 4.4.2, Theorem 5.3.5]{Ito25b}}
    There exist a prismatic $(\cl{G}, \mu)$-display $\mfr{Q}^\univ$ over $(R_{\cl{G}, \mu})_{\Prism, O_F}$ and an isomorphism of $v$-sheaves 
    \[
        (\cl{M}^\ints_{\cl{G},b,\mu})^\wedge_{/x} \cong \Spd(R_{\cl{G}, \mu}^\wedge)
    \]
    such that the universal local $\cl{G}$-shtuka over $(\cl{M}^\ints_{\cl{G},b,\mu})^\wedge_{/x}$ is sent to the one induced by $\mfr{Q}^\univ$. 
\end{thm}

Here, for each perfectoid $R_{\cl{G}, \mu}$-algebra $R$, $\mfr{Q}^\univ\vert_{(R)_{\Prism, O_F}}$ corresponds to a local $\cl{G}$-shtuka over $\Spd(R)_{/O_F}$. In this way, $\mfr{Q}^\univ$ induces a local $\cl{G}$-shtuka $\cl{P}^\univ$ over $\Spd(R_{\cl{G}, \mu})_{/O_F}$, not only over $\Spd(R_{\cl{G}, \mu}^\wedge)_{/O_F}$.

\subsection{Local representability under hyperspecial levels}
\label{ssec:locint}

In this section, we prove the local representability of integral models of local Shimura varieties 
under hyperspecial levels. 

Let $\cl{G}'$ be a parahoric group scheme with a morphism $\cl{G}' \to \cl{G}$. We consider the integral model $\cl{M}_{\cl{G}',b,\mu}^\ints$ at the level $\cl{G}'(O_F)$. By \Cref{prop:exadmit} and \Cref{lem:compperfcov}, $R_{\cl{G}, \mu}$ admits a very good cover. Thus, we may apply \Cref{thm:replevel} to $\cl{P}^\univ$. Let $\mfr{Y}$ be the flat moduli space of $\cl{G}'$-level structures on $\cl{P}^\univ$ and let $\mfr{Y}^\wedge$ be the $\mfr{m}$-adic completion of $\mfr{Y}$. 

\begin{thm}\label{thm:localrep}
    The fiber product $\cl{M}^\mrm{int}_{\cl{G}',b,\mu}\times_{\cl{M}^\mrm{int}_{\cl{G},b,\mu}} (\cl{M}^\mrm{int}_{\cl{G},b,\mu})^\wedge_{/x}$ is represented by $\mfr{Y}^{\wedge}$. In particular, for every $y \in \cl{M}_{\cl{G}',b,\mu}^\mrm{int}(\ov{k})$, the formal completion $(\cl{M}_{\cl{G}',b,\mu}^\mrm{int})^\wedge_{/y}$ is representable by a unibranch $p$-torsion free complete local Noetherian ring. 
\end{thm}

\begin{proof}
    Let $Z=\Sht_{\cl{G}'}\times_{\Sht_{\cl{G}},\cl{P}^\univ} \Spd(R_{\cl{G}, \mu})$ and let $Z^\bdd \subset Z$ be the closed subsheaf parametrizing $M_{\cl{G}',-\mu}^\loc$-bounded $\cl{G}'$-level structures. Since $\mfr{Y}$ represents the $v$-closure of $Z_\eta$ in $Z$, there is an inclusion $\mfr{Y}^\diamond \hookrightarrow Z^\bdd$. 

    First, we show that $Z^\bdd_s$ is represented by a perfect scheme. By \Cref{lem:fibSht} and \Cref{prop:repatinfty}, $Z_s$ is represented by a perfect scheme $Z_s^\red$. The $\cl{G}'$-level structure over $Z_s$ is represented by an $F$-crystal with $\cl{G}'$-structure $P'_s$ on $Z_s^\red$. By \cite[Lemma 1.22]{Zhu17}, the locus $Z_s^{\bdd,\red} \subset Z_s^\red$ where the Frobenius of $P'_s$ is bounded by $M_{\cl{G}',-\mu}^\loc$ is closed. For each geometric point $x\colon \Spa(C,C^+) \to Z_s$, the $\cl{G}'$-level structure is bounded by $M_{\cl{G}',-\mu}^\loc$ at $x$ if and only if the corresponding morphism $\Spec(C) \to Z_s$ factors through $Z_s^{\bdd,\red}$. Thus, $Z_s^\bdd$ is represented by $Z_s^{\bdd, \red}$. Since $\mfr{Y}$ is proper over $\Spf(R)$, $\mfr{Y}_s^\perf \hookrightarrow Z_s^\bdd$ is a closed immersion. 

    By \Cref{defi:v_sheaf_integral_model}, $\cl{M}^\mrm{int}_{\cl{G}',b,\mu}\times_{\cl{M}^\mrm{int}_{\cl{G},b,\mu}}(\cl{M}^\mrm{int}_{\cl{G},b,\mu})^\wedge_{/x}$ is isomorphic to $Z^\bdd \times_{\Spd(R_{\cl{G}, \mu})} \Spd(R_{\cl{G}, \mu}^\wedge)$. For every $y \in \cl{M}^\mrm{int}_{\cl{G}',b,\mu}(\ov{k})$, $(\cl{M}^\mrm{int}_{\cl{G'},b,\mu})^\wedge_{/y}$ is topologically flat by \cite[Proposition 3.4.1 (1)]{PR24}, so $(\cl{M}^\mrm{int}_{\cl{G'},b,\mu})^\wedge_{/y}$ is contained in $(\mfr{Y}^\wedge)^\diamond$ as we have $Z^\bdd_\eta=\mfr{Y}^\diamond_\eta$. Thus, we may apply \Cref{lem:tabatclsd} to the $\mfr{m}$-adic completion of $\mfr{Y}_s^\perf \hookrightarrow Z_s^\bdd$ and we see that $\mfr{Y}_s^\perf \hookrightarrow Z_s^\bdd$ is an isomorphism after the $\mfr{m}$-adic completion. Thus, $(\mfr{Y}^\wedge)^\diamond \hookrightarrow \cl{M}^\mrm{int}_{\cl{G}',b,\mu}\times_{\cl{M}^\mrm{int}_{\cl{G},b,\mu}}(\cl{M}^\mrm{int}_{\cl{G},b,\mu})^\wedge_{/x}$ is an isomorphism on each fiber over $\Spd(\bb{Z}_p)$, so it is an isomorphism by \cite[Lemma 12.11]{Sch17}. 

    Next, we prove the second claim. Let $y\in \cl{M}_{\cl{G}',b,\mu}^\ints(\ov{k})$ and take $x\in \cl{M}_{\cl{G},b,\mu}^\ints(\ov{k})$ to be the image of $y$. Let $\Spf(S) \subset \mfr{Y}$ be an affine open neighborhood of $y$ and let $S^+ \subset S[\tfrac{1}{p}]$ be the integral closure of $S$. Since $R_{\cl{G}, \mu}$ is normal and $(S[\tfrac{1}{p}],S^+)$ is \'{e}tale over $(R_{\cl{G}, \mu}[\tfrac{1}{p}],R_{\cl{G}, \mu})$, $S^+$ is normal by \Cref{cor:normaletloc}. Since $S$ is $p$-torsion free, reduced and excellent, $S\subset S^+$ is a finite extension. Let $\mfr{n}\subset S$ be the maximal ideal corresponding to $y$ and let $S^\wedge$ (resp.\ $S^{+,\wedge}$) be the $\mfr{n}$-adic completion of $S$ (resp.\ $S^+$). By \cite[Theorem 79]{Mat80}, $S^{+,\wedge}$ is normal. Since $S[\tfrac{1}{p}]=S^+[\tfrac{1}{p}]$ and $S\subset S^+$ is a finite extension, $S^\wedge$ is $p$-torsion free and $S^\wedge[\tfrac{1}{p}]=S^{+,\wedge}[\tfrac{1}{p}]$. By construction, $(\cl{M}^\ints_{\cl{G}',b,\mu})^\wedge_{/y}$ is represented by $S^\wedge$, so it is enough to show that $S^\wedge$ is unibranch. 

    By \cite[Proposition 18.6.12]{EGA4-4}, it is enough to show that $S^\wedge$ has a unique minimal prime. Since $S^\wedge$ is $p$-torsion free, the set of minimal primes of $S^\wedge$ is bijective to that of $S^\wedge[\tfrac{1}{p}] \cong S^{+,\wedge}[\tfrac{1}{p}]$. Let $S^{+,\wedge} = \prod_{i\in I} S_i^+$ be the decomposition into normal domains. Since $S^{+,\wedge}$ is finite over $S^\wedge$, $S_i^+$ is a $p$-torsion free complete local Noetherian ring. Now, we have $\Spa(S^\wedge)_\eta \cong \prod_{i\in I} \Spa(S_i^+)_\eta$. Since $\Spa(S_i^+)_\eta$ is nonempty for every $i\in I$, it follows from \cite[Proposition 3.4.1(2)]{PR24} that $I$ is a singleton. Thus, $S^{+,\wedge}$ is a normal domain and the claim follows. 
\end{proof}

The following lemma is needed in the above proof. We say that a perfect formal scheme over $\ov{\bb{F}}_p$ is formally of finite type if it is Zariski locally isomorphic to $\Spf(A)$ with $A=(A_0^\perf)^\wedge$ for some formally finite type $\ov{\bb{F}}_p$-algebra $A_0$. Note that such $A_0$ is an adic $\ov{\bb{F}}_p$-algebra obtained as the completion of a finite type $\ov{\bb{F}}_p$-algebra. 

\begin{lem}\label{lem:tabatclsd}
    Let $\mfr{Y}$ be a perfect formal scheme perfectly formally of finite type over $\ov{\bb{F}}_p$ and let $\mfr{Z} \subset \mfr{Y}$ be a closed perfect formal subscheme. If $(\mfr{Y}^\diamond)^\wedge_{/y} \subset \mfr{Z}^\diamond$ as subsheaves of $\mfr{Y}^\diamond$ for every $y\in \mfr{Y}(\ov{\bb{F}}_p)$, we have $\mfr{Y}=\mfr{Z}$. 
\end{lem}
    
\begin{proof}
    We may assume that $\mfr{Y}$ is affine. Let $\mfr{Y}=\Spf(A)$ and $\mfr{Z}=\Spf (A/I)$. Suppose that $I\neq 0$ and let $a\in I$ be a nonzero element. Since $a$ vanishes at every closed point of $\mfr{Y}$, we see that $a$ is topologically nilpotent. Let $f_1,\ldots,f_n \in A$ be generators of an ideal of definition of $A$. By \Cref{lem:perfdrep}, $\Spa(A/aA) \subset \Spa(A)$ is a proper closed subspace and its complement is the union of $R(\tfrac{f_1^k,\ldots,f_n^k}{a})$ for $k\geq 0$. Thus, there is some $k\geq 0$ such that $R(\tfrac{f_1^k,\ldots,f_n^k}{a}) \neq \phi$. The closure of $R(\tfrac{f_1^k,\ldots,f_n^k}{a})$ in $\Spa(A)^\an$ is contained in $R(\tfrac{f_1^{k+1},\ldots,f_n^{k+1}}{a})$, so in particular, it is away from $\Spd(A/aA)$. However, since the specialization map is closed, the image of the closure under the specialization map of $\Spa(A)$ is a closed subset of $\Spec(A_\red)$. In particular, it contains a closed point $y\in \Spec(A_\red)$. Then, we see that $(\mfr{Y}^\diamond)^\wedge_{/y}$ contains a point in $R(\tfrac{f_1^{k+1},\ldots,f_n^{k+1}}{a})$, and we get a contradiction. 
\end{proof}

\begin{rmk}
    We find obstacles in running similar arguments for the full representability of $\cl{M}^\ints_{\cl{G}',b,\mu}$. Suppose that $\cl{M}^\ints_{\cl{G},b,\mu}$ is representable and let $\Spf(R) \subset \cl{M}^\ints_{\cl{G},b,\mu}$ be an affine open subspace. Let $R_0$ be the underlying $p$-adic ring of $R$. First, we need to extend the universal local $\cl{G}$-shtuka over $\Spd(R)_{/O_F}$ to $\Spd(R_0)_{/O_F}$ and then we need to show that $R_0$ admits a good cover. If the local $\cl{G}$-shtuka comes from a $p$-divisible group, the first property follows from the algebraizability of $p$-divisible groups, and the second property follows from \Cref{lem:compperfcov} if the $p$-adic uniformization is available (cf.\ \cite[Corollary 4.9]{Mie20}). 
\end{rmk}

\subsection{Representability of local model diagrams}

Though we show the local representability of $\cl{M}^\ints_{\cl{G}',b,\mu}$, it is hard to study the singularities of $\cl{M}^\ints_{\cl{G}',b,\mu}$ in this way. One way to study the singularities is to construct a local model diagram. In this section, we propose a modification of the $v$-sheaf theoretic local model diagram in \cite[Section 4.9.3]{PR24} and show the representability of the diagram. 


\begin{const} \label{const:local_model_first}
    Let $\mfr{U} \subset \mfr{Y}$ be an affine formal open subscheme. For every map $\Spd(R,R^+) \to \mfr{U}^\diamond$ over $\Spd(O_{\breve{F}})$, it uniquely formalizes to a map $\Spd(R^+) \to \mfr{U}^\diamond$ and the $\cl{G}'$-level structure over $\Spd(R^+)$ corresponds to a $\cl{G}'$-BKF module $P'_{R^+}$ by \Cref{prop:GBKF}. The functor taking such $(R,R^+)$ to the set of the trivializations of $\varphi^*P'_{R^+}\vert_{R^+}$ is a $v$-sheaf over $\mfr{U}^\diamond$ by \cite[Theorem 8.7]{Sch17}. Let $\wtd{\mfr{U}}^\diamond$ denote that $v$-sheaf and let $\wtd{\mfr{Y}}^\diamond$ be a $v$-sheaf over $\mfr{Y}^\diamond$ obtained by gluing $\wtd{\mfr{U}}^\diamond$ for all $\mfr{U} \subset \mfr{Y}$. 
\end{const}

For technical reasons, we first work with the $p$-adic formal scheme $\mfr{Y}$. After establishing \Cref{prop:LMDrep}, we may take the $\mfr{m}$-adic completion or the formal completion at a closed point of $\mfr{Y}$ to recover a usual local model diagram. 

\begin{rmk}
    The difference from \cite{PR24} is that for each $(R, R^+)$ over $\mfr{Y}$, $\wtd{\mfr{Y}}^\diamond$ parametrizes a trivialization of a $\cl{G}'$-torsor over $R^+$ and $\mfr{\wtd{Y}}^\diamond$ is a $\cl{G}'^\diamond$-torsor over $\mfr{Y}$. In \cite{PR24}, a $\cl{G}'^\diamondsuit$-torsor over $\mfr{Y}$ is considered instead, which is bigger than $\mfr{\wtd{Y}}^\diamond$ and is not represented directly by formal schemes (see \cite[Conjecture 4.9.3]{PR24} for their precise formulation). 
\end{rmk}

Recall that $\cl{M}_{\cl{G}',-\mu}^\loc$ is the $v$-sheaf theoretic local model introduced in \cite{SW20}. Following \cite[Section 4.9.1]{PR24}, we construct a map $\wtd{\mfr{Y}}^\diamond \to \cl{M}^\loc_{\cl{G}',-\mu}$. 

\begin{lem} \label{lem:validity_of_qx}
    In the setting of \Cref{const:local_model_first}, let $t \colon \cl{G}'_{R^+} \cong \varphi^*P'_{R^+} \vert_{R^+}$ be a trivialization. It admits a lift $\wtd{t} \colon \cl{G}'_{W_{O_F}(R^{\flat +})} \cong \varphi^* P'_{R^+}$ over $W_{O_F}(R^{\flat+})$ and the composition 
    \[
        \varphi_{P'_{R^+}} \circ \wtd{t} \colon \cl{G}'_{W_{O_F}(R^{\flat +})}[1/\xi_{R^+}] \cong \varphi^* P'_{R^+}[1/\xi_{R^+}] \cong P'_{R^+}[1/\xi_{R^+}]
    \]
    provides a map $\Spd(R^+) \to \cl{M}_{\cl{G}', -\mu}^\loc$. Moreover, this map is independent of the choice of $\wtd{t}$. 
\end{lem}

\begin{proof}
    The existence of $\wtd{t}$ follows from the smoothness of $\cl{G}'$. By definition, $\varphi_{P'_{R^+}} \circ \wtd{t}$ corresponds to a map $\Spd(R^+) \to \Gr_{\cl{G}}$. By the proof of \Cref{thm:localrep}, $\varphi_{P'_{R^+}}$ is bounded by $\cl{M}_{\cl{G}', -\mu}^\loc$, so it implies that the map factors through $\cl{M}_{\cl{G}', -\mu}^\loc$. 
    
    It remains to show that the resulting map $\Spd(R^+) \to \cl{M}_{\cl{G}', -\mu}^\loc$ is independent of the choice of $\wtd{t}$. Suppose that we have another choice. It is written as $\wtd{t} \circ \ell_g$ for some $g \in \cl{G}'(W_{O_F}(R^{\flat +}))$ where $\ell_g$ denotes the left translation by $g$ and $g\vert_{R^{+}}$ is trivial. By \cite[Lemma 3.15, Theorem 7.22]{AGLR22}, the natural $L^+_{O_F}\cl{G}'$-action (see \cite[(4.29), Proposition 4.13]{AGLR22}) on $\cl{M}^\loc_{\cl{G}', -\mu}$ factors through $\cl{G}'^{\diamondsuit}$. Taking the composition with $\ell_g$ is the action of $g \colon \Spd(R^+) \to L^+_{O_F}\cl{G}'$. Since $g\vert_{R^+}$ is trivial, its action is trivial. Thus, $\varphi_{P'_{R^+}} \circ \wtd{t} \circ \ell_g$ induces the same map as $\varphi_{P'_{R^+}} \circ \wtd{t}$.
\end{proof}

\begin{defi} \label{defi:local_model_diagram}
    In the setting of \Cref{const:local_model_first}, let $\pi_x \colon \wtd{\mfr{Y}}^\diamond \to \mfr{Y}^\diamond$ be the natural map and let $q_x \colon \wtd{\mfr{Y}}^\diamond \to \cl{M}^\loc_{\cl{G}', -\mu}$ be the map sending $t$ to the composition $\Spd(R, R^+) \hookrightarrow \Spd(R^+) \to \cl{M}^\loc_{\cl{G}', -\mu}$ associated to $t$ under \Cref{lem:validity_of_qx}. We say that the diagram 
    \begin{center}
        \begin{tikzcd}
            & \wtd{\mfr{Y}}^\diamond \ar[ld, "\pi_x"'] \ar[rd, "q_x"] & \\
            \mfr{Y}^\diamond & &
            \cl{M}^\mrm{loc}_{\cl{G}',-\mu}
        \end{tikzcd}
    \end{center}
    is the $v$-sheaf theoretic local model diagram. 
\end{defi}

From now on, we will prove the representability of this diagram in the category of $p$-adic formal schemes. First, we prove the hyperspecial case $\cl{G}'=\cl{G}$. 

\begin{lem}
    There is a $\cl{G}$-torsor $Q$ over $R$ with a $\cl{G}$-equivariant morphism $Q\to M^\mrm{loc}_{\cl{G},-\mu}$ that represents the $v$-sheaf theoretic local model diagram for $\cl{G}'=\cl{G}$. 
\end{lem}

\begin{proof}
    We prove by the descent theory of prismatic $(\cl{G},\mu)$-displays developed in \cite[Section 6]{Ito25b}. Recall the notation in \Cref{sssec:prism_display}. 

    As in \cite[Lemma 6.2.4 (2)]{Ito25b}, let $(A, I) = (\mfr{S}, \cl{E})$ and let $(A',I') \in (R_{\cl{G}, \mu})_{\Prism, O_F}$ be the coproduct $(A,I)\coprod (A,I)$. Let $p_i\colon (A,I) \to (A',I')$ ($i=1,2$) be the natural inclusion and let $m\colon (A',I') \to (A,I)$ be the multiplication map such that $m\circ p_1=m\circ p_2=\mrm{id}_{(A,I)}$. 

    For a morphism $f\colon (B,J)\to (B',J')$ in $(R_{\cl{G}, \mu})_{\Prism, O_F}$, let $\ov{f}\colon B/J\to B'/J'$ be the quotient map and let $\psi_f\colon \mfr{Q}^\univ_{B'/J'}\to \ov{f}^*\mfr{Q}^\univ_{B/J}$ be the map induced by the functoriality of $\mfr{Q}^\univ$. Note that $\ov{p}_1 = \ov{p}_2$ since $A / I = R_{\cl{G}, \mu}$. 

    Consider an isomorphism 
    \[
        \iota \colon p_1^*\mfr{Q}^\univ_{(A,I)}\cong \mfr{Q}^\univ_{(A', I')} \cong p_2^*\mfr{Q}^\univ_{(A,I)}
    \]
    induced by the functoriality of $\mfr{Q}^\univ$. It satisfies $m^*\iota = \id$, and this uniquely determines $\iota$ by \cite[Proposition 6.4.1]{Ito25b}. If we write $\mfr{Q}^\mrm{univ}_{(A,I)}=\mfr{Q}_Y$ with $Y\in \cl{G}(A)_I$, the proof of \cite[Proposition 6.4.1, p.1751]{Ito25b} (for the surjectivity) shows that $\iota$ is given by an element $g\in \cl{G}(IK)$ (equivalent to the claim $g \in G(dK)$ in loc. cit.). It implies that the map $\ov{p_1}^*\mfr{Q}^\univ_{A/I} \cong \ov{p_2}^*\mfr{Q}^\univ_{A/I}$ induced by $\iota$ is equal to $g_{(A',I')}^*\id_{\mfr{Q}^\univ_{A/I}}$. This means $\psi_{p_1}=\psi_{p_2}$. 

    Now, let $(B,J)$ be a perfect $O_F$-prism over $R_{\cl{G}, \mu}$. The $\cl{G}$-torsor $\mfr{Q}^\univ_{(B,J)}$ is canonically isomorphic to $\varphi^*P_{B/J}$ where $P_{B/J}$ is the corresponding $\cl{G}$-BKF module over $B/J$ (see \Cref{rmk:display_vs_BKF}). By Andr\'{e}'s flatness lemma (see \cite[Theorem 7.14]{BS22}), there exists a faithfully flat cover $(B,J)\to (B',J')$ that admits a morphism $f\colon (A,I)\to (B',J')$. 

    We show that $\psi_f$ is independent of the choice of $f$. Let $f'\colon (A,I)\to (B',J')$ be another map and let $(f, f')\colon (A',I')\to (B',J')$ be the induced map. Then, we have
    \[
        \psi_f=\psi_{(f,f')}\circ \ov{(f, f')}^*\psi_{p_1} =\psi_{(f,f')}\circ \ov{(f,f')}^*\psi_{p_2}=\psi_{f'}.
    \]
    Thus, $\psi_f$ is independent of $f$. This implies that $\psi_f$ descends to $\psi_{(B,J)}\colon \mfr{Q}^\univ_{B/J}\cong g_{(B,J)}^*\mfr{Q}^\univ_{A/I}$. This isomorphism is compatible with any transition map, so $\mfr{Q}^\univ_{A/I}$ represents $\wtd{\mfr{Y}}^\diamond$ for $\mfr{Y} = \Spd(R_{\cl{G}, \mu})$. Next, we show the representability of $q_x \colon \wtd{\mfr{Y}}^\diamond \to \cl{M}_{\cl{G}, -\mu}^\loc$. 

    In the hyperspecial case, we have $M^\loc_{\cl{G}, -\mu} \cong \cl{G} / \cl{P}_\mu$ where 
    \[
        \cl{P}_\mu = \{g \in \cl{G} \mid \lim_{t \to 0} \mu(t) g \mu(t)^{-1} \; \text{exists} \}. 
    \]
    Then, a $\cl{G}$-equivariant map $\mfr{Q}^\univ_{A/I} \to M^\loc_{\cl{G}, -\mu}$ corresponds to a $\cl{P}_\mu$-subtorsor $P$ of $\mfr{Q}^\univ_{A / I}$. As in \cite[Definition 5.4.1]{Ito25b}, the natural map $\cl{G}_\mu(A, I) \to \cl{G}(A)$ factors through $\cl{P}_\mu(A)$ and $\mfr{Q}^\univ_{(A, I)}$ is naturally equipped with a $\cl{P}_\mu$-subtorsor $P$, which is called the Hodge filtration in loc. cit. We will show that the associated $\cl{G}$-equivariant map $\mfr{Q}^\univ_{A/I} \to M^\loc_{\cl{G}, -\mu}$ represents $q_x$. 
    
    Fix an identification $\cl{G}(A)_I \cong \cl{G}(A)$ via $\cl{E} \in I$ and take $Y \in \cl{G}(A)$ so that $\mfr{Q}^\univ_{(A, I)} \cong \cl{Q}_Y$. By \Cref{rmk:banal_display_BKF}, the composition 
    \[
        \alpha \colon \mfr{Q}^\univ_{(A, I)}[1/I] \cong \varphi^*(\mfr{Q}^\univ_{(A, I)})_\BK[1/I] \cong (\mfr{Q}^\univ_{(A, I)})_\BK[1/I]
    \]
    is given by $g \mapsto \mu(\cl{E})g$. As in \Cref{lem:validity_of_qx}, $\alpha$ induces a relative position map $\mfr{Q}^\univ_{A/I} \to M_{\cl{G}, -\mu}^\loc$ and its explicit description implies that the relative position map corresponds to $\cl{P}_\mu \subset \cl{G}$, which is the Hodge filtration of $\cl{Q}_Y$. The same holds for any perfect $O_F$-prism over $(A, I)$, so the Hodge filtration represents $q_x$. 
\end{proof}

\begin{prop}\label{prop:LMDrep}
    Let $\Spf(S) \subset \mfr{Y}$ be an affine open formal subscheme. There is a finite type affine $S$-scheme $Q'_S$ with a $\cl{G}'$-action and a $\cl{G}'$-equivariant morphism $Q'_S\to M^\mrm{loc}_{\cl{G}',-\mu}$ such that the $p$-adic completion
    \begin{center}
        \begin{tikzcd}
            & (Q'_S)^\wedge \ar[ld] \ar[rd] & \\
            \Spf(S) & &
            M^\loc_{\cl{G}',-\mu}
        \end{tikzcd}
    \end{center}
    represents the $v$-sheaf theoretic local model diagram over $\Spd(S)$. 
\end{prop}

\begin{proof}
    Let $(R_\bullet,\Gamma_\bullet)$ be a good cover of $R_{\cl{G}, \mu}$. By the proof of \Cref{thm:dilatation} and \Cref{rmk:existenceY}, $S$ admits a maximal good cover $(S_\bullet,\Gamma_\bullet)$ constructed as in \Cref{cor:expaffrep}.
    
    First, we will construct $Q'_S$. The $\cl{G}'$-level structure over $\Spd(S_\infty)$ corresponds to a $\cl{G}'$-BKF module $P'_{S_\infty}$ by \Cref{prop:GBKF}. Let $Q'_{S_\infty} = \varphi^*P'_{S_\infty} \vert_{S_\infty}$. By construction, we have $\wtd{\mfr{Y}}^\diamond \times_{\mfr{Y}^\diamond} \Spd(S_\infty) \cong (Q'^\wedge_{S_\infty})^\diamond$ and a natural $\cl{G}'$-equivariant map $Q'_{S_\infty} \to Q\otimes_R S_\infty$. By \Cref{lem:G'subtors}, this map corresponds to an $\cl{H}$-subtorsor of $Q\otimes_R (S_\infty/\pi S_\infty)$. By \Cref{lem:densityU} and \Cref{cor:expaffrep}, $S_\infty/\pi S_\infty \cong \colim_{n\geq 0} S_n/\pi S_n$. Thus, the $\cl{H}$-subtorsor of $Q$ can be defined over $S_n/\pi S_n$ for some sufficiently large $n\geq 0$. Again by \Cref{lem:G'subtors}, it corresponds to a $\cl{G}'$-subtorsor $Q'_{S_n}$ of $Q\otimes_R S_n$ and we have $Q'_{S_\infty} \cong Q'_{S_n} \otimes_{S_n} S_\infty$ equivariantly under $\Gamma_\infty^n$. Let $Q'^{\wedge}_{S_n}$ be the $p$-adic completion of $Q'_{S_n}$. By \Cref{lem:bcquot}, $(Q'^{\wedge}_{S_n})^{\diamond}$ is a geometric quotient of $(Q'^{\wedge}_{S_\infty})^{\diamond}$ by $\Gamma_\infty^n$, so we see that $\wtd{\mfr{Y}}^\diamond\times_{\mfr{Y}^\diamond} \Spd(S_n) \cong (Q'^{\wedge}_{S_n})^{\diamond}$. 

    Now, we have a natural map $Q'_{S_n} \to Q \otimes_R S_n$ and it is an isomorphism outside $V(p)$. In particular, we have $\Gamma(Q,\cl{O})\otimes_R S_n \hookrightarrow \Gamma(Q'_{S_n}, \cl{O}) \hookrightarrow \Gamma(Q[\tfrac{1}{p}], \cl{O}) \otimes_{R} S_n$. Since $Q'_{S_n}$ is flat over $S_n$ and $S_\infty/S_n$ is $p$-torsion free, we have $\Gamma(Q'_{S_n}, \cl{O}) = \Gamma(Q'_{S_\infty}, \cl{O}) \cap \Gamma(Q[\tfrac{1}{p}], \cl{O}) \otimes_R S_n$ inside $\Gamma(Q[\tfrac{1}{p}], \cl{O}) \otimes_R S_\infty$. Since $\Gamma_\infty$ acts on $P'_{S_\infty}$, $\Gamma_\infty$ also acts on $Q'_{S_\infty}$ and the action is compatible with the natural action on $Q\otimes_R S_\infty$. Thus, $\Gamma(Q'_{S_\infty}, \cl{O})$ is $\Gamma_\infty$-stable in $\Gamma(Q[\tfrac{1}{p}], \cl{O}) \otimes_R S_\infty$, so $\Gamma(Q'_{S_n}, \cl{O})$ is $\Gamma_n$-stable in $\Gamma(Q[\tfrac{1}{p}], \cl{O}) \otimes_R S_n$. Let $A_S=\Gamma(Q'_{S_n}, \cl{O})^{\Gamma_n}$. By \Cref{prop:finquot}, $(Q'^{\wedge}_{S_n})^{\diamond} \to \Spd(A_S)$ is a geometric quotient by $\Gamma_n$, so $\wtd{\mfr{Y}}^\diamond \times_{\mfr{Y}^\diamond} \Spd(S) \cong \Spd(A_S)$. Since $S_n$ is finite over $S$ and $Q'_{S_n}$ is of finite type over $S_n$, $A_S$ is of finite type over $S$. Moreover, the $\Gamma(\cl{G}', \cl{O})$-comodule structure of $\Gamma(Q'_S, \cl{O})$ naturally descends to $A_S$, so $Q'_S = \Spec(A_S)$ admits a $\cl{G}'$-action. 

    Next, we show that $(Q'^{\wedge}_{S_n})^{\diamond} \to \cl{M}^\loc_{\cl{G}',-\mu}$ is representable. Since $R_{\cl{G}, \mu}$ is a complete local ring with an algebraically closed residue field, $Q$ is a trivial $\cl{G}$-torsor.  Since $\cl{G}_{\ov{k}} \to \cl{G}_{\ov{k}}/\cl{H}_{\ov{k}}$ is Zariski locally split, for every sufficiently small affine open subset $\Spf(B) \subset \Spf(S_n)$, we can lift the element of $(Q/\cl{H})(B/\pi B)$ corresponding to $Q'_B = Q'_{S_n} \otimes_{S_n} B$ to an element of $Q(B/\pi B)$. Since $Q$ is smooth, it can be lifted to $Q(B)$, so $Q'_B$ is a trivial $\cl{G}'$-torsor. Let $t\colon \Spec(B) \to Q'_B$ be a trivialization. The composition $\Spd(B) \to (Q'^{\wedge}_{B})^{\diamond} \to \cl{M}^{\loc}_{\cl{G}',-\mu}$ can be represented by a morphism $\Spf(B) \to M^{\loc}_{\cl{G}',-\mu}$ by \Cref{lem:uniqmapmaxlgood} since its generic fiber can be represented by 
    \[
        \Spa(B)_\eta \to Q'^\wedge_{B,\eta} \to Q^\wedge_{B,\eta} \to M^\loc_{\cl{G},-\mu,\eta}\cong M^\loc_{\cl{G}',-\mu,\eta}. 
    \]
    Thus, we get a $\cl{G}'$-equivariant morphism $Q'^\wedge_B \to M^{\loc}_{\cl{G}',-\mu}$ that represents $(Q'^{\wedge}_{B})^{\diamond} \to \cl{M}^\loc_{\cl{G}',-\mu}$. It is easy to see that this map is independent of the choice of $t$. Thus, we can glue these morphisms to get a morphism $Q'^\wedge_{S_n} \to M^\loc_{\cl{G}',-\mu}$. By construction, it represents $(Q'^{\wedge}_{S_n})^{\diamond} \to \cl{M}^\loc_{\cl{G}',-\mu}$. 

    Now, we have the following diagram. 
    \begin{center}
        \begin{tikzcd}
            Q'^\wedge_{S_n} \ar[r] \ar[dr] &  Q'_{S_n} \ar[r] \ar[d, dotted]&  Q_{S_n} \ar[d] \\
            & M^\loc_{\cl{G}',-\mu} \ar[r] & M^\loc_{\cl{G},-\mu}.
      \end{tikzcd}
    \end{center}
    Since $Q'_{S_n} \to Q_{S_n}$ and $M^\loc_{\cl{G}',-\mu} \to M^\loc_{\cl{G},-\mu}$ are isomorphisms outside $V(p)$, we may apply \cite[Tag 0ARB]{stacks-project} to $Q'_{S_n}$ and $Q_{S_n} \times_{M^\loc_{\cl{G},-\mu}} M^\loc_{\cl{G}',-\mu}$ to see that $Q'^\wedge_{S_n} \to Q_{S_n} \times_{M^\loc_{\cl{G},-\mu}} M^\loc_{\cl{G}',-\mu}$ is uniquely algebraizable. Thus, we can uniquely fill the middle vertical arrow $Q'_{S_n} \to M^\loc_{\cl{G}',-\mu}$. 

    Since $Q_{S_n} \to M^\loc_{\cl{G},\mu}$ is invariant under $\Gamma_n$ by construction, $Q'_{S_n} \to M^\loc_{\cl{G}',-\mu}$ is also invariant under $\Gamma_n$. Since $Q'_{S_n} \to Q'_S$ is a coarse moduli space by \cite[Theorem 3.1]{Con05}, the above map factors through $Q'_S$. Since $(Q'^\wedge_{S_n})^\diamond \to \Spd(A_S)$ is surjective, the map represents $q_x\vert_{(Q'^\wedge_S)^\diamond}$. 
\end{proof}

\begin{rmk}
    In general, it is expected that $Q'_S$ is a $\cl{G}'$-torsor and the $\mfr{m}$-adic completion of $Q'_S\to M^\mrm{loc}_{\cl{G}',-\mu}$ is formally smooth, as a generalization of Grothendieck-Messing theory. The above proof is not enough to establish these properties. 
\end{rmk}

\begin{cor}
    The $v$-sheaf theoretic local model diagram in \Cref{defi:local_model_diagram} is representable in the category of $p$-adic formal schemes. 
\end{cor}
\begin{proof}
    We keep the notation in \Cref{prop:LMDrep}. Let $A_S^\wedge$ be the $p$-adic completion of $A_S$. We will show that $\Spf(A_S^\wedge)$ can be glued to a $p$-adic formal scheme $\wtd{\mfr{Y}}$ over $\mfr{Y}$. Let $\Spf(S') \hookrightarrow \Spf(S)$ be a distinguished affine formal open subscheme. Let $S'_n=S_n \otimes_S S'$. Since $S_n$ is finite over $S$, $S'_n$ is $p$-adically complete, and $(S'_\bullet, \Gamma_\bullet)$ is a maximal good cover of $S'$ by \Cref{lem:maxlgoodcovZarloc}. By construction, we have $Q'_{S'_n} \cong Q'_{S_n} \otimes_{S_n} S'_n$. Since $S \to S'$ is flat, we see that  $A_{S'} \cong A_S \otimes_S S'$. In particular, $\Spf(A_{S'}^\wedge) \hookrightarrow \Spf(A_S^\wedge)$ is a distinguished open immersion. Thus, we can glue $\Spf(A_S^\wedge)$ for all $\Spf(S) \hookrightarrow \mfr{Y}$ to get a $p$-adic formal scheme $\wtd{\mfr{Y}}$. By construction, $\wtd{\mfr{Y}}$ represents $\wtd{\mfr{Y}}^\diamond$ and admits a map $\wtd{\mfr{Y}} \to M^\loc_{\cl{G}', -\mu}$ that represents $q_x$. 
\end{proof}

\subsection{Canonical integral models of Shimura varieties}

In this section, we construct canonical integral models of Shimura varieties under hyperspecial levels assuming the existence of those at hyperspecial levels. In this section, $(-)^\wedge$ denotes the $p$-adic completion. 

Let $(\mbf{G}, \mbf{X})$ be a Shimura datum satisfying the axiom (SV5). Let $\msf{K}=\msf{K}_p\msf{K}^p\subset \mbf{G}(\bb{A}_f)$ be a compact open subgroup such that $\msf{K}^p\subset \mbf{G}(\bb{A}_f^p)$ is a sufficiently small (or neat) compact open subgroup and $\msf{K}_p\subset \mbf{G}(\bb{Q}_p)$ is a parahoric subgroup. Let $\mbf{E}$ be the reflex field of $(\mbf{G},\mbf{X})$. Let $v$ be a place of $\mbf{E}$ over $p$ and set $E=\mbf{E}_v$. The Shimura variety $\Sh_\msf{K}(\mbf{G},\mbf{X})$ is expected to admit an integral model $\scr{S}_{\msf{K}}$ over $O_E$. In the abelian type, $\scr{S}_\msf{K}$ is constructed in \cite{Kis10}, \cite{KP18} and \cite{KPZ24} under mild assumptions. 

The theory of canonical integral models is developed to characterize a family $\{\scr{S}_{\msf{K}}\}_{\msf{K}^p}$. In \cite{PR24}, canonical integral models are characterized in terms of local shtukas. Let us briefly recall the theory of canonical integral models after \cite{PR24}. 

Let $G=\mbf{G}_{\bb{Q}_p}$ and let $\cl{G}$ be a parahoric group scheme of $G$ over $\bb{Z}_p$ with $\cl{G}(\bb{Z}_p)=\msf{K}_p$. Let $\{\mu\}$ be the geometric conjugacy class of minuscule cocharacters of $G$ corresponding to the opposite of $\mbf{X}$ due to our sign convention. 
Now, $E$ is the reflex field of $\{\mu\}$. Let $k$ be the residue field of $E$ and fix an algebraic closure $\ov{k}$ of $k$. Let $O_{\breve{E}}=W_{O_E}(\ov{k})$ and $\breve{E}=O_{\breve{E}}[\tfrac{1}{p}]$. 
The axioms of canonical integral models $\{\scr{S}_\msf{K}\}_{\msf{K}^p}$, which is a family of normal flat separated $O_E$-scheme of finite type, are as follows (see \cite[Theorem 1.3.2]{PR24}): 

\begin{enumerate}
    \setcounter{enumi}{-1}
    \item $\{\scr{S}_\msf{K}\}_{\msf{K}^p}$ admits finite \'{e}tale prime-to-$p$ Hecke correspondences. 
    \item For every discrete valuation ring $R$ of mixed characteristic over $O_E$, 
    \[(\varprojlim_{\msf{K}^p} \Sh_{\msf{K}}(\mbf{G},\mbf{X}))(R[\tfrac{1}{p}]) = (\varprojlim_{\msf{K}^p} \scr{S}_\msf{K})(R).\]
    \item The local $\cl{G}$-shtuka $\cl{P}_{\msf{K},E}$ over $\Sh_{\msf{K}}(\mbf{G},\mbf{X})_E$ extends to a local $\cl{G}$-shtuka $\cl{P}_\msf{K}$ over $\scr{S}_\msf{K}$. 
    \item For every $x\in \scr{S}_\msf{K}(\ov{k})$, there exists an isomorphism 
    \[\Theta_x \colon (\cl{M}^\ints_{\cl{G},b_x,\mu})^\wedge_{/x_0} \cong (\scr{S}^\wedge_{\msf{K}/x})^{\diamond}\]
    for a suitable $b_x \in G(\breve{\bb{Q}}_p)$ and a base point $x_0 \in \cl{M}_{\cl{G},b_x,\mu}^\ints(\ov{k})$ that is compatible with local $\cl{G}$-shtukas on each side. 
\end{enumerate}

Note that in the condition (2), $\cl{P}_\msf{K}$ (resp.\ $\cl{P}_{\msf{K},E}$) is defined over the $v$-sheafification of the $p$-adic completion $\scr{S}_\msf{K}^\wedge$ (resp.\ the adic space $\Sh_{\msf{K}}(\mbf{G},\mbf{X})_E^\ad$). 

Now, by \cite{DHKZ24b}, \cite{IKY23} and \cite{DY24}, most of the integral models $\scr{S}_\msf{K}$ constructed in \cite{Kis10}, \cite{KP18} and \cite{KPZ24} are verified to be canonical in the sense of \cite{PR24}. Here, we explain a relation between integral models at two different parahoric levels in terms of level structures on local shtukas, which can be expected from the diagram in \cite[Theorem VII]{DHKZ24a}. 

Let $\cl{G}'$ be another parahoric group scheme of $G$ over $\bb{Z}_p$ with a morphism $\cl{G}'\to \cl{G}$. It corresponds to a parahoric subgroup $\msf{K}'_p \subset \msf{K}_p$. Let $\msf{K}'=\msf{K}'_p\msf{K}^p$ for each $\msf{K}^p$. 

\begin{thm}\label{thm:canint}
    Suppose that $\msf{K}_p$ is hyperspecial and there exists a system of canonical integral models $\{\scr{S}_\msf{K}\}_{\msf{K}^p}$ of $\{\mrm{Sh}_\msf{K}(\mbf{G},\mbf{X})_E\}_{\msf{K}^p}$ in the sense of Pappas-Rapoport. 
    For every $\msf{K}^p$, there is a normal flat separated algebraic space $\scr{S}_\msf{K'}$ of finite type over $O_E$ with $\scr{S}_{\msf{K}', E} \cong \mrm{Sh}_{\msf{K'}}(\mbf{G},\mbf{X})_E$ such that 
    \begin{enumerate}
        \item $\scr{S}_{\msf{K'}}$ is proper and surjective over $\scr{S}_{\msf{K}}$, and
        \item $\scr{S}_{\msf{K'}}^\wedge$ is the flat moduli space of $\cl{G}'$-level structures on $\cl{P}_\msf{K}$. 
    \end{enumerate}
    Moreover, $\{\scr{S}_\msf{K'}\}_{\msf{K}^p}$ satisfies the above axioms (0)-(3) of canonical integral models. 
\end{thm}

\begin{proof}
    Let $x\in \scr{S}_\msf{K}(\ov{k})$ be a closed point. By the axiom (3), we have $(\scr{S}^\wedge_{\msf{K}/x})^\diamond \cong (\cl{M}_{\cl{G},b_x,\mu}^\ints)^\wedge_{/x_0}$. By \cite{Bar22} and \cite{Ito25a}, $(\cl{M}_{\cl{G},b_x,\mu}^\ints)^\wedge_{/x_0}$ is represented by $R_{\cl{G}, \mu} \cong O_{\breve{F}}\llbracket T_1,\ldots,T_r \rrbracket$ for some $r\geq 0$. Note that $E$ is unramified over $\bb{Q}_p$ since $G$ splits over $\breve{\bb{Q}}_p$. Since $\scr{S}_\msf{K}$ is normal, flat and of finite type, it follows from \cite[Proposition 18.4.1]{SW20} that $(\scr{S}_\msf{K})^\wedge_{/x} \cong \Spf(O_{\breve{F}}\llbracket T_1,\ldots,T_r \rrbracket)$. Thus, $\scr{S}_\msf{K}$ is smooth over $O_E$. 

    By \Cref{prop:smoothadmit}, we may apply \Cref{thm:replevel} to the local $\cl{G}$-shtuka $\cl{P}_\msf{K}$. Let $\cl{S}_{\msf{K}'}$ be the flat moduli space of $\cl{G}'$-level structures on $\cl{P}_\msf{K}$. Let $x\in \scr{S}_\msf{K}(\ov{k})$ be a closed point and take $b_x$ and $x_0$ as in the axiom (3). We apply \Cref{thm:localrep} after replacing $\cl{P}^\univ$ to the one induced from $\cl{P}_\msf{K}$ along the pullback $\Spec(R) \to \scr{S}_\msf{K}$, which does not matter in the proof. Then, $\cl{M}^\mrm{int}_{\cl{G}',b_x,\mu}\times_{\cl{M}^\mrm{int}_{\cl{G},b_x,\mu}} (\cl{M}^\mrm{int}_{\cl{G},b_x,\mu})^\wedge_{/x_0}$ is represented by $\mfr{Y}^{\wedge}$. By construction, we have a map $\mfr{Y}^\diamond \to \cl{S}_{\msf{K}'}^\diamond$. Since there is a natural morphism $\mfr{Y}_\eta \to \cl{S}_{\msf{K}',\eta}$ of adic spaces by \cite[Lemma 15.6]{Sch17}, that map is uniquely represented by a morphism $\mfr{Y} \to \cl{S}_{\msf{K}'}$ by \Cref{lem:uniqmapmaxlgood}. Thus, we have a morphism $\mfr{Y} \to \cl{S}_{\msf{K}'} \times_{\scr{S}_\msf{K}^\wedge} \Spf(R)$. Since $\Spec(R) \to \scr{S}_\msf{K}$ is flat, $\cl{S}_{\msf{K}'} \times_{\scr{S}_\msf{K}^\wedge} \Spf(R)$ is flat over $O_{\breve{F}}$. Then, $\mfr{Y}^\diamond$ and $\cl{S}_{\msf{K}'}^\diamond \times_{(\scr{S}_\msf{K}^\wedge)^\diamond} \Spd(R)$ are both thick closed subsheaves of $\Sht_{\cl{G}'}\times_{\Sht_{\cl{G}}, \cl{P}_\msf{K}} \Spd(R)$, so it follows from the uniqueness of thick closures that $\mfr{Y}^\diamond \cong \cl{S}_{\msf{K}'}^\diamond \times_{(\scr{S}_\msf{K}^\wedge)^\diamond} \Spd(R)$. Thus, 
    \[\cl{S}_{\msf{K}'}^\diamond \times_{(\scr{S}_{\msf{K}}^\wedge)^\diamond}(\scr{S}_{\msf{K}/x}^\wedge)^\diamond \cong \cl{M}^\ints_{\cl{G}',b_x,\mu}\times_{\cl{M}^\ints_{\cl{G},b_x,\mu}} (\cl{M}^\ints_{\cl{G},b_x,\mu})^\wedge_{/x_0}.\]
    In particular, for every $y\in \cl{S}_{\msf{K}'}(\ov{k})$, there exists an isomorphism $(\cl{S}_{\msf{K}'/y}^\wedge)^\diamond \cong (\cl{M}^\ints_{\cl{G}',b_y,\mu})^\wedge_{/y_0}$ with $b_y=b_x$ and $y_0 \in \cl{M}^\ints_{\cl{G}',b_x,\mu}(\ov{k})$ where $x\in \scr{S}_\msf{K}(\ov{k})$ is the image of $y$. 

    Now, since $\cl{S}_{\msf{K}',\eta}$ is finite etale over $(\scr{S}_\msf{K}^\wedge)_\eta$, $\cl{S}_{\msf{K}',\eta}$ is smooth, so in particular, normal. It follows from \cite[Lemma 2.13]{AGLR22} and \Cref{lem:uniqmapmaxlgood} that $\cl{S}_{\msf{K}'}$ is absolutely weakly normal. Thus, it follows from \cite[Proposition 3.4.1 (2)]{PR24} and \cite[Proposition 2.39]{AGLR22} that $\cl{S}_{\msf{K}'}$ is normal.

    Next, we show that $\cl{S}_{\msf{K}',\eta} \to (\scr{S}_{\msf{K}}^\wedge)_\eta$ is the pullback of $\Sh_{\msf{K}'}(\mbf{G},\mbf{X})_E \to \Sh_{\msf{K}}(\mbf{G},\mbf{X})_E$. By construction in \cite[Proposition 4.1.2]{PR24}, $\cl{P}_{\msf{K},E}$ corresponds to the pro-\'{e}tale $\msf{K}_p$-cover
    \[
        \varprojlim_{\msf{K}''_p \subset \msf{K}_p} \Sh_{\msf{K}''_p\msf{K}^p}(\mbf{G},\mbf{X}) \to \Sh_\msf{K}(\mbf{G},\mbf{X}). 
    \]
    By the proof of \Cref{lem:fibSht}, $\Sht_{\cl{G}',\eta}\times_{\Sht_{\cl{G},\eta},\cl{P}_{\msf{K},E}} \Sh_\msf{K}(\mbf{G},\mbf{X})_E^\ad$ classifies $\msf{K}'_p$-subcovers of that pro-\'{e}tale $\msf{K}_p$-cover, and $\Sh_{\msf{K}''_p\msf{K}^p}(\mbf{G},\mbf{X})$ is isomorphic to $(\varprojlim_{\msf{K}''_p} \Sh_{\msf{K}''_p\msf{K}^p}(\mbf{G},\mbf{X}))/\msf{K}''_p$. Thus, we have $\Sht_{\cl{G}',\eta}\times_{\Sht_{\cl{G},\eta},\cl{P}_{\msf{K},E}} \Sh_\msf{K}(\mbf{G},\mbf{X})_E^\ad \cong \Sh_{\msf{K}'}(\mbf{G},\mbf{X})_E^\ad$ and we get the claim. 

    Finally, we construct an algebraic model $\scr{S}_{\msf{K}'}$ of $\cl{S}_{\msf{K}'}$. Let $U\subset \scr{S}_{\msf{K}}$ be an arbitrary affine open subset and let $U=\Spec(A)$. Let $A[\tfrac{1}{p}] \to B$ be the finite \'{e}tale morphism representing $\Sh_{\msf{K}'}(\mbf{G},\mbf{X})_E \to \Sh_\msf{K}(\mbf{G},\mbf{X})_E$ over $U_\eta$. Let $B^+ \subset B$ be the integral closure of $A$. Let $\cl{S}_{\msf{K}',U} \subset \cl{S}_{\msf{K}'}$ be the open formal subscheme over $U^\wedge \subset \scr{S}_{\msf{K}}^\wedge$. Since $\cl{S}_{\msf{K}',U}$ admits a maximal good cover by construction, it follows from \Cref{lem:propoftmaxlgood} that $\Gamma(\cl{S}_{\msf{K}',U,\eta}, \cl{O}^{\circ \circ}) \hookrightarrow \Gamma(\cl{S}_{\msf{K}',U}, \cl{O}) \hookrightarrow \Gamma(\cl{S}_{\msf{K}',U,\eta}, \cl{O}^{+}) \cong (B^+)^{\wedge}$. Thus, there is a finite $A$-subalgebra $B_0\subset B^+$ such that $B_0[\tfrac{1}{p}]=B$ and there is a homomorphism $B_0^\wedge \to \Gamma(\cl{S}_{\msf{K}',U}, \cl{O})$. By \Cref{prop:formalmod}, $\cl{S}_{\msf{K}',U} \to \Spf(B_0^\wedge)$ is a formal modification. By \cite[Theorem 3.2]{Art70}, there is a unique proper morphism $\scr{S}_{\msf{K}',U} \to \Spec(B_0)$ of algebraic spaces that is an isomorphism outside $V(p)$ and induces $\cl{S}_{\msf{K}',U} \to \Spf(B_0^\wedge)$ by taking the $p$-adic completion. It is easy to see from the uniqueness of $\scr{S}_{\msf{K}',U}$ that $\scr{S}_{\msf{K}',U}$ is independent of the choice of $B_0 \subset B$ and its construction is functorial in $U$. In particular, we can glue them to obtain an algebraic space $\scr{S}_{\msf{K}'}$ over $\scr{S}_\msf{K}$ such that $\scr{S}_{\msf{K}'}[\tfrac{1}{p}] \cong \Sh_{\msf{K}'}(\mbf{G},\mbf{X})_E$ and $\scr{S}_{\msf{K}'}^\wedge \cong \cl{S}_{\msf{K}'}$. 
    
    We prove that $\scr{S}_{\msf{K}'}$ is normal. Take an \'{e}tale cover $V \to \scr{S}_{\msf{K}'}$ from a scheme $V$ and consider the $p$-adic completion $V^\wedge \to \cl{S}_{\msf{K}',U}$. Since any $p$-adically completely \'{e}tale homomorphism can be written as the $p$-adic completion of an \'{e}tale homomorphism by \cite{Elk73}, it follows from \cite[Theorem 79]{Mat80} and the normality of $\cl{S}_{\msf{K}'}$ that $V^\wedge$ is normal. For any affine open subscheme $\Spec(C) \subset V$, $C\to C^\wedge \times C[\tfrac{1}{p}]$ is faithfully flat, so $C$ is normal by \cite[Corollary 21.3]{Mat80}. Thus, $V$ is normal, and $\scr{S}_{\msf{K}'}$ is normal. Moreover, since $\cl{S}_{K',U}$ is $p$-torsion free, $V^\wedge$ is $p$-torsion free. Thus, $C$ is $p$-torsion free and $V$ is flat over $O_E$. 
 
    Now, we have an integral model $\scr{S}_{\msf{K}'}$ of $\Sh_{\msf{\msf{K}'}}(\mbf{G},\mbf{X})_E$. By construction, $\scr{S}_{\msf{K}'}$ is proper over $\scr{S}_\msf{K}$. Since $\scr{S}_\msf{K}$ is flat over $O_E$ and $\Sh_{\msf{K}'}(\mbf{G},\mbf{X})_E \to \Sh_{\msf{K}}(\mbf{G},\mbf{X})_E$ is surjective, $\scr{S}_{\msf{K}'} \to \scr{S}_\msf{K}$ is surjective. Thus, the condition (1) and (2) is verified for $\scr{S}_{\msf{K}'}$. We will verify that $\{\scr{S}_{\msf{K}'}\}_{\msf{K}^p}$ satisfies the axioms (0)-(3). 

    First, we check the axiom (1). Let $R$ be a discrete valuation ring of mixed characteristic over $O_E$. Since $\scr{S}_{\msf{K}'}$ is proper over $\scr{S}_\msf{K}$, it follows from the valuative criterion that any element $y\in \Sh_{\msf{K}'}(\msf{G},\msf{X})(R[\tfrac{1}{p}])$ can be uniquely lifted to $\scr{S}_{\msf{K}'}(R)$ as soon as we have a lift to $\scr{S}_\msf{K}(R)$ of the image of $y$ in $\Sh_{\msf{K}}(\msf{G},\msf{X})(R[\tfrac{1}{p}])$. Thus, the extension property of $\{\scr{S}_{\msf{K}'}\}_{\msf{K}^p}$ follows from that of $\{\scr{S}_{\msf{K}}\}_{\msf{K}^p}$. The axiom (2) follows from the construction because the universal $\cl{G}'$-level structure over $\cl{S}_{\msf{K}'}$ is the unique extension of $\cl{P}_{\msf{K}',E}$. The axiom (3) is already proved for $\cl{S}_{\msf{K}'}$. Thus, it is enough to check that $\{\scr{S}_{\msf{K}'}\}_{\msf{K}^p}$ admits finite \'{e}tale prime-to-$p$ Hecke correspondences. 

    Let $\msf{K}^{p}_i$ ($i=0,1$) be sufficiently small (or neat) compact open subgroups with $g \msf{K}^p_0 g^{-1} \subset \msf{K}_p^1$ for $g \in G(\bb{A}_f^p)$. Let $\msf{K}_i=\msf{K}_p\msf{K}^p_i$ (resp.\ $\msf{K}'_i=\msf{K}'_p\msf{K}^p_i$) for $i=0,1$. By assumption, we have a finite \'{e}tale prime-to-$p$ Hecke correspondence $\scr{S}_{\msf{K}_0} \to \scr{S}_{\msf{K}_1}$ associated with $g$. By \cite[Theorem 2.7.7, Proposition 4.1.2]{PR24}, $\cl{P}_{\msf{K}_0}$ is isomorphic to the pullback of $\cl{P}_{\msf{K}_1}$ along the Hecke correspondence of $g$. Thus, we have a closed immersion $\cl{S}_{\msf{K}'_0}^\diamond \to \cl{S}_{\msf{K}'_1}^\diamond \times_{(\scr{S}_{\msf{K}_1}^\wedge)^\diamond} (\scr{S}_{\msf{K}_0}^\wedge)^\diamond$ by construction. Since $\scr{S}_{\msf{K}_0} \to \scr{S}_{\msf{K}_1}$ is finite \'{e}tale, $\scr{S}_{\msf{K}'_1} \times_{\scr{S}_{\msf{K}_1}} \scr{S}_{\msf{K}_0}$ is normal and flat over $O_E$. In particular, $\cl{S}_{\msf{K}'_1}^\diamond \times_{(\scr{S}_{\msf{K}_1}^\wedge)^\diamond} (\scr{S}_{\msf{K}_0}^\wedge)^\diamond$ is thick, so $\cl{S}_{\msf{K}'_0}^\diamond \to \cl{S}_{\msf{K}'_1}^\diamond \times_{(\scr{S}_{\msf{K}_1}^\wedge)^\diamond} (\scr{S}_{\msf{K}_0}^\wedge)^\diamond$ is an isomorphism. By \cite[Proposition 18.4.1]{SW20}, $\cl{S}_{\msf{K}'_0} \to \cl{S}_{\msf{K}'_1} \times_{\scr{S}_{\msf{K}_1}^\wedge} \scr{S}_{\msf{K}_0}^\wedge$ is an isomorphism, so it follows from \cite[Theorem 3.2]{Art70} that $\scr{S}_{\msf{K}'_0} \to \scr{S}_{\msf{K}'_1} \times_{\scr{S}_{\msf{K}_1}} \scr{S}_{\msf{K}_0}$ is an isomorphism. In particular, we obtain a finite \'{e}tale Hecke correspondence $\scr{S}_{\msf{K}'_0} \to \scr{S}_{\msf{K}'_1}$ associated with $g$. 
\end{proof}

As a result, we obtain the following corollary. We keep the notation in \Cref{thm:canint}. 

\begin{cor}
    The $p$-adic completion $\scr{S}_{\msf{K}'}^\wedge$ represents the $v$-sheaf parametrizing $\cl{G}'$-level structures on $\cl{P}_\msf{K}$ bounded by $\cl{M}_{\cl{G}',- \mu}^\loc$. Moreover, $(\cl{M}^\ints_{\cl{G}',b,\mu})^\wedge_{/y}$ is representable by a normal $p$-torsion free complete local Noetherian ring if $b=b_x$ and $y$ is over $x_0$ for some $x\in \scr{S}_{\msf{K}}(\ov{k})$. 
\end{cor}
\begin{proof}
    The claim follows from $\cl{S}_{\msf{K}'}^\diamond \times_{(\scr{S}_{\msf{K}}^\wedge)^\diamond}(\scr{S}_{\msf{K}/x}^\wedge)^\diamond \cong \cl{M}^\ints_{\cl{G}',b_x,\mu}\times_{\cl{M}^\ints_{\cl{G},b_x,\mu}} (\cl{M}^\ints_{\cl{G},b_x,\mu})^\wedge_{/x_0}$. 
\end{proof}


Even if $\scr{S}_{\msf{K}', E}$ is a scheme and $\scr{S}_{\msf{K'}, \bb{F}_p}$ is locally projective over a scheme $\scr{S}_{\msf{K},\bb{F}_p}$, it does not directly follow that $\scr{S}_{\msf{K}'}$ is a scheme (for example, see \cite[Section 5.2]{Mat15}). We will show that $\scr{S}_{\msf{K}'}$ is a canonical integral model of $\Sh_{\msf{K}'}(\mbf{G},\mbf{X})_{E}$ if there exists one, by partially following the proof of uniqueness of canonical integral models (see \cite[Theorem 7.1.7]{Pap23}). 

\begin{prop} \label{prop:equalcan}
    Under the situation in \Cref{thm:canint}, if there is a system of canonical integral models $\{S_{\msf{K}'}\}$ of $\{\Sh_{\msf{K}'}(\mbf{G}, \mbf{X})_E\}$, we have an isomorphism $S_{\msf{K}'} \cong \scr{S}_{\msf{K}'}$.
\end{prop}

\begin{proof}
    By \cite[Corollary 4.3.2]{PR24}, there is a natural map $S_{\msf{K}'} \to \scr{S}_{\msf{K}}$. Since $\cl{P}_{\msf{K}'}$ is a $\cl{G}'$-level structure on $\cl{P}_{\msf{K}}$ over $(S_{\msf{K}'}^\wedge)^\diamond$, we have a map $(S_{\msf{K}'}^\wedge)^\diamond \to \Sht_{\cl{G}'}\times_{\Sht_{\cl{G}}, \cl{P}_{\msf{K}}} (\scr{S}_\msf{K}^\wedge)^\diamond$. Since $S_{\msf{K}'}$ is flat over $O_E$, $(S_{\msf{K}',\eta}^\wedge)^\diamond$ is dense in $(S_{\msf{K}'}^\wedge)^\diamond$ (see \cite[Lemma 4.4]{Lou20}), so we have a map $(S_{\msf{K}'}^\wedge)^\diamond \to (\cl{S}_{\msf{K'}})^\diamond$. By \cite[Proposition 18.4.1]{SW20}, it is representable by a morphism $S_{\msf{K}'}^\wedge \to \cl{S}_{\msf{K}'}$. 
    
    Let $U\subset \scr{S}_\msf{K}$ be an affine open subset and let $U=\Spec(A)$. Let $A[\tfrac{1}{p}] \to B$ be the finite \'{e}tale morphism representing $\Sh_{\msf{K}'}(\mbf{G},\mbf{X})_E \to \Sh_\msf{K}(\mbf{G},\mbf{X})_E$ over $U_\eta$. Let $B^+ \subset B$ be the integral closure of $A$. Let $S_{\msf{K}',U} \subset S_{\msf{K}'}$ be the inverse image of $U$. Since $S_{\msf{K}',E} \cong \Sh_{\msf{K}'}(\mbf{G}, \mbf{X})_E$, we have $p^N\cdot B^+ \subset \Gamma(S_{\msf{K}',U},\cl{O})$ for some $N\geq 0$. By setting $B_0 = A+p^N\cdot B^+$ for sufficiently large $N\geq 0$, we get a finite $A$-algebra $B_0$ admitting $S_{\msf{K}',U} \to \Spec(B_0)$ and $\scr{S}_{\msf{K}',U} \to \Spec(B_0)$ that are isomorphisms outside $V(p)$. By \cite[Tag 0ARB]{stacks-project}, $S_{\msf{K}',U}^\wedge \to \cl{S}_{\msf{K}',U}$ is uniquely algebraizable to $S_{\msf{K}',U} \to \scr{S}_{\msf{K}',U}$. By gluing these morphisms, we see that $S_{\msf{K}'}^\wedge \to \cl{S}_{\msf{K}'}$ is uniquely algebraizable to $S_{\msf{K}'} \to \scr{S}_{\msf{K}'}$. We will show that this is an isomorphism. 

    Let $y\in S_{\msf{K}',U}(\ov{k})$ be a closed point and let $x\in \scr{S}_{\msf{K}}(\ov{k})$ be the image of $y$. By \cite[(4.7.1), (4.8.2)]{PR24}, we have the following commutative diagram for a unique $y'\in \scr{S}_{\msf{K}'}(\ov{k})$. 
    \begin{center}
        \begin{tikzcd}
            (S_{\msf{K}'/y}^\wedge)^\diamond \ar[d] \ar[r, "\cong"] & (\cl{M}^\ints_{\cl{G}',b_y,\mu})^\wedge_{/y_0} \ar[d] & (\scr{S}_{\msf{K}'/y'}^\wedge)^\diamond \ar[l,"\cong"'] \ar[d] \\
            (\scr{S}_{\msf{K}/x}^\wedge)^\diamond  \ar[r, "\cong"] & (\cl{M}^\ints_{\cl{G},b_x,\mu})^\wedge_{/x_0} & (\scr{S}_{\msf{K}/x}^\wedge)^\diamond \ar[l,"\cong"'] 
        \end{tikzcd}
    \end{center}
    Since $(S_{\msf{K}'}^\wedge)^\diamond \to (\scr{S}_{\msf{K}'}^\wedge)^\diamond$ is induced by the $\cl{G}'$-level structure $\cl{P}_{\msf{K'}}$, the isomorphism $(S_{\msf{K}'/y}^\wedge)^\diamond \cong (\scr{S}_{\msf{K}'/y'}^\wedge)^\diamond$ commutes with $(S_{\msf{K}'}^\wedge)^\diamond \to (\scr{S}_{\msf{K}'}^\wedge)^\diamond$. In particular, $y'$ is the image of $y$ and we have $S_{\msf{K}'/y}^\wedge \cong \scr{S}_{\msf{K}'/y'}^\wedge$ by \cite[Proposition 18.4.1]{SW20}. It follows that $S_{\msf{K}'}(\ov{k}) \to \scr{S}_{\msf{K}'}(\ov{k})$ is injective since if $y_0, y_1\in S_{\msf{K}'}(\ov{k})$ maps to $y' \in \scr{S}_{\msf{K}'}(\ov{k})$, both of $(S_{\msf{K}'/y_i}^\wedge)_\eta$ ($i=0,1$) isomorphically map to $(\scr{S}_{\msf{K}'/y'}^\wedge)_\eta$, so the injectivity of $(S_\msf{K'}^\wedge)_\eta \to (\scr{S}_{\msf{K}'}^\wedge)_\eta$ implies $y_0=y_1$. 
    
    Let $V\to \scr{S}_{\msf{K}'}$ be an \'{e}tale cover by an $O_E$-scheme $V$ and let $S_{\msf{K}',V} = S_{\msf{K}'}\times_{\scr{S}_{\msf{K}'}} V$. Since $S_{\msf{K}',V}(\ov{k}) \to V(\ov{k})$ is injective and $(S_{\msf{K}',V})_{E} \to V_{E}$ is an isomorphism, $S_{\msf{K}',V}\to V$ is quasi-finite by \cite[Corollaire 13.1.4]{EGA4-4}. Since $S_{\msf{K}'}$ and $\scr{S}_{\msf{K'}}$ are separated over $O_E$, we may apply Zariski's main theorem to $S_{\msf{K}',V} \to V$. Since it is birational and $V$ is normal, $S_{\msf{K}',V} \to V$ is an open immersion. 

    Suppose that there is a closed point $y_0\in V(\ov{k})$ outside $S_{\msf{K}',V}$. Since $V$ is flat over $O_E$, $y_0$ can be lifted to a point $y\in V(R)$ with $R$ a complete discrete valuation ring of mixed characteristic over $O_E$. Since $R$ is strictly henselian and prime-to-$p$ Hecke correspondences are finite \'{e}tale, $y$ can be lifted to $y^p \in \lim_{\msf{K}^p} \scr{S}_{\msf{K}'}(R)$. By the extension property of $S_{\msf{K}'}$, $\{y^p_{R[1/p]}\}$ can be lifted to an $R$-valued point of $\lim_{\msf{K}^p} S_{\msf{K}'}$. In particular, $y_{R[1/p]}$ can be lifted to $S_{\msf{K}'}(R)$. Since $\scr{S}_{\msf{K}'}$ is separated, this contradicts the assumption on $y_0$. Thus, we see that $S_{\msf{K}',V} \cong V$ and $S_{\msf{K}'}\to \scr{S}_{\msf{K}'}$ is an isomorphism. 
\end{proof}

In particular, \Cref{prop:equalcan} implies that the functoriality of canonical integral models automatically satisfies \cite[Axiom 3.1]{HR17} under hyperspecial levels. In fact, this can be checked for every $\msf{K}'_p \subset \msf{K}_p$ directly by following the proof of \cite[Proposition 7.1.9]{Pap23}. 

\begin{prop} \label{prop:transpropersurj}
    Let $\{\scr{S}_{\msf{K}}\}_{\msf{K}^p}$ (resp.\ $\{\scr{S}_{\msf{K'}}\}_{\msf{K}^p}$) be a system of canonical integral models of $\{\Sh_{\msf{K}}(\mbf{G},\mbf{X})_E\}_{\msf{K}^p}$ (resp.\ $\{\Sh_{\msf{K'}}(\mbf{G},\mbf{X})_E\}_{\msf{K}^p}$). The morphism $\scr{S}_{\msf{K}'} \to \scr{S}_{\msf{K}}$ constructed in \cite[Corollary 4.3.2]{PR24} is proper and surjective. 
\end{prop}
\begin{proof}
    Let $\scr{S}_{\msf{K}'} \xrightarrow{j} Y \xrightarrow{\pi} \scr{S}_{\msf{K}}$ be a Nagata compactification of $\scr{S}_{\msf{K}'} \to \scr{S}_{\msf{K}}$ (see \cite[Theorem 4.1]{Con07}). By replacing $Y$ with the Zariski closure of $j(\scr{S}_{\msf{K}'})$, we may assume that $j$ has a dense image. Since $\Sh_{\msf{K}'}(\mbf{G},\mbf{X}) \to \Sh_{\msf{K}}(\mbf{G},\mbf{X})$ is finite \'{e}tale, we have $Y_E \cong \Sh_{\msf{K}'}(\mbf{G},\mbf{X})_E$ and $Y-j(\scr{S}_{\msf{K}'})$ lies in the special fiber of $Y$. Suppose that we have a closed point $y_0 \in Y(\ov{k})$ outside $j(\scr{S}_{\msf{K}'})$. Since $Y$ is flat over $O_E$, $y_0$ can be lifted to a point $y\in Y(R)$ with $R$ a complete discrete valuation ring of mixed characteristic over $O_E$. Let $x \in \scr{S}_{K}(R)$ be the image of $y$. Since $R$ is strictly henselian, $x$ can be lifted to some $x^p \in \lim_{\msf{K}^p} \scr{S}_{\msf{K}}(R)$. Then, $x^p_{R[1/p]} \times y_{R[1/p]}$ defines an element of $\lim_{\msf{K}^p} \Sh_{\msf{K}'}(\mbf{G},\mbf{X})(R[1/p])$, so we see that $y_{R[1/p]}$ extends to $\scr{S}_{\msf{K}'}(R)$. Since $Y$ is separated, $y$ lies in $j(\scr{S}_{\msf{K}'})$, but this is a contradiction. Thus, $j$ is an isomorphism and $\scr{S}_{\msf{K}'}\to \scr{S}_{\msf{K}}$ is proper. Since $\scr{S}_{\msf{K}}$ is flat over $O_E$ and $\Sh_{\msf{K}'}(\mbf{G},\mbf{X}) \to \Sh_{\msf{K}}(\mbf{G},\mbf{X})$ is surjective, $\scr{S}_{\msf{K}'}\to \scr{S}_{\msf{K}}$ is a proper surjection. 
\end{proof}

Note that \Cref{prop:transpropersurj} does not rely on the axiom (SV5). However, the proof of \Cref{thm:canint} relies on (SV5) in that $\Sh_{\msf{K}'}(\mbf{G},\mbf{X}) \to \Sh_{\msf{K}}(\mbf{G},\mbf{X})$ is a finite \'{e}tale covering locally isomorphic to the finite set $\msf{K}_p/\msf{K}'_p$. Without (SV5), there is a contribution of the center of $\mbf{G}$ and $\Sh_{\msf{K}'}(\mbf{G},\mbf{X})_E^\ad \to \Sh_{\msf{K}}(\mbf{G},\mbf{X})_E^\ad$ cannot be directly interpreted as a moduli space of level structures. 


\renewcommand\bibfont{\footnotesize}
\printbibliography

\end{document}